\documentclass[11pt,a4paper,reqno]{amsart}

\usepackage{amsmath,amsfonts,amssymb,amsthm,color, comment, tabularx,mathabx}  
\usepackage{dsfont} 
\usepackage{esint} 
\usepackage{cancel} 
\usepackage[utf8]{inputenc}
\usepackage{graphicx}
\usepackage{hyperref}
\hypersetup{
    colorlinks=true,
    linkcolor=blue,
    citecolor=blue,
    filecolor=magenta,      
    urlcolor=cyan,
}

\usepackage{ mathrsfs }

\usepackage[mathscr]{euscript}

\usepackage[
backend=biber,
style=alphabetic,
sorting=ynt
]{biblatex}
\newcommand{\Addresses}{{
\noindent  \textsc{Martin Hsu}, Purdue University, \texttt{hsu263@purdue.edu}\\
\textsc{Fred Lin,} University of Bonn, \texttt{fredlin@math.uni-bonn.de}\\
\textsc{Amelia Stokolosa,
}University of Wisconsin-Madison, \texttt{stokolosa@wisc.edu}

}}

\theoremstyle{plain}
\newtheorem{theorem}{Theorem}[section]

\newtheorem{lemma}[theorem]{Lemma}

\theoremstyle{definition}

\newtheorem{conjecture}[theorem]{Conjecture}
\newtheorem{example}[theorem]{Example}

\theoremstyle{remark}
\newtheorem{remark}[theorem]{Remark}

\newtheorem*{theorem*}{Theorem}

\newcommand{\MH}[1]{
}

\definecolor{olive}{rgb}{0.42, 0.56, 0.14}
\newcommand{\AS}[1]{{\color{olive}{[#1]}}}

\newcommand{\br}[1]{\left( #1 \right)}
\newcommand{\Br}[1]{\left[ #1 \right]}
\newcommand{\BR}[1]{\left\{ #1 \right\}}
\newcommand{\ang}[1]{\left< #1 \right>}
\newcommand{\verts}[1]{\left\vert #1 \right\vert}
\newcommand{\Verts}[1]{\left\Vert #1 \right\Vert}

\newcommand{\1}{\mathds{1}}
\newcommand{\Tr}{\mathrm{Tr}}
\newcommand{\Mod}{\mathrm{Mod}}

\newtheorem*{claim}{Claim}

\numberwithin{equation}{section}

\def\R{\mathbb{R}}

\definecolor{afb}{rgb}{0.36, 0.54, 0.66}

\DeclareMathOperator{\sgn}{sgn}
\DeclareMathOperator{\supp}{supp}

\DeclareMathOperator{\loc}{loc}
\DeclareMathOperator{\Dil}{Dil}

\newcommand{\mA}{\mathcal{A}}

\newcommand{\mE}{\mathcal{E}}

\newcommand{\mI}{\mathcal{I}}

\newcommand{\mK}{\mathcal{K}}

\newcommand{\M}{\mathcal{M}}

\newcommand{\mQ}{\mathcal{Q}}
\newcommand{\mS}{\mathcal{S}}

\newcommand{\mZ}{\mathcal{Z}}

\newcommand{\Max}[2]{\M^{(#1)}_{#2}}

\newcommand{\N}{\mathbb{N}}
\newcommand{\Z}{\mathbb{Z}}

\newcommand{\tbf}{\textbf{t}}

\newcommand{\ap}{\alpha}
\newcommand{\vp}{\varphi}

\newcommand{\lb}{\lambda}
\newcommand{\Lb}{\Lambda}
\newcommand{\g}{\gamma}

\newcommand{\Om}{\Omega}

\newcommand{\ep}{\epsilon}
\newcommand{\sg}{\sigma}
\newcommand{\kp}{\kappa}

\newcommand{\D}[2]{\Delta^{(#1)}_{#2}}

\newcommand{\wh}[1]{\widehat{#1}}
\newcommand{\whf}{\widehat{f}}

\newcommand{\wc}[1]{\widecheck{#1}}

\newcommand{\wt}[1]{\widetilde{#1}}
\newcommand{\wte}{\widetilde{\eta}}

\newcommand{\xb}{\overline{x}}
\newcommand{\yb}{\overline{y}}

\newcommand{\zb}{\overline{z}}

\newcommand{\p}{\partial}

\newcommand{\lp}{\left(}
\newcommand{\rp}{\right)}

\newcommand{\lf}{\left|}
\newcommand{\rf}{\right|}

\newcommand{\la}{\langle}
\newcommand{\ra}{\rangle}

\newcommand{\norm}[2]{\lf \lf #1\rf \rf_{#2}}

\newcommand{\Hnorm}[3]{\lf \lf #1\rf \rf_{H^{(#2, #3)}}}

\newcommand{\Co}[1]{C^{\infty}_0(#1)}

\definecolor{darkgreen}{rgb}{0.0, 0.26, 0.15}
\definecolor{orange}{rgb}{1.0, 0.49, 0.0}

\makeatletter
\def\@tocline#1#2#3#4#5#6#7{\relax
  \ifnum #1>\c@tocdepth 
  \else
    \par \addpenalty\@secpenalty\addvspace{#2}%
    \begingroup \hyphenpenalty\@M
    \@ifempty{#4}{%
      \@tempdima\csname r@tocindent\number#1\endcsname\relax
    }{%
      \@tempdima#4\relax
    }%
    \parindent\z@ \leftskip#3\relax \advance\leftskip\@tempdima\relax
    \rightskip\@pnumwidth plus4em \parfillskip-\@pnumwidth
    #5\leavevmode\hskip-\@tempdima
      \ifcase #1
       \or\or \hskip 1em \or \hskip 2em \else \hskip 3em \fi%
      #6\nobreak\relax
    \hfill\hbox to\@pnumwidth{\@tocpagenum{#7}}\par
    \nobreak
    \endgroup
  \fi}
\makeatother

\title[A Variant of Triangular Hilbert Transform]{A study guide for \textit{Trilinear smoothing inequalities and a variant of the triangular Hilbert transform}}
\author{Martin Hsu, Fred Lin, Amelia Stokolosa}
\date{}

\begin{document}

\begin{abstract}
    This article is a study guide for ``Trilinear smoothing inequalities and a variant of the triangular Hilbert transform'' by Christ, Durcik, and Roos \cite{CHRIST2021107863}. We first present the standard techniques in the study of oscillatory integrals with the simpler toy model of a Hilbert transform along a parabola. These standard techniques prove to be insufficient in the study of the triangular Hilbert transform with curvature. The central and novel idea in their proof of the $L^p$-boundedness of the triangular Hilbert transform with curvature is a trilinear smoothing inequality which we also examine in this article.



\end{abstract}

\maketitle

\normalsize
\tableofcontents

\section{Notation}

\begin{center}

\begin{tabularx}{0.8\textwidth} { 
  | >{\raggedright\arraybackslash}X 
  | >{\centering\arraybackslash}X 
  | >{\raggedleft\arraybackslash}X | }
\hline
    $\1_E$ & the characteristic function of $E$\\
\hline
    \(a\lesssim b \) & \(\exists C>0\) such that \(a \leq C\cdot b\)\\
\hline
$a \sim b$ & $a \lesssim b \textrm{ and } a \gtrsim b$\\
\hline
\(a\ll b\) & \((b+a) \sim b \sim (b-a)\)\\
\hline
$k_1 \vee k_2$ & $\max\BR{k_1, k_2}$\\
\hline
\(k_1\wedge k_2\) & \(\min\BR{k_1,k_2}\)\\
\hline
$\mathcal{M} f $ & the Hardy-Littlewood maximal operator\\
\hline
$\mathcal{M}_\sigma f$ & the shifted maximal operator\\
\hline
$\mathscr{M}\br{f,g}$ & the bilinear maximal operator\\
\hline
$\la x \ra$ & $(1+|x|^2)^{1/2}$\\
\hline
$K \Subset X$ & $K$ is compactly supported in $X$\\
\hline
$\mathcal{S}$ & the space of Schwartz functions\\
\hline
$\mathcal{D}$ & the multiplicative derivative \\
\hline
\end{tabularx}

\end{center}

\section{Introduction}
\subsection{Brascamp-Lieb inequalities} \ 

Recall the following standard inequalities:
\begin{itemize}
           \item (H\"older's inequality) For $1\leq p_{1},p_{2}\leq \infty$, $\frac{1}{p_{1}}+\frac{1}{p_{2}}=1$,
           \[
           \left|\int_{\mathbb{R}}f_{1}(x)f_{2}(x)dx\right|\lesssim \|f_{1}\|_{L^{p_{1}}}\|f_{2}\|_{L^{p_{2}}}.
           \]
           
            \item (Young's convolution inequality) For $1\leq p_{1},p_{2},p_{3}\leq \infty$, $\frac{1}{p_{1}}+\frac{1}{p_{2}}+\frac{1}{p_{3}}=2$,
           \[
           \left|\int_{\mathbb{R}^{2}}f_{1}(x_{1})f_{2}(x_{1}-x_{2})f_{3}(x_{2})dx_{1}dx_{2}\right|\lesssim \prod_{j=1}^{3}\|f_{j}\|_{L^{p_{j}}}.
           \]
           
           \item (Loomis–Whitney inequality)
           \[
           \left|\int_{\mathbb{R}^{3}}f_{1}(x_{2},x_{3})f_{2}(x_{1},x_{3})f_{3}(x_{1},x_{2})dx_{1}dx_{2}dx_{3}\right|\lesssim \prod_{j=1}^{3}\|f_{j}\|_{L^{2}}.
           \]
\end{itemize}
The above inequalities belong to a larger class of estimates known as the ``Brascamp-Lieb inequalities''. 

\begin{theorem}[2008, J.Bennett, A.Carbery, M.Christ, T.Tao \cite{bennett2008brascamp}]
Let $\{B_{j}\}_{j=1}^{m}$, $B_{j}:\mathbb{R}^{n}\rightarrow \mathbb{R}^{n_{j}}$, define linear projections. Then
\begin{equation}
\left|\int_{\mathbb{R}^{n}}\prod_{j=1}^{m}f_{j}(B_{j}x)dx\right|\lesssim_{B_{j},p_{j}} \prod_{j=1}^{m}\|f_{j}\|_{L^{p_{j}}(\mathbb{R}^{j})}
\end{equation}
if and only if $\{B_{j}\}_{j=1}^{m}$ and $\{p_{j}\}_{j=1}^{m}$ satisfy the following two conditions:\\
\begin{align}
    &\quad n=\sum_{j=1}^{m}\frac{n_{j}}{p_{j}}\label{cond1BL},\\
    &\quad \operatorname{dim}V\leq \sum_{j=1}^{m}\frac{\operatorname{dim}(B_{j}V)}{p_{j}}\; \text{for all subspaces}\; V\subseteq \mathbb{R}^{n} \label{cond2BL}.
\end{align}
   \end{theorem}

\begin{remark}
    The necessity of conditions \eqref{cond1BL} and \eqref{cond2BL} follows from a scaling argument on the whole space and on subspaces, respectively. 
\end{remark}

\subsection{Singular Brascamp-Lieb Inequalities} \ 

There are several variants of the Brascamp-Lieb inequality. In the following, we introduce a singular variant.
Adding a singular kernel on the left-hand side, many important multilinear singular integrals can be written in the following form:
\begin{equation}
\left|\int_{\mathbb{R}^{n}}\prod_{j=1}^{m}f_{j}(B_{j}x)\textcolor{red}{K(B_{0}x)}dx\right|\lesssim_{B_{j},p_{j}} \prod_{j=1}^{m}\|f_{j}\|_{L^{p_{j}}(\mathbb{R}^{j})},
\end{equation}
where $\wh{K}$ satisfies the Mikhlin multiplier condition; that is 
\begin{equation}\label{mikhlincond}
|\partial^{\alpha}\widehat{K}(\xi)|\lesssim |\xi|^{-|\alpha|}
\end{equation}
     for all $\xi \neq 0$ and all multi-indices $\alpha$ up to a suitably large order. Unlike with the Brascamp-Lieb inequality, there is no characterization of the $L^p$-boundedness of the singular Brascamp-Lieb inequality. One thus needs to study singular Brascamp-Lieb inequalities in a  case by case basis. 

\begin{example}[Hilbert Transform]
    Define the Hilbert transform as follow
 \begin{equation}
      (Hf)(x):=\operatorname{p.v.}\int_{\mathbb{R}}f(x-y)\frac{1}{y}dy.
 \end{equation}
 In 1928, M.Riesz showed that the Hilbert transform is bounded on $L^{p}$ spaces for $1<p<\infty$.
 In the dual form, we have
 \begin{equation}
      \left| \operatorname{p.v.}\int_{\mathbb{R}^{2}}f_{1}(x_{1})f_{2}(x_{2})\frac{1}{x_{1}-x_{2}}dx_{1}dx_{2} \right| \lesssim \prod_{j=1}^{2}\|f_{j}\|_{L^{p_{j}}},
 \end{equation}
 for $1<p_{1},p_{2}<\infty$ and $\frac{1}{p_{1}}+\frac{1}{p_{2}}=1$.
\end{example}
We may write the Hilbert transform in the Fourier multiplier form, 
\begin{equation}
    (Hf)(x)=c\int_{\mathbb{R}}\operatorname{sgn}(\xi)\widehat{f}(\xi)e^{2\pi ix\xi}d\xi, \quad c\neq 0.
\end{equation}
In general, we may define the Fourier multiplier operator $T_{m}$ of the associated function $m$ as follows
\begin{equation}
    (T_{m}f)(x):=\int_{\mathbb{R}}m(\xi)\widehat{f}(\xi)e^{2\pi ix\xi}d\xi.
\end{equation}
In 1956, Mikhlin proved that if $m$ satisfied the Mikhlin condition \eqref{mikhlincond}, then $T_{m}$ is bounded on $L^{p}$ where $1<p<\infty$.
\begin{example}[Coifman-Meyer Operator]
For a function $m$, we can define the associated bilinear Fourier multiplier operator 
\begin{equation}
      T_{m}(f_{1},f_{2})(x):=\int_{\mathbb{R}^{2}}m(\xi_{1},\xi_{2})\widehat{f}_{1}(\xi_{1})\widehat{f}_{2}(\xi_{2})e^{2\pi ix(\xi_{1}+\xi_{2})}d\xi_{1}d\xi_{2}.
\end{equation}
In 1978, R.Coifman and Y.Meyer showed that if $m$ satisfied the Mikhlin condition \eqref{mikhlincond}, then $T_{m}$ is bounded from $L^{p_{1}}\times L^{p_{2}}$ to $L^{p_{3}}$ where $\frac{1}{p_{1}}+\frac{1}{p_{2}}=\frac{1}{p_{3}}$, $1<p_{1},p_{2},p_{3}<\infty$.  In the physical side and in the dual form, we can rewrite the Coifman-Meyer operator as
\begin{equation}
    \int_{\mathbb{R}^{3}}f_{1}(y_{1})f_{2}(y_{2})f_{3}(x)K(x-y_{1},x-y_{2})dxdy_{1}dy_{2}
\end{equation}
which is again written in a singular Brascamp-Lieb form.
\end{example}

\begin{example}[Bilinear Hilbert Transform]
Notice that the singularity set of the Coifman-Meyer operator is only at the origin. We may also study the multilinear Fourier multiplier which has a nontrivial singularity set. Define the bilinear Hilbert transform as follows
\begin{equation}
    BH_{\beta}(f_{1},f_{2})(x):=\operatorname{p.v.}\int_{\mathbb{R}^{2}}f_{1}(x+t)f_{2}(x+\beta t)\frac{1}{t}dt.
\end{equation}
Written in the multiplier form, we have
\begin{equation}
     BH_{\beta}(f_{1},f_{2})(x)=\int_{\mathbb{R}^{2}}m(\xi_{1}-\beta \xi_{2})\widehat{f}_{1}(\xi_{1})\widehat{f}_{2}(\xi_{2})e^{2\pi ix(\xi_{1}+\xi_{2})}d\xi_{1}d\xi_{2}
\end{equation}
whose singularity set is a line instead of a point. In 1997, C.Thiele and M.Lacey\cite{MLCT1997BHT} showed that
    for $\frac{1}{p_{3}^{'}}=\frac{1}{p_{1}}+\frac{1}{p_{2}}$ with $1<p_{3}'<2$, $2<p_{1},p_{2}<\infty$ and $f_{1},f_{2}\in \mathcal{S}(\mathbb{R})$, we have
\begin{equation}
    \|BH_{\beta}(f_{1},f_{2})\|_{L^{p_{3}^{'}}}\lesssim_{p,\beta}\|f_{1}\|_{L^{p_{1}}}\cdot \|f_{2}\|_{L^{p_{2}}}
\end{equation}
\end{example}

\begin{example}[Two Dimensional Bilinear Hilbert Transform, Twisted Paraproduct]
    In 2010, C.Demeter and C.Thiele\cite{demeter2010two} studied the following two dimensional BHT
\begin{equation}
BH_{2D}(f_{1},f_{2})(x,y):=\int_{\mathbb{R}^{2}}f_{1}((x,y)+(s,t))f_{2}((x,y)+A(s,t))K(s,t)dsdt.
\end{equation}
They classified all the cases and proved some bounds for each case except for one degenerate case below, which they called the twisted paraproduct:
\begin{equation}
T(f_{1},f_{2})(x,y):=\int_{\mathbb{R}^{2}}f_{1}(x+s,y)f_{2}(x,y+t)K(s,t)dsdt.
\end{equation}
 In 2012, V.Kovac proved a bound for the twisted paraproduct\cite{kovavc2012boundedness}. Later, P.Durcik utilized the twisted technology and proved bounds for the following multilinear singular integral with a more entangled structure ( see \cite{durcik20144} and \cite{durcik2017}):
 \begin{equation}
     \int_{\mathbb{R}^{4}}f_{1}(x,y)f_{2}(x',y)f_{3}(x',y')f_{4}(x,y')K'(x'-x,y'-y)dxdx'dydy'.
 \end{equation}
\end{example}

Usually, kernels in lower dimensions are difficult to estimate. The kernel and functions appearing in the two-dimensional bilinear Hilbert transform have inputs in two dimensions. We can study a bilinear operator with a two-dimensional input function and a one-dimensional kernel. Define the triangular Hilbert transform as follows:
\begin{equation}
    T(f_{1},f_{2})(x):=\operatorname{p.v.}\int_{\mathbb{R}^{2}}f_{1}(x+t,y)f_{2}(x,y+t)\frac{dt}{t}. 
\end{equation}
Written in the dual form and after a change of variables, we obtain the following trilinear form:
\begin{equation}
      \Lambda (f_{1},f_{2},f_{3})=\operatorname{p.v.}\int_{\mathbb{R}^{3}}f_{0}(x_{1},x_{2})f_{1}(x_{2},x_{0})f_{2}(x_{0},x_{1})\frac{1}{x_{0}+x_{1}+x_{2}}dx_{0}dx_{1}dx_{2}  .
\end{equation}
The following is one of the main conjectures in the study of multilinear singular integrals.
\begin{conjecture}
    The trilinear form $\Lambda$ satisfies the following $L^{3}$ estimates:
\begin{equation}
    |\Lambda (f_{1},f_{2},f_{3})|\lesssim \prod_{i=0}^{2}\|f_{i}\|_{L^{3}}.
\end{equation}
\end{conjecture}
We can then ask for which $p_{i}$ does the above estimate hold. But since we can permute $f_{0},f_{1},f_{2}$, if we have a bound for the triangular Hilbert transform at any single exponent, then by permutation and interpolation, we have the bound in the conjecture above. The reason why the triangular Hilbert transform is important is that it implies both the uniform bound of the bilinear Hilbert transform and the bound of the Carleson operator.

\begin{remark}
    The reader is encouraged to refer to the survey \cite{durcik2021singular} for more on singular Brascamp-Lieb inequalities.
\end{remark}

\subsection{Introducing curvature} \ 

 Now we turn to the case when the map inside those functions is not just linear. One of the simplest nonlinear maps is the monomial $t\mapsto t^{\alpha}$ for nontrivial $\alpha$. We have the following $L^{2}$ bound of the bilinear Hilbert transform along monomial

\begin{theorem}[2013, X.Li \cite{XL2013}]
    Define the bilinear Hilbert transform along monomial $\Gamma =(t,t^{\alpha})$ for $\alpha \geq 2$ as follows:
\begin{equation}
    BH_{\Gamma}(f_{1},f_{2})(x):=\operatorname{p.v.}\int_{\mathbb{R}^{2}}f_{1}(x+t)f_{2}(x+ t^{\alpha})\frac{1}{t}dt.
\end{equation}
Then for $f_{1},f_{2}\in \mathcal{S}(\mathbb{R})$, we have
\begin{equation}
    \|BH_{\Gamma}(f_{1},f_{2})\|_{L^{1}}\lesssim \|f_{1}\|_{L^{2}}\|f_{2}\|_{L^{2}}.
\end{equation}
\end{theorem}

Although bounds for the triangular Hilbert transform remain an open problem, M.Christ, P.Durcik and J.Roos proved $L^p$ bounds for the triangular Hilbert transform along a parabola:

\begin{equation}\label{eq triangular HT}
    T(f_{1},f_{2})(x,y):=\operatorname{p.v.}\int_{\mathbb{R}}f_{1}(x+t,y)f_{2}(x,y+t^{2})\frac{dt}{t}.
\end{equation}

\vspace{.5in}

\subsection{Main results} \

\begin{theorem}[Theorem 1 in \cite{CHRIST2021107863}]\label{thm 1 - THT Lp}
For $f_{1},f_{2}\in \mathcal{S}(\mathbb{R}^{2})$, $T$ as defined in \eqref{eq triangular HT} is bounded as a map from $L^{p}\times L^{q}$ to $L^{r}$, where $ \frac{1}{p}+\frac{1}{q}=\frac{1}{r}$, $p,q\in (1,\infty )$, and $r\in [1,2)$.
\end{theorem}

\begin{theorem}[Theorem 5 in \cite{CHRIST2021107863}]\label{thm 5 - smoothing}
    Let $\zeta$ be a smooth function with compact support in $\mathbb{R}^{2}\times (\mathbb{R}\setminus \{0\})$. Consider the following trilinear form
\begin{equation}
    \Lambda(f_{1},f_{2},f_{3}):=\int_{\mathbb{R}^{3}}f_{1}(x+t,y)f_{2}(x,y+t^{2})f_{3}(x,y)\zeta (x,y,t)dxdydt.
\end{equation}
Then there exist $C=C(\zeta)>0$ and $\sigma >0$ such that for all $f_{1},f_{2},f_{3}\in \mathcal{S}(\mathbb{R})$, we have
\begin{equation}\label{smoothing1}
    | \Lambda(f_{1},f_{2},f_{3})|\leq C \|f_{1}\|_{H^{(-\sigma ,0)}}\|f_{2}\|_{H^{(0 ,-\sigma)}}\|f_{3}\|_{L^{\infty}},
\end{equation}
where
\begin{equation*}
    \|f\|^{2}_{H^{(a ,b)}}:=\int_{\mathbb{R}^{2}}|\widehat{f}(\xi_{1},\xi_{2})|^{2}(1+|\xi_{1}|^{2})^{\frac{a}{2}}(1+|\xi_{2}|^{2})^{\frac{b}{2}}d\xi_{1}d\xi_{2}.
\end{equation*}
\end{theorem}

\begin{remark}
The inequality \eqref{smoothing1} is equivalent to 
\begin{equation}\label{smoothing2}
     | \Lambda(f_{1},f_{2},f_{3})|\lesssim \lambda^{-\sigma}\|f_{1}\|_{L^{2}}\|f_{2}\|_{L^{2}}\|f_{3}\|_{L^{\infty}},
\end{equation}
for $\lambda >1$ under the assumption that $\widehat{f}_{j}$ is supported in $|\xi_{j}| \sim \lambda$ for at least one index $j=1,2$.
\end{remark}

\begin{remark}
There can be no smoothing inequality of the form
\begin{equation*}
    | \Lambda(f_{1},f_{2},f_{3})|\lesssim \|f_{1}\|_{H^{(0 ,-\sigma)}}\|f_{2}\|_{H^{(-\sigma ,0)}}\|f_{3}\|_{L^{\infty}}.
\end{equation*}
Suppose, for a contradiction, that such an inequality holds. Observe the following modulation symmetry of our operator:
\begin{equation}\label{smoothingimpossible}
    \Lambda(f_{1},f_{2},f_{3})=\Lambda(M_{a}f_{1},M_{b}f_{2},M_{-a-b}f_{3}),
\end{equation}
where $M_{a}f(x):=e^{2\pi i ax}f(x)$.

Using the equivalent form \eqref{smoothing2}, with $\widehat{f}_{j}$ supported where $|\xi_{j(\operatorname{mod}2)+1}| \sim \lambda$ for at least one index $j=1,2$. Then by taking $a$ or $b$ large enough such that the frequency support of $M_{a}f_{1}$ or $M_{b}f_{2}$ is away from $|\xi_{j(\operatorname{mod}2)+1}| \sim \lambda$, the the right-hand side of \eqref{smoothingimpossible} vanishes.
\begin{align*}
    |\Lambda(f_{1},f_{2},f_{3})|=|\Lambda(M_{a}f_{1},M_{b}f_{2},M_{-a-b}f_{3})|\lesssim \lambda^{-\sigma}\|M_{a}f_{1}\|_{L^{2}}\|M_{b}f_{2}\|_{L^{2}}\|M_{-a-b}f_{3}\|_{L^{\infty}}=0.
\end{align*}
We obtain a contradiction. We thus cannot have \eqref{smoothingimpossible}.
\end{remark}

\vspace{.3in}

\begin{theorem}[Theorem 2 in \cite{CHRIST2021107863}]\label{thm_ani_twist_para}
      For $f_{1},f_{2}\in \mathcal{S}(\mathbb{R}^{2})$, we define
\begin{equation}
    \begin{aligned}
        T_{m}(f_{1},f_{2})(x,y)
        := & 
        \int_{\mathbb{R}^{2}}f_{1}(x+s,y)f_{2}(x,y+t)k(s,t)dsdt\\
        = &
        \int_{\R^2}
            m\br{\xi,\eta}
            \widehat{f_1}^{\br{1}}\br{\xi,y}
            \widehat{f_2}^{\br{2}}\br{x,\eta}
            e^{2\pi i\br{x\xi+y\eta}}
        d\xi d\eta,
    \end{aligned}
\end{equation}
where $m=\widehat{k}$ satisfies the following symbol estimate
\begin{equation}
|\partial^{\alpha}_{\xi}\partial^{\beta}_{\eta}m(\xi,\eta)|\lesssim_{\alpha, \beta}(|\xi|^{\frac{1}{k}}+|\eta|^{\frac{1}{l}})^{- k\alpha -l\beta},
\end{equation}
where $\alpha,\beta\geq 0$ up to a large finite number. Then $T$ maps from $L^{p}\times L^{q}$ to $L^{r}$ for $\frac{1}{p}+\frac{1}{q}=\frac{1}{r}$, $p,q\in (1,\infty ),\:r\in (\frac{1}{2},2)$.
\end{theorem}

\vspace{.3in}

\subsection{Acknowledgments} \ 
We would like to thank the organizers of the ``study guide writing workshop 2023'' at University of Pennsylvania, and in particular our mentor Phil Gressman, for suggesting this paper and for organizing a memorable workshop. We would also like to thank the other participants of this workshop for the many inspiring and thought-provoking conversations we shared in Philadelphia. Also, the third author was supported in part by NSF DMS- 2037851.

\vspace{.5in}

\section{Toy model: Hilbert transform along a parabola}

\subsection{Littlewood-Paley decomposition} \

Let $\vp \in C^{\infty}_c(\R)$ be an even function s.t. $\vp(\zeta) \equiv 1$ for $|\zeta| \leq 1$, $\supp \vp \subseteq \{\zeta\in \R; |\zeta| \leq 2\}$, and $0 \leq \vp \leq 1$. Define $\psi(\zeta) = \vp(\zeta) - \vp(2\zeta)$; or equivalently, \(\psi=\vp-\Dil^{\infty}_{1/2}\vp\). As such, $\supp \psi \subseteq \{\zeta \in \R; |\zeta| \sim 1\}$. By construction, we thus have the following partition of unity:
\begin{equation*}
    \1_{\R\setminus\BR{0}}= \sum_{k\in\Z}\psi_k.
\end{equation*}
Consider another test function $\wt{\psi}$ supported on an annulus slightly larger than $\supp \psi$ so that $\psi \wt{\psi} = \psi$. For convenience, we also write:
\begin{equation*}
    \vp_k\br{\zeta}:=\vp\br{\frac{\xi}{2^k}}.
\end{equation*}
Similarly, we denote $\psi_k$ and $\wt{\psi}_k$. We can thus denote two types of Littlewood Paley projections. On one hand, for single-variable functions $f: \R \rightarrow \R$, we define
\begin{align*}
    \wh{\Delta_{k}f}(\xi) := \psi_k (\xi) \wh{f}(\xi).
\end{align*}
On the other hand, for functions with two variables $f: \R^2 \rightarrow \R$, we define
\begin{align*}
    \wh{\Delta^{\br{l}}_{k}f}(\xi_1, \xi_2) := \psi_k (\xi_l) \wh{f}(\xi_1, \xi_2)
\end{align*}

\subsection{Symmetry of the operator} \

Before introducing the triangular Hilbert transform along a parabola, let's first illustrate the main philosophy of the argument through a simpler object: the Hilbert transform along a parabola, denoted by $H_P$ and defined as follows:
\begin{equation}
    H_P f\br{x,y}:=\mathrm{p.v.}\int_{\R} f\br{x+t,y+t^2}\frac{dt}{t}.
\end{equation}
One can notice some similarities between the Hilbert transform along a parabola $H_P$ and the triangular Hilbert transform along a parabola $T$. Indeed consider the two operators given side by side in the kernel form and in the multiplier form: 
\begin{equation*}
    \left\{
    \begin{aligned}
        H_P f\br{x,y}
        = &
        \mathrm{p.v.}\int_{\R}f\br{x+t,y+t^2}\frac{dt}{t}\\
        = &
        \int_{\R^2}
             \mathrm{p.v.}\int_{\R}
                e^{2\pi i\br{\xi t+\eta t^2}}
            \frac{dt}{t}
            \widehat{f}\br{\xi,\eta}e^{2\pi i\br{\xi x+\eta y}}
        d \xi d\eta\\
        T\br{f_1,f_2}\br{x,y}=&
        \mathrm{p.v.}\int_{\R}f_1\br{x+t,y}f_2\br{x,y+t^2}\frac{dt}{t}\\
        =&
        \int_{\R^2}
             \mathrm{p.v.}\int_{\R}
                e^{2\pi i\br{\xi t+\eta t^2}}
            \frac{dt}{t}
            \widehat{f_1}^{\br{1}}\br{\xi,y}\widehat{f_2}^{\br{2}}\br{x,\eta}e^{2\pi i\br{\xi x+\eta y}}
        d \xi d\eta.
    \end{aligned}
    \right.
\end{equation*}

After dualizing $H_P$, we observe more similarities with $T$:
\begin{equation}
    \begin{aligned}
        \ang{H_P f_1,f_2}
        = &
        \int_{\R} 
            \int_{\R^2} 
                f_1\br{x+t, y}f_2\br{x,y-t^2} 
            dxdy
        \frac{dt}{t}\\
        \approx &
        \int_{\R^2}
            T\br{f_1,f_2}\br{x,y}
        dxdy
        =\ang{T\br{f_1,f_2},1}.
    \end{aligned}
\end{equation}
However, the analysis of $T$ is in fact much harder than that of $H_P$. For example, when evaluating the $L^1$ norm of the triangular Hilbert transform along a parabola, we observe that 
\begin{equation*}
    \Verts{T\br{f_1,f_2}}_{L^1}
    =\sup_{\Verts{g}_{L^\infty}=1} \verts{\ang{T\br{f_1,f_2},g}} \geq \verts{\ang{T\br{f_1,f_2},1}}=\verts{\ang{H_P f_1,f_2}}.
\end{equation*}
Despite the differences between \(H_P\) and \(T\), the reduction principle demonstrated below is standard. It applies to both \(H_P\), \(T\) and more generally to operators with curved phase structures in general. As such, we chose to include it in our study guide as a toy model for the more involved proof of the $L^p$ boundedness of the triangular Hilbert transform along a parabola (see section 4 for the proof of the latter). 

• \textbf{Characteristic symmetries} :

To highlight the characteristic symmetries satisfied by the Hilbert transform along a parabola, first define the following operators:
\begin{align*}
    \mathrm{Tr}_{x_0,y_0} f(x,y) &= f(x-x_0,y-y_0),\\
    \mathrm{Mod}_{\xi_0,\eta_0} f(x,y) &= e^{2\pi i\br{\xi_0 x+\eta_0 y}}f(x,y),\\
    \mathrm{Dil}^{p}_{a,b} f(x,y) &= \br{ab}^{-1/p}f\br{\frac{x}{a},\frac{y}{b}}.
\end{align*}

Central to the proof below, are the following symmetries of the Hilbert transform along a parabola:

        \noindent(1) \textbf{Translation symmetry:}
        \begin{equation}
            \Tr_{x_0,y_0}H_P=H_P\Tr_{x_0,y_0}.
        \end{equation}
        (2) \textbf{Anisotropic dilation symmetry:}
        \begin{equation}
            \Dil^p_{s,s^2}H_P=H_P\Dil^p_{s,s^2}.
        \end{equation}
    
The \textit{anisotropic} dilation symmetry is an indicator of the presence of \textit{curvature}. One notable fact is the absence of modulation symmetry. This indicates that the operator reacts differently to high-frequency inputs versus low-frequency inputs.

\subsection{Decompositions} \ 

• \textbf{Treatment of the singularity:} We perform a decomposition on the singular kernel that respects the dilation symmetry structure
\begin{equation*}
        \frac{1}{t}=\sum_{j\in\Z} \frac{\psi\br{2^j t}}{t}
\end{equation*}
    and thus, we obtain the following natural decomposition: 
    \begin{equation*}
        H_P:=\sum_{j\in\Z}H_j,
    \end{equation*}
where
    \begin{equation*}
        H_jf\br{x,y}:=
        \int_{\R}
            f\br{
                x+t,
                y+t^2
            }
            \frac{\psi\br{2^jt}}{t}
        dt.
    \end{equation*}
Note that the dilation symmetry is encoded within the following relation: 
    \begin{equation*}
        H_j =\Dil^p_{2^{-j},2^{-2j}} H_0 \Dil^p_{2^j,2^{2j}}.
    \end{equation*}
Alternatively, we may phrase it in terms of multipliers:
    \begin{equation}\label{eq_multiplier_1st_appearence}
        \widehat{H_j f}\br{\xi,\eta}
        =m
        \br{\frac{\xi}{2^j},\frac{\eta}{2^{2j}}}
        \widehat{f}\br{\xi,\eta}\text{ with }
        m\br{\xi,\eta}:=
        \int_{\R}
            e^{2\pi i\br{\xi t+\eta t^2}}
            \frac{\psi\br{t}}{t}
        dt,
    \end{equation}
    This also suggests that, more often than not, the analysis of \(H_j\) can be reduced to the analysis on the unit scale operator \(H_0\). Namely, \(H_0\) is an anchor point if we perform our analysis respecting the dilation symmetry.

    • \textbf{Multiplier behaviors:} For a brief moment, we focus on the unit scale part \(H_0\). Since the behavior of \(m\) relates directly to the boundedness of \(H_0\) and in turn that of \(H_P\), we dedicate this paragraph to the analysis of \(m\). The goal is to obtain something better than just the trivial bound:
    \begin{equation}\label{eq_partial_m_trivial_est}
        \verts{\partial^\alpha_\xi\partial^\beta_\eta m\br{\xi,\eta}}
        \leq
        \int_{\R}
            \verts{
                e^{2\pi i\br{\xi t+\eta t^2}}
                \underset{
                =:\psi_{\alpha,\beta}\br{t}
                }{
                    \underline{\underline{
                        \br{2\pi i}^{\alpha+\beta}t^{\alpha+2\beta-1}\psi\br{t}
                    }}
                }
            }
        dt
        =
        \Verts{
            \psi_{\alpha,\beta}
        }_{L^1}
        \underset{\alpha,\beta}{\sim}
        \Verts{
            \psi
        }_{L^1}
        \sim
        1.
    \end{equation}
    By summing over all scales \(j\), we will recover the boundedness of \(H_P\).

        -- \textbf{Non-oscillatory part:} \(\verts{\xi}\vee\verts{\eta} \lesssim 1\). Heuristically, the lack of oscillation,
        \begin{equation*}
            e^{2\pi i\br{\xi t+\eta t^2}}\approx e^{0}=1
        \end{equation*}
        and the fact that \(\frac{\psi\br{t}}{t}\) is odd (has mean zero) together suggest that
        \begin{equation*}
            m\br{\xi,\eta}\approx
            \int_{\R}
                1\cdot\frac{\psi\br{t}}{t}
            dt=0.
        \end{equation*}
        Formally, we can measure the error
        \begin{equation}\label{eq non oscillatory}
            \begin{aligned}
                \verts{m\br{\xi,\eta}}=& \verts{m\br{\xi,\eta}-0}
                =
                \verts{
                    \int_{\R} 
                        \br{e^{2\pi i \br{\xi t+\eta t^2}}-1}
                        \frac{\psi\br{t}}{t}
                    dt
                }\\
                \leq & 
                \int_{\R} 
                    \verts{e^{2\pi i \br{\xi t+\eta t^2}}-1}
                    \frac{\verts{\psi\br{t}}}{\verts{t}}
                dt
                \lesssim 
                \int_{\verts{t}\sim 1}
                    \verts{\xi t+\eta t^2}
                    \frac{\verts{\psi\br{t}}}{\verts{t}}
                dt\\
                \lesssim & \verts{\xi}+\verts{\eta} \sim \verts{\xi}\vee\verts{\eta}
            \end{aligned}
        \end{equation}
        -- \textbf{Non-stationary part:} \(\verts{\xi}\vee\verts{\eta}\gg \verts{\xi}\wedge\verts{\eta} \gg 1\). 
        The drastic size difference between \(\xi\) and \(\eta\) forces the derivative of the phase function to be large. Indeed,
        \begin{equation*}
            \verts{t}\sim 1
            \implies
            \verts{\frac{d}{dt}\br{\xi t+\eta t^2}}=
            \verts{\xi+2\eta t}\sim \verts{\xi}\vee\verts{\eta}\gg 1.
        \end{equation*}
        The non-stationary phase principle suggests an arbitrary fast order of decay:
        \begin{equation}\label{eq non-stationary}
            \verts{\partial^\alpha_\xi\partial^\beta_\eta m\br{\xi,\eta}}\underset{\alpha,\beta, N}{\lesssim}\br{\verts{\xi}\vee\verts{\eta}}^{-N} \ll 1.
        \end{equation}
        To verify the above statement, we first observe that
        \begin{equation*}
            e^{2\pi i\br{\xi t+\eta t^2}}
            =\br{2\pi i }^{-N}\br{\frac{1}{\xi+2\eta t} \cdot
            \frac{d}{dt}}^N e^{2\pi i\br{\xi t+\eta t^2}}.
        \end{equation*}
        By integration by parts\footnote{Recalling the earlier the notation $\psi_{\ap, \beta}(t) = (2\pi i )^{\ap+ \beta} t^{\ap + 2\beta -1} \psi(t)$.},
        \begin{equation*}
            \begin{aligned}
                \partial^\alpha_\xi\partial^\beta_\eta m\br{\xi,\eta}
                = & \int_{\R}
                    e^{2\pi i\br{\xi t+\eta t^2}}
                    \psi_{\alpha,\beta}\br{t}
                dt\\
                =  \br{2\pi i}^{-N}
                    &
                \int_{\R} 
                    \psi_{\alpha,\beta}\br{t}
                    \br{\frac{1}{\xi+2\eta t} \cdot
                    \frac{d}{dt}}^N 
                    e^{2\pi i\br{\xi t+\eta t^2}}
                dt\\
                =   \br{-2\pi i}^{-N}
                    &
                \int_{\R}
                    e^{2\pi i\br{\xi t+\eta t^2}}
                    \br{
                        \frac{d}{dt}\frac{1}{\xi+2\eta t} 
                    }^N
                    \psi_{\alpha,\beta}\br{t}
                dt.
            \end{aligned}
        \end{equation*}
        The differential operator
        \begin{equation*}
            \br{
                \frac{d}{dt}\frac{1}{\xi+2\eta t}
            }^N
        \end{equation*}
        can in fact be decomposed into a linear combination of operators of the form:
        \begin{equation*}
            \frac{\eta^j}{\br{\xi+2\eta t}^{N+j}}\cdot\frac{d^{N-j}}{dt^{N-j}}.
        \end{equation*}
        for $j \leq N$. We will prove it by induction: 
        \begin{itemize}
            \item The base case \(N=0\) holds trivially.
            \item Assume the statement is true for \(N=k\). The operator
            \begin{equation*}
                \br{
                    \frac{d}{dt}\frac{1}{\xi+2\eta t}
                }
                \br{
                    \frac{d}{dt}\frac{1}{\xi+2\eta t} 
                }^{k}
            \end{equation*}
            can thus be decomposed into a linear combination of the operators of the form:
            \begin{equation*}
                \begin{aligned}
                    &\br{
                        \frac{d}{dt}\frac{1}{\xi+2\eta t} 
                        \cdot
                    }
                    \frac{\eta^j}{\br{\xi+2\eta t}^{k+j}}\cdot\frac{d^{k-j}}{dt^{k-j}}
                \end{aligned}
            \end{equation*}
            where $j \leq k$. We can rewrite this expression as
            \begin{equation*}
                \begin{aligned}
                    \frac{d}{dt}
                    \br{
                        \frac{\eta^j}{\br{\xi+2\eta t}^{k+1+j}}\cdot\frac{d^{k-j}}{dt^{k-j}}
                    }
                    =& -2\br{k+j+1}\cdot
                    \frac{\eta^{j+1}}{\br{\xi+2\eta t}^{k+1+j+1}}\cdot\frac{d^{k-j}}{dt^{k-j}}\\
                    +&
                    \frac{\eta^j}{\br{\xi+2\eta t}^{k+1+j}}\cdot\frac{d^{k+1-j}}{dt^{k+1-j}}.
                \end{aligned}
            \end{equation*}
        This finishes the proof of the claim.
        \end{itemize}
        As a direct consequence, we have the following estimate:
        \begin{equation*}
        \verts{
            \frac{\eta^j}{\br{\xi+2\eta t}^{N+j}}\cdot\frac{d^{N-j}}{dt^{N-j}}\psi_{\alpha,\beta}
        }
        \underset{\alpha,\beta,N}{
        \lesssim}
        \frac{\verts{\eta}^j}{\br{\verts{\xi}\vee\verts{\eta}}^{N+j}}\cdot\Verts{\psi}_{C^N}
        \underset{N}
        {\lesssim}
        \br{\verts{\xi}\vee\verts{\eta}}^{-N}.
        \end{equation*}
        Finally, we conclude that
        \begin{equation*}
            \begin{aligned}
                \verts{\partial^\alpha_\xi\partial^\beta_\eta m\br{\xi,\eta}}
                \underset{N}{\lesssim}
                \int_{\verts{t}\sim 1}
                    \verts{
                        \br{
                            \frac{d}{dt}\frac{1}{\xi+2\eta t} 
                            \cdot
                        }^N
                        \psi_{\alpha,\beta}\br{t}
                    }
                dt
                \underset{\alpha,\beta,N}{\lesssim} & \br{\verts{\xi}\vee\verts{\eta}}^{-N}.
            \end{aligned}
        \end{equation*}
        -- \textbf{Stationary part:} \(\verts{\xi}\sim
        \verts{\eta} \gg 1\). A simple calculation shows that
        \begin{equation}
            0=\left.\frac{d}{dt}\right\vert_{t=t_0}\br{\xi t+\eta t^2}=\xi+2\eta t_0\iff t_0=\frac{-\xi}{2\eta}
        \end{equation}
        where by our assumption, \(\verts{t_0}\sim 1\). Roughly speaking, the critical point \(t_0\) lies within a neighborhood of \( \supp \frac{\psi\br{t}}{t}\). Usually, we  resort to
        the stationary phase principle:
        \begin{equation}
            m\br{\xi,\eta}\approx e^{2\pi i\br{\frac{\sgn\br{\eta}}{8}-\frac{\xi^2}{4\eta}}}\frac{\psi\br{\frac{-\xi}{2\eta}}}{\verts{\eta}^{1/2}\cdot \frac{-\xi}{2\eta}}.
        \end{equation}
        Yet, we don't actually need the full strength of the stationary phase formula, but rather, a simple application of the van der Corput lemma to obtain a control on the amplitude:
        \begin{equation}\label{eq van der corput's}
            \verts{m\br{\xi,\eta}}
            \lesssim 
            \verts{
                \frac{\partial^2}{\partial t^2} 
                2\pi i\br{\xi t +\eta t^2}
            }^{-\frac{1}{2}}
            \eqsim
            \verts{\eta}^{-\frac{1}{2}}.
        \end{equation}
        As a direct consequence, we have the following form of smoothing inequality:
        \begin{lemma}[A cheap smoothing inequality for \(H_0\)]\label{lemma cheap smoothing}
            Given a function \(f\in L^2\br{\R^2}\) with \(\widehat{f}\) supported on the region \(\left\{\br{\xi,\eta}\in\R^2\middle\vert  \  \lambda\sim\verts{\xi}\vee\verts{\eta}\right\}\), we have the following estimate: 
            \begin{equation*}
                \Verts{H_0 f}_{L^2_{\rm loc}}\lesssim \lambda^{-1/2}\Verts{f}_{L^2}.
            \end{equation*}
        In other words, we have:
            \begin{equation*}
                \Verts{H_0 f}_{L^2_{\rm loc}}\lesssim \Verts{f}_{H^{\br{-1/2,0}}}\wedge\Verts{f}_{H^{\br{0,-1/2}}}.
            \end{equation*}
        \end{lemma}

\vspace{.2in}

    • \textbf{Littlewood-Paley decomposition:}
    Consider the Littlewood-Paley projection of \(H_0\):
    \begin{equation*}
        H_0=\sum_{k\in\Z^2}H_0 \br{\Delta_{k_1}\otimes\Delta_{k_2}}.
    \end{equation*}
    Using the particular dilation symmetry of $H_0$, we obtain an adapted Littlewood-Paley decomposition of \(H_j\): 
    \begin{equation*}
        \begin{aligned}
            H_j= & \Dil^p_{2^{-j},2^{-2j}} H_0 \Dil^p_{2^j,2^{2j}}\\
            = & \sum_{k\in\Z^2} \Dil^p_{2^{-j},2^{-2j}} H_0 \br{\Delta_{k_1}\otimes\Delta_{k_2}} \Dil^p_{2^j,2^{2j}}\\
            = & \sum_{k\in\Z^2} \Dil^p_{2^{-j},2^{-2j}} H_0  \Dil^p_{2^j,2^{2j}}
            \br{\Delta_{j+k_1}\otimes\Delta_{2j+k_2}}\\
            = & \sum_{k\in\Z^2} H_j
            \br{\Delta_{j+k_1}\otimes\Delta_{2j+k_2}}.
        \end{aligned}
    \end{equation*}
In this way, \(k\in\Z^2\) reflects the relative frequency scale size:
    \begin{equation*}
        \begin{aligned}
            \mathcal{F}
            \br{
                H_j\br{
                    \Delta_{k_1+j}\otimes\Delta_{k_2+2j}
                }f
            }
            \br{\xi,\eta}
            = &
             m\br{2^{-j}\xi,2^{-2j}\eta}
            \psi_{k_1}\br{2^{-j}\xi}
            \psi_{k_2}\br{2^{-2j}\eta}
            \widehat{f}\br{\xi,\eta}.
        \end{aligned}
    \end{equation*}
    Altogether, we obtain the following decomposition of $H_P$:
    \begin{equation}
        H_P=
        \sum_{j\in\Z}
            \sum_{
                k\in\Z^2
            }
                H_j\br{\Delta_{j+k_1}\otimes\Delta_{2j+k_2}}.
    \end{equation}
    We can now utilize the information about the phase behavior we've acquired so far to consider three separate cases:
    \begin{equation}
        H_P=H^L+H^M+H^H
    \end{equation}
    where
    \begin{equation}
        \left\{
            \begin{aligned}
                H^L = &
                \sum_{j\in\Z}
                    \sum_{\substack{
                        k\in\Z^2\\
                        : k_1\vee k_2\leq 0
                    }}
                        H_j\br{\Delta_{j+k_1}\otimes\Delta_{2j+k_2}}
                \\
                H^M = &
                \sum_{j\in\Z}
                    \sum_{\substack{
                        k\in\Z^2\\
                        : k_1\vee k_2>0\\
                        \verts{k_1-k_2}\geq 100
                    }}
                        H_j\br{\Delta_{j+k_1}\otimes\Delta_{2j+k_2}}\\
                H^H = &
                \sum_{j\in\Z}
                    \sum_{\substack{
                        k\in\Z^2\\
                        : k_1\vee k_2>0\\
                        \verts{k_1-k_2}< 100
                    }}
                        H_j\br{\Delta_{j+k_1}\otimes\Delta_{2j+k_2}}.
            \end{aligned}
        \right.
    \end{equation}
    This rearrangement separates different phase behaviors of the multiplier corresponding to \(H_j\br{\Delta_{j+k_1}\otimes\Delta_{2j+k_2}}\).

\subsection{A cheap $L^2$ bound} \ 

We now demonstrate a cheap bound through direct calculation of the \(L^\infty\) norm of the multiplier restricted to the three cases listed above. 

\subsubsection{Low and mixed frequencies} \

-- \(H^L\) \textbf{pieces:} We may rewrite the multiplier of \(H^L\) as:
        \begin{equation}
            \begin{aligned}
                m_L\br{\xi,\eta}
                :=&
                \sum_{j\in\Z}
                    \sum_{\substack{
                        k\in\Z^2\\
                        :k_1\vee k_2\leq 0
                    }}
                        m\br{2^{-j}\xi,2^{-2j}\eta}
                        \psi_{k_1}\br{
                            2^{-j}\xi
                        }
                        \psi_{k_2}\br{
                            2^{-2j}\eta
                        }\\
                = &
                \sum_{
                    j\in\Z
                }
                    m\br{2^{-j}\xi,2^{-2j}\eta}
                    \phi\br{
                        2^{-j}\xi
                    }
                    \phi\br{
                        2^{-2j}\eta
                    }.
            \end{aligned}
        \end{equation}
        By construction, \(m_L\) satisfies anisotropic dilation symmetry:
        \begin{equation}
            m_L\br{2^{-j_0}\xi,2^{-2j_0}\eta}=m_L\br{\xi,\eta}.
        \end{equation}
        After rescaling, we can thus focus on the analysis of \(m_L\) on the frequency region \(\verts{\xi}\vee\verts{\eta} \sim  1\) and apply the estimate corresponding to the \textbf{non-oscillatory phase} behavior (see \eqref{eq non oscillatory}) and obtain:
        \begin{equation}\label{eq mL bdd}
            \begin{aligned}
                \verts{m_L\br{\xi,\eta}}
                \leq &
                \sum_{j\in\Z}
                    \verts{
                        m\br{2^{-j}\xi,2^{-2j}\eta}
                        \phi\br{2^{-j}\xi}
                        \phi\br{2^{-2j}\eta}
                    }\\
                = &
                \sum_{j \geq 0}
                    \verts{
                        m\br{2^{-j}\xi,2^{-2j}\eta}
                        \phi\br{2^{-j}\xi}
                        \phi\br{2^{-2j}\eta}
                    }\\
                \lesssim &
                \sum_{j \geq 0}
                    \br{
                        \verts{2^{-j}\xi}\vee
                        \verts{2^{-2j}\eta}
                    }
                \lesssim
                \sum_{j \geq 0}
                    2^{-j}
                \sim 1.
            \end{aligned}
        \end{equation}
        This proves the \(L^2\) boundedness of \(H^L\).

         -- \(H^M\) \textbf{pieces:} With the same trick mentioned above, we rewrite the multiplier of \(H^M\) into the following two pieces:
        \begin{equation}
            \begin{aligned}
                m_M\br{\xi,\eta}:= &
                \sum_{j\in\Z}
                    \sum_{\substack{
                        k\in\Z^2\\
                        :k_1\vee k_2>0\\
                        \verts{k_1-k_2}\geq 100
                    }}
                        m\br{2^{-j}\xi,2^{-2j}\eta}
                        \psi_{k_1}\br{
                            2^{-j}\xi
                        }
                        \psi_{k_2}\br{
                            2^{-2j}\eta
                        }\\
                = &
                \sum_{j\in\Z}
                    \sum_{k_1\in\N}
                        m\br{2^{-j}\xi,2^{-2j}\eta}
                        \psi_{k_1}\br{
                            2^{-j}\xi
                        }
                        \phi_{k_1-100}\br{
                            2^{-2j}\eta
                        }\\
                + &
                \sum_{j\in\Z}
                    \sum_{k_2\in\N}
                        m\br{2^{-j}\xi,2^{-2j}\eta}
                        \phi_{k_2-100}\br{
                            2^{-j}\xi
                        }
                        \psi_{k_2}\br{
                            2^{-2j}\eta
                        }.
            \end{aligned}
        \end{equation}
        The estimate corresponding to the \textbf{non-stationary phase behavior} (see \eqref{eq non-stationary}) yields
        \begin{equation}
            \begin{aligned}
                \verts{m_M\br{\xi,\eta}}
                \leq &
                \sum_{j\in\Z}
                    \sum_{k_1\in\N}
                        \verts{
                            m\br{2^{-j}\xi,2^{-2j}\eta}
                            \psi_{k_1}\br{
                                2^{-j}\xi
                            }
                            \phi_{k_1-100}\br{
                                2^{-2j}\eta
                            }
                        }\\
                + &
                \sum_{j\in\Z}
                    \sum_{k_2\in\N}
                        \verts{
                            m\br{2^{-j}\xi,2^{-2j}\eta}
                            \phi_{k_2-100}\br{
                                2^{-j}\xi
                            }
                            \psi_{k_2}\br{
                                2^{-2j}\eta
                            }
                        }\\
                \underset{N}{\lesssim} &
                \sum_{k_1\in\N}
                    2^{-N k_1}
                    \sum_{j\in\Z}
                        \verts{
                            \psi_{k_1}\br{
                                2^{-j}\xi
                            }
                        }\\
                + &
                \sum_{k_2\in\N}
                    2^{-N k_2}
                    \sum_{j\in\Z}
                        \verts{
                            \psi_{k_2}\br{
                                2^{-2j}\eta
                            }
                        }
                \lesssim 
                \sum_{
                    k\in\N
                }
                    2^{-N k}
                \lesssim 1
            \end{aligned}
        \end{equation}
        This proves the \(L^2\) boundedness of \(H^M\).

\subsubsection{High frequencies} \

         -- \(H^H\) \textbf{pieces:} The multiplier of \(H^H\) takes the following form:
        \begin{equation}
            \begin{aligned}
                m_H\br{\xi,\eta}:= &
                \sum_{j\in\Z}
                    \sum_{\substack{
                        k\in\Z^2\\
                        :k_1\vee k_2>0\\
                        \verts{k_1-k_2}< 100
                    }}
                        m\br{2^{-j}\xi,2^{-2j}\eta}
                        \psi_{k_1}\br{2^{-j}\xi}
                        \psi_{k_2}\br{2^{-2j}\eta}\\
                = &
                \sum_{\substack{
                    k\in\Z^2\\
                    :k_1\vee k_2>0\\
                    \verts{k_1-k_2}<100
                }}
                    m_k\br{\xi,\eta},
            \end{aligned}
        \end{equation}
        where
        \begin{equation}
            m_k\br{\xi,\eta}
            :=
            \sum_{
                j\in\Z
            }
                m\br{2^{-j}\xi,2^{-2j}\eta}
                \psi_{k_1}\br{2^{-j}\xi}
                \psi_{k_2}\br{2^{-2j}\eta}.
        \end{equation}
        By the \textbf{amplitude} information of the \textbf{stationary phase} (see \eqref{eq van der corput's}), we have:
        \begin{equation}
            \verts{m_k\br{\xi,\eta}}
            \leq
            2^{-k_1/2}
            \sum_{j\in\Z}
                \psi_{k_1}\br{2^{-j}\xi}
                \psi_{k_2}\br{2^{-2j}\eta}
            \lesssim 2^{-k_1/2}
        \end{equation}
        Recalling that $|k_1 - k_2| <100$, we sum over \(k_1 \in \Z\) we obtain the \(L^2\) boundedness of \(H^H\).

\subsection{$L^p$ bound} \ 

Here is a sketch of our strategy to prove \(L^p\) boundedness: 
    \begin{equation*}
        H_P \leadsto
        \left\{
            \begin{aligned}
                H^L:& & \overset{\textbf{Non-Oscillatory}}{\Longrightarrow}  &\text{ Less singular Calderón-Zygmund Operator}\\
                H^M:& & \overset{\textbf{Non-Stationary}}{\Longrightarrow}    &\text{ Less singular Calderón-Zygmund Operator}\\
                H^H:& & \overset{\textbf{\phantom{No}Stationary\phantom{n-}}}{\Longrightarrow}   &\text{ Smoothing Inequality}
            \end{aligned}
        \right.
    \end{equation*}

\subsubsection{Low and mixed frequencies} \

        -- \(H^L\) \textbf{case:} Recall that \(m_L\) satisfies an anisotropic dilation symmetry:
        \begin{equation}
            m_L\br{2^{-j_0}\xi,2^{-2j_0}\eta}=m_L\br{\xi,\eta}.
        \end{equation}
        After rescaling, we thus only need to focus on the analysis of \(m_L\) on the frequency region \(\verts{\xi}\vee\verts{\eta} \sim 1\). To verify the Mikhlin multiplier criteria, 
        \begin{equation}
            \begin{aligned}
                \verts{
                    \partial^\alpha_\xi
                    \partial^\beta_\eta
                    m_L\br{\xi,\eta}
                }
                \leq &
                \sum_{j\in\Z}
                    \verts{
                        \partial^\alpha_\xi
                        \partial^\beta_\eta
                        \br{
                            m\br{2^{-j}\xi,2^{-2j}\eta}
                            \phi\br{2^{-j}\xi}
                            \phi\br{2^{-2j}\eta}
                        }
                    }\\
                = &
                \sum_{j \geq 0}
                    \verts{
                        \partial^\alpha_\xi
                        \partial^\beta_\eta
                        \br{
                            m\br{2^{-j}\xi,2^{-2j}\eta}
                            \phi\br{2^{-j}\xi}
                            \phi\br{2^{-2j}\eta}
                        }
                    }\\
                \underset{\alpha,\beta}{\lesssim} &
                \sum_{j \geq 0}
                    2^{-j\br{\alpha+2\beta}}
                \sim 1\textrm{ whenever }\alpha+2\beta>0.
            \end{aligned}
        \end{equation}
        Combined with the result \(\verts{m_L}\lesssim 1\) we derived in \eqref{eq mL bdd}, a rescaling yields the \textbf{anisotropic Mikhlin multiplier criteria}:
        \begin{equation}\label{eq_mL_sym_est}
            \verts{
            \partial^\alpha_\xi
            \partial^\beta_\eta
            m_L\br{\xi,\eta}
            }\underset{\alpha,\beta}{\lesssim}
            \br{\verts{\xi}+\verts{\eta}^{\frac{1}{2}}}^{-\alpha-2\beta}
        \end{equation}
        We thus have the \(L^p\) boundedness of \(H^L\) for all \(1<p<\infty\).
        
        -- \(H^M\) \textbf{pieces:} We again only need to verify the \textbf{anisotropic Mikhlin multiplier criteria} for \(\partial^\alpha_\xi\partial^\beta_\eta m_M\) on the frequency region \( \verts{\xi}\vee\verts{\eta} 
        \sim 1\). The \textbf{non-stationary phase} estimate \eqref{eq non-stationary} and the chain rule yield:
        \begin{equation}\label{eq_mM_sym_est}
            \begin{aligned}
                \verts{
                    \partial^\alpha_\xi
                    \partial^\beta_\eta
                    m_M\br{\xi,\eta}
                }
                \leq &
                \sum_{j\in\Z}
                    \sum_{k_1\in\N}
                        \verts{
                            \partial^\alpha_\xi
                            \partial^\beta_\eta
                            \br{
                                m\br{2^{-j}\xi,2^{-2j}\eta}
                                \psi_{k_1}\br{
                                    2^{-j}\xi
                                }
                                \phi_{k_1-100}\br{
                                    2^{-2j}\eta
                                }
                            }
                        }\\
                + &
                \sum_{j\in\Z}
                    \sum_{k_2\in\N}
                        \verts{
                            \partial^\alpha_\xi
                            \partial^\beta_\eta
                            \br{
                                m\br{2^{-j}\xi,2^{-2j}\eta}
                                \phi_{k_2-100}\br{
                                    2^{-j}\xi
                                }
                                \psi_{k_2}\br{
                                    2^{-2j}\eta
                                }
                            }
                        }\\
                \underset{\alpha,\beta,N,M}{\lesssim} &
                \sum_{\substack{
                    j,k_1\in\Z\\
                    :k_1>0
                }}
                    2^{-\br{N-\alpha-\beta} k_1-\alpha\br{j+k_1}-\beta\br{2j+k_1-100}}
                    \1_{\verts{\xi}\sim 2^{j+k_1}}
                    \1_{\verts{\eta}\lesssim 2^{2j+k_1-100}}
                    \\
                + &
                \sum_{\substack{
                    j,k_2\in\Z\\
                    :k_2>0
                }}
                    2^{-\br{M-\alpha-\beta} k_2-\alpha\br{j+k_2-100}-\beta\br{2j+k_2}}
                    \1_{\verts{\xi}\lesssim 2^{j+k_2-100}}
                    \1_{\verts{\eta}\sim 2^{2j+k_2}}\\
                \underset{\alpha,\beta}{\lesssim} &
                \sum_{\substack{
                    j,k_1\in\Z\\
                    :k_1>0
                }}
                    2^{-\br{N-\alpha-2\beta} k_1}\verts{\xi}^{-\alpha-2\beta}
                    \1_{\verts{\xi}\sim 2^{j+k_1}}
                    \1_{\verts{\eta}\ll\verts{\xi}^2}
                    \\
                + &
                \sum_{\substack{
                    j,k_2\in\Z\\
                    :k_2>0
                }}
                    2^{-\br{M-\frac{\alpha}{2}-\beta} k_2}\verts{\eta}^{-\frac{\alpha}{2}-\beta}
                    \1_{2^{-k_2} \ll \frac{\verts{\eta}}{\verts{\xi}^2}}
                    \1_{\verts{\eta}\sim 2^{2j+k_2}}\\
                \underset{\alpha,\beta}{\lesssim} &
                1
                + 
                \br{1 \wedge \frac{\verts{\eta}}{\verts{\xi}^2}}^{M-\frac{\alpha}{2}-\beta}\verts{\eta}^{-\frac{\alpha}{2}-\beta}
                \lesssim 1\textrm{ whenever }M,N> \alpha+2\beta>0
            \end{aligned}
        \end{equation}
        We thus have the \(L^p\) boundedness of \(H^M\) for all \(1<p<\infty\).

\subsubsection{High frequency} \

        -- \(H^H\) \textbf{pieces:} Recall the following decomposition:
        \begin{equation}
            m_H\br{\xi,\eta}= 
            \sum_{\substack{
                k\in\Z^2\\
                :k_1\vee k_2>0\\
                \verts{k_1-k_2}<100
            }}
                m_k\br{\xi,\eta},
        \end{equation}
        where \(m_k\br{\xi,\eta}\) corresponds to the multiplier for the operator $H^{(k)}$:
        \begin{equation}
            H^{\br{k}}:=
            \sum_{j\in\Z}
                H_j\circ\br{\Delta_{j+k_1}\otimes \Delta_{2j+k_2}}.
        \end{equation}
        We've already established an \(L^2\) bound with exponential decay from the previous discussion (see \eqref{eq van der corput's}):
        \begin{equation}
            \Verts{
                H^{\br{k}} f
            }_{L^2}
            \lesssim \Verts{m_k}_{L^\infty}
            \Verts{f}_{L^2}
            \sim
            2^{-|k|/2}
            \Verts{f}_{L^2}.
        \end{equation}
        We aim to obtain the analogous statement for the \(L^p\) bound:
        \begin{equation}
            \Verts{
                H^{\br{k}} f
            }_{L^p}
            \underset{p}{\lesssim} 2^{-\epsilon_p \verts{k}}
            \Verts{f}_{L^p}
        \end{equation}
        and sum over \(k\) to recover the \(L^p\) boundedness of \(H^H\). Yet, thanks to the decay at \(p=2\), we can be much more lenient with the estimate through interpolation. More specifically, the following bound with polynomial growth would suffice for our purpose:
        \begin{equation}\label{eq_H_k_Lp_poly_growth_est}
            \Verts{H^{\br{k}} f}_{L^p}
            \underset{p,N}{\lesssim} \verts{k}^N
            \Verts{f}_{L^p}.
        \end{equation}
        To prove the above estimate, we first look at the individual pieces and recall that:
        \begin{equation}\label{eq_dil_rel_HH_one_piece}
            H_j\circ\br{\Delta_{j+k_1}\otimes \Delta_{2j+k_2}}=\Dil^p_{2^{-j},2^{-2j}}\circ H_0\circ\br{\Delta_{k_1}\otimes \Delta_{k_2}}\circ \Dil^p_{2^j,2^{2j}}.
        \end{equation}
        We now focus on the estimate on the central part \(H_0\circ \br{\Delta_{k_1}\otimes \Delta_{k_2}}\). On the physical side, we have:
        \begin{equation}\label{eq_H_0kk_naive_est}
            \verts{H_0\br{\Delta_{k_1}\otimes \Delta_{k_2}}f}
            \leq 
            \verts{f} \ast \verts{K_k},
        \end{equation}
        where \(\widehat{K_k}(\xi, \eta ):=m\br{\xi,\eta}\psi_{k_1}\br{\xi}\psi_{k_2}\br{\eta}\). 
        It remains to estimate the kernel \(K_k\). Explicit calculation shows that:
        \begin{equation*}
            \begin{aligned}
                K_k\br{x,y}
                =&
                \int_{\R}
                    \int_{\R^2}
                        e^{
                            2\pi i
                            \br{
                                \xi\br{x+t}+
                                \eta\br{y+t^2}
                            }
                        }
                        \psi_{k_1}\br{\xi}
                        \psi_{k_2}\br{\eta}
                    d\xi d\eta
                    \frac{\psi\br{t}}{t}
                dt\\
                = &
                \int_{\R}
                    \br{
                        \Dil^1_{2^{-k_1}}
                        \widecheck{\psi}
                    }\br{x+t}
                    \br{
                        \Dil^1_{2^{-k_2}}
                        \widecheck{\psi}
                    }\br{y+t^2}
                    \frac{\psi\br{t}}{t}
                dt\\
                = &
                \int_{\R}
                    \br{
                        \Tr_{-t}
                        \Dil^1_{2^{-k_1}}
                        \widecheck{\psi}
                    }\br{x}
                    \br{
                        \Tr_{-t^2}
                        \Dil^1_{2^{-k_2}}
                        \widecheck{\psi}
                    }\br{y}
                    \frac{\psi\br{t}}{t}
                dt\\
                = &
                \int_{\R}
                    \br{
                        \Dil^1_{2^{-k_1}}
                        \Tr_{-2^{k_1} t}
                        \widecheck{\psi}
                    }\br{x}
                    \br{
                        \Dil^1_{2^{-k_2}}
                        \Tr_{-2^{k_2}t^2}
                        \widecheck{\psi}
                    }\br{y}
                    \frac{\psi\br{t}}{t}
                dt.
            \end{aligned}
        \end{equation*}
        By dominating \(\verts{\widecheck{\psi}(x)} \underset{N}{\lesssim} \la x \ra^{-N} \), we can estimate the expression directly:
        \begin{equation*}
                \verts{
                    K_k\br{x,y}
                }
                \underset{N}{\lesssim} 
                \Verts{
                    \br{
                        \Dil^1_{2^{-k_1}}
                        \Tr_{-2^{k_1} t}
                        \la \cdot \ra^{-N}
                    }\br{x}
                    \br{
                        \Dil^1_{2^{-k_2}}
                        \Tr_{-2^{k_2}t^2}
                        \la \cdot \ra^{-N}
                    }\br{y}
                }_{L^1\br{\verts{t}\sim 1}}.
        \end{equation*}
        We then introduce the following naive but useful discretization:
        \begin{lemma}\label{lem_omega_discr}
            \begin{equation*}
                \la x \ra^{-1} \sim \sum_{z\in\Z} \la z \ra^{-1} \1_{\Br{0,1}}(x-z).
            \end{equation*} 
        \end{lemma} 
        This allows us to rewrite the estimate:
        \begin{equation}\label{eq_kernel_to_shifted_I}
            \begin{aligned}
                \verts{
                    K_k\br{x,y}
                }
                \underset{N}{\lesssim} &
                \Bigg\Vert 
                    \br{
                        \Dil^1_{2^{-k_1}}
                        \Tr_{-2^{k_1} t}
                        \sum_{z_1\in\Z} \la z_1 \ra^{-N}
                        \Tr_{z_1}\1_{\Br{0,1}}
                    }\br{x}\\
                    \cdot & \phantom{\big\Vert}
                    \br{
                        \Dil^1_{2^{-k_2}}
                        \Tr_{-2^{k_2}t^2}
                        \sum_{z_2\in\Z} \la z_2 \ra^{-N}
                        \Tr_{z_2}\1_{\Br{0,1}}
                    }\br{y}
                \Bigg\Vert_{L^1\br{\verts{t}\sim 1}}\\
                = &
                \Big\Vert
                    \big\Vert 
                    \la z_1 \ra^{-N}
                    \br{
                            \Dil^1_{2^{-k_1}}\Tr_{z_1-2^{k_1} t}\1_{\Br{0,1}}
                    }\br{x}
                        \\
                &\phantom{\big\vert}
                    \cdot
                    \la z_2 \ra^{-N}
                    \br{
                        \Dil^1_{2^{-k_2}}\Tr_{z_2-2^{k_2} t^2}\1_{\Br{0,1}}
                    }\br{y}
                    \big\Vert_{\ell^1\br{z\in\Z^2}}
                \Big\Vert_{L^1\br{\verts{t}\sim 1}}.
            \end{aligned}
        \end{equation}
        \begin{remark}
            By the above estimate, the kernel \(K_k\) can thus be decomposed into \textbf{weighted summands and averages} of characteristic functions of \textbf{shifted intervals} with \textbf{some dilation factors.} 
        \end{remark}
        To control \(\verts{f}\ast\verts{K_k}\) with the estimate \eqref{eq_kernel_to_shifted_I}, it becomes natural to introduce the \textbf{shifted dyadic maximal function} \(\mathcal{M}_\sigma\)
            \footnote{By Stein's book \cite{Stein1993HABook} (p.78),
            \begin{align*}
                \norm{\M_{\sg}g}{L^p} \lesssim \log(2+ |\sg|)^{1/p} \norm{g}{L^p}, \hspace{.3in} 1 < p \leq \infty
            \end{align*}}
        :
        \begin{equation}
            \begin{aligned}
                \mathcal{M}_{\sigma}g\br{x}
                := &
                \sup_{k\in\Z}
                    2^{-k}
                    \int_{
                        \Br{
                            \sigma 2^k,
                            \br{\sigma+1}2^k
                        }
                    }
                        \verts{g\br{x-t}}
                    dt\\
                = & 
                \sup_{k\in\Z}
                    \verts{g}\ast
                    \Dil^1_{2^k}\Tr_\sigma\1_{\Br{0,1}}\br{x}
                \text{, where }\sigma \in \R.
            \end{aligned}
        \end{equation}
        By definition,
        \begin{equation*}
            \begin{aligned}
                &
                \verts{H_0\br{\Delta_{k_1}\otimes \Delta_{k_2}}f}
                \leq  \verts{f}\ast\verts{K_k}\\
                \underset{N}{\lesssim } &
                \Verts{
                    \Verts{
                        \ang{z_1}^{-N}
                        \ang{z_2}^{-N}
                        \mathcal{M}^{\br{2}}_{z_2-2^{k_2}t^2}
                        \mathcal{M}^{\br{1}}_{z_1-2^{k_1}t} f
                    }_{\ell^1\br{z\in\Z^2}}
                }_{L^1\br{\verts{t}\sim 1}}.
            \end{aligned}
        \end{equation*}
        In addition, since \(\mathcal{M}_\sigma\) commutes with dilation, the relation \eqref{eq_dil_rel_HH_one_piece} gives a pointwise bound uniform in \(j\in\Z\):
        \begin{equation}\label{eq_HHj_max_est}
            \begin{aligned}
                \verts{H_j\br{\Delta_{j+k_1}\otimes \Delta_{2j+k_2}}f}
                \underset{N}{\lesssim }
                \Verts{
                    \Verts{
                        \ang{z_1}^{-N}
                        \ang{z_2}^{-N}
                        \mathcal{M}^{\br{2}}_{z_2-2^{k_2}t^2}\mathcal{M}^{\br{1}}_{z_1-2^{k_1}t} f
                    }_{\ell^1\br{z\in\Z^2}}
                }_{L^1\br{\verts{t}\sim 1}}.
            \end{aligned}
        \end{equation}
        We can improve this bound. By dualizing the expression:
        \begin{multline}
                \ang{H_j\br{\Delta_{j+k_1}\otimes \Delta_{2j+k_2}}f,g}\\
                = 
                \int_{\R^2}
                    m\br{2^{-j}\xi,2^{-2j}\eta}
                        \psi_{k_1}\br{2^{-j}\xi}
                        \psi_{k_2}\br{2^{-2j}\eta}
                        \widehat{f}\br{\xi,\eta}
                        \overline{
                            \widehat{g}\br{\xi,\eta}
                        }
                d\xi d\eta
        \end{multline}
        and making use of the fact that \(\psi=\psi\cdot\widetilde{\psi}\), we can rewrite the above expression as: 
        \begin{equation}
            \begin{aligned}
                &
                \ang{H_j\br{\Delta_{j+k_1}\otimes \Delta_{2j+k_2}}f,g}\\
                = &
                \int_{\R^2}
                    m\br{2^{-j}\xi,2^{-2j}\eta}
                        \widetilde{\psi}_{k_1}\br{2^{-j}\xi}
                        \widetilde{\psi}_{k_2}\br{2^{-2j}\eta}\\
                        &\phantom{\int_{\R}}\cdot 
                        \psi_{k_1}\br{2^{-j}\xi}
                        \psi_{k_2}\br{2^{-2j}\eta}
                        \widehat{f}\br{\xi,\eta}
                        \overline{
                            \widetilde{\psi}_{k_1}\br{2^{-j}\xi}
                            \widetilde{\psi}_{k_2}\br{2^{-2j}\eta}
                            \widehat{g}\br{\xi,\eta}
                        }
                d\xi d\eta\\
                = &
                \ang{
                    H_j\br{\widetilde{\Delta}_{j+k_1}\otimes \widetilde{\Delta}_{2j+k_2}}f_{j,k},
                    g_{j,k}
                },
            \end{aligned}
        \end{equation}
        where for simplicity, we adopt the following shorthand:
        \begin{equation*}
        \left\{
            \begin{aligned}
                f_{j,k}:=&\br{\Delta_{j+k_1}\otimes \Delta_{2j+k_2}}f\\
                g_{j,k}:=&\br{\widetilde{\Delta}_{j+k_1}\otimes \widetilde{\Delta}_{2j+k_2}}g
            \end{aligned}
        \right.
        \end{equation*}
        After applying the estimate \eqref{eq_HHj_max_est} and summing over \(j\in\Z\), we have
        \begin{equation}
                \verts{\ang{H^{\br{k}} f,g}}
                \underset{N}{\lesssim} 
                \Verts{
                    \Verts{
                        \ang{z_1}^{-N}
                        \ang{z_2}^{-N}
                        \mathcal{B}_{t,z,k}\br{f,g}
                    }_{\ell^1\br{z\in\Z^2}}
                }_{L^1\br{\verts{t}\sim 1}}
        \end{equation}
        where the bilinear form is defined as below:
        \begin{equation}
                \mathcal{B}_{t,z,k}\br{f,g}
                := 
                \sum_{j\in\Z}
                    \ang{
                        \mathcal{M}^{\br{2}}_{z_2-2^{k_2}t^2}
                        \mathcal{M}^{\br{1}}_{z_1-2^{k_1} t} 
                        f_{j,k},
                        \verts{
                        g_{j,k}
                        }
                    }.
        \end{equation}
        For now, we focus on the bilinear form with fixed \(z\in\Z^2\) and \(\verts{t}\sim 1\). Applying Cauchy-Schwartz inequality to \(\sum_{j\in\Z}\) and H\"{o}lder's inequality to the inner product, we get the following:
        \begin{equation}
            \begin{aligned}
                \mathcal{B}_{t,z,k}\br{f,g}
                \leq & 
                \ang{
                    \Verts{
                        \mathcal{M}^{\br{2}}_{z_2-2^{k_2}t^2}\mathcal{M}^{\br{1}}_{z_1-2^{k_1} t} 
                        f_{j,k}
                    }_{\ell^2\br{j}},
                    \Verts{
                        g_{j,k}
                    }_{\ell^2\br{j}}
                }\\
                \leq &
                \Verts{
                    \Verts{
                        \mathcal{M}^{\br{2}}_{z_2-2^{k_2}t^2}\mathcal{M}^{\br{1}}_{z_1-2^{k_1} t} 
                        f_{j,k}
                    }_{\ell^2\br{j}}
                }_{L^p}
                \cdot
                \Verts{
                    \Verts{
                        g_{j,k}
                    }_{\ell^2\br{j}}
                }_{L^q}
            \end{aligned}
        \end{equation}
        We then apply the Fefferman-Stein inequality
            \footnote{\textbf{Fefferman-Stein inequality }- For $1<p<\infty$, the Hardy-Littlewood maximal function $\mathcal{M}$ satisfies the vector-valued inequalities 
            \begin{equation}
                \norm{\lp \sum_{j \in \Z} |\mathcal{M}(f_j)|^2 \rp^{1/2}}{p} \leq C_n (p + \frac{1}{p-1}) \norm{\lp \sum_{j} |f_j|^2 \rp^{1/2}}{L^p}
            \end{equation}}
        established in \cite{GHLR17}:
        \begin{equation}\label{eq_shifted_Fefferman-Stein_ineq}
            \Verts{\Verts{\M_{\sigma} g_j}_{\ell^2\br{j}} }_{L^p} \underset{p}{\lesssim} \log (2 + |\sigma|)^2 \Verts{\Verts{g_j}_{\ell^2\br{j}} }_{L^p}\text{, for }1<p<\infty
        \end{equation}
        twice and finish off with the standard square function estimate to obtain:
        \begin{equation}
            \begin{aligned}
                &
                \mathcal{B}_{t,z,k}\br{f,g}\\
                \underset{p}{\lesssim} & 
                \log\br{2+\verts{z_2-2^{k_2}t^2}}^2
                \log\br{2+\verts{z_1-2^{k_1} t}}^2
                \cdot 
                \Verts{
                    \Verts{
                        f_{j,k}
                    }_{\ell^2\br{j}}
                }_{L^p}
                \cdot
                \Verts{
                    \Verts{
                        g_{j,k}
                    }_{\ell^2\br{j}}
                }_{L^q}\\
                \underset{p}{\lesssim} & 
                \log\br{2+\verts{z_2-2^{k_2}t^2}}^2
                \log\br{2+\verts{z_1-2^{k_1} t}}^2
                \Verts{f}_{L^p}\Verts{g}_{L^q}.
            \end{aligned}
        \end{equation}
        Finally, we have:
        \begin{equation}
            \begin{aligned}
                \verts{
                    \ang{
                        H^{\br{k}} f,
                        g
                    }
                }
                \underset{N,p}{\lesssim}&
                \Big\Vert\phantom{\cdot}
                    \sum_{z_1\in\Z}
                        \ang{z_1}^{-N}
                        \log\br{2+\verts{z_1-2^{k_1} t}}^2
                        \\
                &
                \phantom{\Big\Vert}
                    \cdot\sum_{z_2\in\Z}
                        \ang{z_2}^{-N}
                        \log\br{2+\verts{z_2-2^{k_2}t^2}}^2
                \Big\Vert_{L^1\br{\verts{t}\sim 1}}
                \Verts{f}_{L^p}\Verts{g}_{L^q}\\
                \lesssim &
                \Verts{
                    \log\br{2+\verts{2^{k_1} t}}^2
                    \log\br{2+\verts{2^{k_2} t^2}}^2
                }_{L^1\br{\verts{t}\sim 1}}
                \Verts{f}_{L^p}\Verts{g}_{L^q}\\
                \lesssim &
                \verts{k}^4 \Verts{f}_{L^p}\Verts{g}_{L^q}.
            \end{aligned}
        \end{equation}
        In other words, we've proved \eqref{eq_H_k_Lp_poly_growth_est} with \(N=4\)
        and thus, finish the proof of the \(L^p\) boundedness of \(H^H\).

\newpage

\section{A Variant of the Triangular Hilbert Transform}

We now apply what we've learned from the toy model \(H_P\) to the analysis of the variant of the triangular Hilbert transform defined earlier: 
\begin{equation*}
    T(f_{1},f_{2})(x,y):=\operatorname{p.v.}\int_{\mathbb{R}}f_{1}(x+t,y)f_{2}(x,y+t^{2})\frac{dt}{t}.
\end{equation*}

In this section, we will demonstrate how the boundedness of $T$ reduces to the smoothing inequality and an anisotropic variant of the twisted paraproduct.

To start off the proof of Theorem \ref{thm 1 - THT Lp}, we shall follow the same procedures as that in the analysis of the toy model:
    \subsection{Symmetry of the operator}
    Similar to what we have in the case of \(H_P\), we have the following relations:
    \begin{itemize}
        \item \textbf{Translation symmetry:}
        \begin{equation*}
            \Tr_{x_0,y_0}T\br{f_1,f_2}=T\br{\Tr_{x_0,y_0} f_1,\Tr_{x_0,y_0} f_2}
        \end{equation*}
        \item \textbf{Anisotropic dilation symmetry:}
        \begin{equation*}
            \Dil^r_{2^j,2^{2j}}T(f_{1},f_{2})
            =
            T\br{
                \Dil^p_{2^{j},2^{2j}}f_{1},
                \Dil^q_{2^{j},2^{2j}}f_{2}
            }
            \text{ whenever }\frac{1}{r}=\frac{1}{p}+\frac{1}{q}
        \end{equation*}
        \item \textbf{Trivial Modulation symmetry:}
        \begin{equation*}
            \Mod_{\xi_0,\eta_0}T\br{f_1,f_2}=T\br{\Mod_{0,\eta_0}f_1,\Mod_{\xi_0,0}f_2}.
        \end{equation*}
    \end{itemize}
    Notice \textbf{the absence} of the following form of modulation symmetry:
    \begin{equation*}
        \Mod_{?,?}T\br{f_1,f_2} \overset{?}{=} T\br{\Mod_{\xi_0,0}f_1,\Mod_{0,\eta_0}f_2}.
    \end{equation*}
    This means there could potentially be some biases on the frequency side of those particular arguments. Therefore, we expect the analysis on high-frequency input \(f_1,f_2\) to differ from that on low-frequency input.
    
    \subsection{Decompositions} \ 
    
    We then perform the same dyadic scale decomposition on the kernel
    \(
        \frac{1}{t}=\sum_{j\in\Z} \frac{\psi\br{2^j t}}{t}
    \)
    to respect the \(\Dil\)-symmetry structure of \(T\) and derive:
    \begin{align*}
        T = \sum_{j \in \Z} T_j\text{,\hspace{1.5ex}where\hspace{1.5ex}}T_j(f_1, f_2)(x, y)= \int_{\R} f_1(x+t, y) f_2(x, y+t^2) \frac{\psi(2^j t)}{t} dt.
    \end{align*}
    After a change of variables in $t$, we have
    \begin{equation}
        \begin{aligned}
            T_{j}(f_{1},f_{2})(x,y)
            =&
            \int_{\mathbb{R}}f_{1}(x+t,y)f_{2}(x,y+t^{2})\frac{\psi(2^{j}t)}{t}dt\\
            = & 
            \int_{\mathbb{R}}f_{1}(x+2^{-j}t,y)f_{2}(x,y+2^{-2j}t^{2})\frac{\psi(t)}{t}dt\\
            = &
            \int_{\R^2}
                m\br{2^{-j}\xi,2^{-2j}\eta}
                \widehat{f_1}^{\br{1}}\br{\xi,y}
                \widehat{f_2}^{\br{2}}\br{x,\eta}
                e^{2\pi i \br{x\xi+y\eta}}
            d\xi d\eta,
        \end{aligned}
    \end{equation}
    where we recall the definition of \(m\):
    \begin{equation}\label{eq_multiplier_2nd_appearence}
        m\br{\xi,\eta}
        :=
        \int_{\R}
            e^{2\pi i\br{\xi t+ \eta t^2}}
            \frac{\psi\br{t}}{t}
        dt.
    \end{equation}
    Thus, the anisotropic dilation symmetry takes the following form for the components \(T_j\)s :
    \begin{equation*}
        T_{j}(f_{1},f_{2})=\Dil^r_{2^{-j},2^{-2j}}T_{0}(\Dil^p_{2^{j},2^{2j}}f_{1},\Dil^q_{2^{j},2^{2j}}f_{2})\text{ whenever }\frac{1}{r}=\frac{1}{p}+\frac{1}{q}.
    \end{equation*}
    Again, this dilation symmetry allows us to pass the analysis of \(T_j\) to the analysis on the unit scale \(T_{0}\).
    We point out that \(T_0\) and \(H_0\) share the same multiplier. The key difference lies in the fact that for a bilinear operator like \(T_0\), the \(L^\infty\) boundedness of the multiplier \(m\br{\xi,\eta}\) alone is not strong enough to imply even the local-\(L^2\) boundedness of a bilinear operator. 
    This immediately highlights one of the technicalities.
    Next step, we compose \(T_0\) with Littlewood-Paley projections:
    \begin{equation}
        T_0\br{f_1,f_2}=
        \sum_{\substack{
            k\in\Z^2
        }}
            T_0\br{
                \Delta^{\br{1}}_{k_1}f_1,
                \Delta^{\br{2}}_{k_2}f_2
            }
    \end{equation}
    and examine how it interacts with the dilation symmetry above-mentioned:
    \begin{equation}
        \begin{aligned}
            T_j\br{f_1,f_2}
            = &
            \Dil^r_{2^{-j},2^{-2j}}
            \sum_{
                k\in\Z^2
            }
                T_0\br{
                    \Delta^{\br{1}}_{k_1}\Dil^p_{2^j,2^{2j}}f_1,
                    \Delta^{\br{2}}_{k_2}\Dil^q_{2^j,2^{2j}}f_2
                }\\
            = &
            \sum_{
                k\in\Z^2
            }
                \Dil^r_{2^{-j},2^{-2j}}
                T_0\br{
                    \Dil^p_{2^j,2^{2j}}\Delta^{\br{1}}_{j+k_1}f_1,
                    \Dil^q_{2^j,2^{2j}}\Delta^{\br{2}}_{2j+k_2}f_2
                }\\
            = &
            \sum_{
                k\in\Z^2
            }
                T_j\br{
                    \Delta^{\br{1}}_{j+k_1}f_1,
                    \Delta^{\br{2}}_{2j+k_2}f_2
                }.
        \end{aligned}
    \end{equation}
    Thus, in total, we have:
    \begin{equation}
        T\br{f_1,f_2}
        =
        \sum_{
            j\in\Z
        }
            \sum_{
                k\in\Z^2
            }
                T_j
                \br{
                    \Delta^{\br{1}}_{j+k_1}f_1,
                    \Delta^{\br{2}}_{2j+k_2}f_2
                }.
    \end{equation}
    The structure of the multiplier \(m\br{\xi,\eta}\) suggests that we split \(T\) into the three components (corresponding to low $L$, mixed $M$, and high $H$ frequency components) to study separately:
    \begin{equation}\label{eq 2.1 T LMH}
        T = T^L + T^M +T^H
    \end{equation}
    where each component is defined below:
    \begin{equation*}
        \left\{
        \begin{aligned}
            T^L = &
            \sum_{j \in \Z}       
                \sum_{\substack{k\in \Z^2\\ 
                :k_1 \vee k_2 \leq 0}} T_j
                \br{
                    \D{1}{j+k_1} \otimes \D{2}{2j+k_2} 
                }\footnote{The localization is well-adapted. Indeed, $\D{2}{2j+k_2} f_2$ implies $|y+t^2| \sim 2^{2j}$, in particular $|t^2|\sim 2^{2j}$ then $|t|\sim 2^j$ so $|x+t|\sim 2^j$ makes sense}\\
            T^M = & 
            \sum_{j \in \Z} 
                \sum_{\substack{
                k\in\Z^2\\
                :k_1 \vee k_2 >0 \\ 
                |k_1 - k_2| \geq 100}} T_j
                \br{
                    \D{1}{j+k_1} \otimes \D{2}{2j+k_2}
                }\\
            T^H = & 
            \sum_{j \in \Z} 
                \sum_{\substack{
                k\in\Z^2\\
                :k_1 \vee k_2 >0 \\ 
                |k_1 - k_2| < 100}} T_j
                \br{
                    \D{1}{j+k_1} \otimes \D{2}{2j+k_2}
                }.
        \end{aligned}
        \right.
    \end{equation*}
    
    \subsection{A proof sketch}
    Below, we provide a sketch of our strategy for \(L^p\times L^q\to L^r\) bounds with \(\frac{1}{2}<\frac{1}{r}=\frac{1}{p}+\frac{1}{q}\leq 1\):
    \begin{equation}
        T \leadsto
        \left\{
            \begin{aligned}
                T^L:& & \overset{\textbf{Non-Oscillatory}}{\Longrightarrow}  &\text{ Less singular twisted paraproduct}\\
                T^M:& & \overset{\textbf{Non-Stationary}}{\Longrightarrow}    &\text{ Less singular twisted paraproduct}\\
                T^H:& & \overset{\textbf{\phantom{No}Stationary\phantom{n-}}}{\Longrightarrow}   &\text{ Smoothing Inequality}
            \end{aligned}
        \right.
    \end{equation}
    We will focus on the challenging section of the proof: the high-frequency component $T^H$. The central idea here is the smoothing inequality in Theorem \ref{thm 5 - smoothing}. 
    In contrast, the low-frequency component $T^L$ and the mixed-frequency component $T^M$ both reduce to the case of an anisotropic variant of the twisted paraproduct. 
        \subsection{Low and mixed-frequency pieces} \ 
        
        A direct calculation shows that:
        \begin{equation*}
            T^L\br{f_1,f_2}\br{x,y}
            =
            \int_{\R^2}
                m_L\br{\xi,\eta}
                \widehat{f_1}^{\br{1}}\br{\xi,y}
                \widehat{f_2}^{\br{2}}\br{x,\eta}
                e^{2\pi i \br{x\xi+y\eta}}
            d\xi d\eta,
        \end{equation*}
        and
        \begin{equation*}
            T^M\br{f_1,f_2}\br{x,y}
            =
            \int_{\R^2}
                m_M\br{\xi,\eta}
                \widehat{f_1}^{\br{1}}\br{\xi,y}
                \widehat{f_2}^{\br{2}}\br{x,\eta}
                e^{2\pi i \br{x\xi+y\eta}}
            d\xi d\eta.
        \end{equation*}
        where we've established the \textbf{anisotropic symbol estimates \eqref{eq_mL_sym_est}:}
        \begin{equation*}
            \verts{\partial^\alpha_\xi\partial^\beta_\eta m_L\br{\xi,\eta}}
            \underset{\alpha,\beta}{\lesssim}
            \br{
                \verts{\xi}+
                \verts{\eta}^{\frac{1}{2}}
            }^{-\alpha-2\beta}.
        \end{equation*}
        and \eqref{eq_mM_sym_est}:
        \begin{equation*}
            \verts{\partial^\alpha_\xi\partial^\beta_\eta m_M\br{\xi,\eta}}
            \underset{\alpha,\beta}{\lesssim}
            \br{
                \verts{\xi}+
                \verts{\eta}^{\frac{1}{2}}
            }^{-\alpha-2\beta},
        \end{equation*}
        Hence, 
        by \textit{Theorem \ref{thm_ani_twist_para}}, $T^L$ and \(T^M\) both extend to bounded operators $L^p \times L^q \to L^r$. 
        
        \subsection{High-frequency pieces} \ 
        
        For the high-frequency pieces, we again regroup the decomposition in the following way:
        \begin{equation*}
            T^H\br{f_1,f_2}=
            \sum_{\substack{
                k\in\Z^2\\
                :k_1 \vee k_2 >0 \\ 
                |k_1 - k_2| < 100
            }} 
                T^{\br{k}}\br{f_1,f_2}
        \end{equation*}
        where we define:
        \begin{equation*}
            T^{\br{k}}\br{f_1,f_2}:=
            \sum_{j\in\Z}
            T_j\br{\Delta^{\br{1}}_{j+k_1}f_1,\Delta^{\br{2}}_{2j+k_2} f_2}.
        \end{equation*}
        The goal is again to obtain bounds of the following form:
        \begin{equation}\label{eq_Tk_dk_p_bound}
            \Verts{T^{\br{k}}\br{f_1,f_2}}_{L^r}\underset{r,p,q}{\lesssim}2^{-\epsilon_r \verts{k}}\Verts{f_1}_{L^p}\Verts{f_2}_{L^q}
        \end{equation}
        Yet, drawing ideas from the discussion on the toy model, we know \eqref{eq_Tk_dk_p_bound} can be achieved by interpolating between the following two types of bounds:
        \begin{equation}\label{eq_Tk_dk_2_dec}
            \Verts{T^{\br{k}}\br{f_1,f_2}}_{L^1}\lesssim 2^{-\sigma \verts{k}} \Verts{f_1}_{L^2}\Verts{f_2}_{L^2}
        \end{equation}
        and
        \begin{equation}\label{eq_Tk_dk_p_grw}
            \Verts{T^{\br{k}}\br{f_1,f_2}}_{L^r}\underset{p,q}{\lesssim} \verts{k}^N \Verts{f_1}_{L^p}\Verts{f_2}_{L^q}.
        \end{equation}
        What's left is to derive \eqref{eq_Tk_dk_2_dec} and \eqref{eq_Tk_dk_p_grw}.
        Starting with \eqref{eq_Tk_dk_p_grw}, one may perform an analysis similar to that of \(H^{\br{k}}\). Via dilation symmetry
        \begin{equation}\label{eq_TH_j_dil_sym}
            T_j\br{\Delta^{\br{1}}_{j+k_1}f_1,\Delta^{\br{2}}_{2j+k_2}f_2}=\Dil^r_{2^{-j},2^{-2j}}T_0\br{\Delta^{\br{1}}_{k_1}\Dil^p_{2^j,2^{2j}}f_1,\Delta^{\br{2}}_{k_2}\Dil^q_{2^j,2^{2j}}f_2},
        \end{equation}
        the discussion reduces to the \(j=0\) case. A direct calculation shows that:
        \begin{equation*}
            T_0
            \br{
                \Delta^{\br{1}}_{k_1}f_1,
                \Delta^{\br{2}}_{k_2}f_2
            }\br{x,y}
            = 
            \int_{\R^2}
                f_1\br{x-s,y}
                f_2\br{x,y-t}
                K_k\br{s,t}
            dsdt,
        \end{equation*}
        where we recall
        \begin{equation*}
            \widehat{K_k}\br{\xi,\eta}
                := 
                m\br{\xi,\eta}
                \psi_{k_1}\br{\xi}
                \psi_{k_2}\br{\eta}.
        \end{equation*}
        At this stage, we can mimic what we've done with the toy model. We put the absolute value inside and estimate directly with the pointwise estimate on the kernel \eqref{eq_kernel_to_shifted_I} to obtain:
        \begin{equation}
            \begin{aligned}
                &
                \verts{T_0
                    \br{
                        \Delta^{\br{1}}_{k_1}f_1,
                        \Delta^{\br{2}}_{k_2}f_2
                    }\br{x,y}
                }\\
                \underset{N}{\lesssim} &
                \Verts{
                    \Verts{
                        \ang{z_1}^{-N}\ang{z_2}^{-N}\cdot
                        \mathcal{M}^{\br{1}}_{z_1-2^{k_1}t}f_1\br{x,y}\cdot
                        \mathcal{M}^{\br{2}}_{z_2-2^{k_2}t^2}f_2\br{x,y}
                    }_{\ell^1\br{z\in\Z^2}}
                }_{L^1\br{\verts{t}\sim 1}}.
            \end{aligned}
        \end{equation}
        Again, since \(\mathcal{M}_\sigma\) commutes with dilation, as an immediate consequence of relation \eqref{eq_TH_j_dil_sym}, we have the following estimate for free:
        \begin{equation}
            \begin{aligned}
                &
                \verts{T_j
                    \br{
                        \Delta^{\br{1}}_{j+k_1}f_1,
                        \Delta^{\br{2}}_{2j+k_2}f_2
                    }\br{x,y}
                }\\
                \underset{N}{\lesssim} &
                \Verts{
                    \Verts{
                        \ang{z_1}^{-N}\ang{z_2}^{-N}\cdot
                        \mathcal{M}^{\br{1}}_{z_1-2^{k_1}t}f_1\br{x,y}\cdot
                        \mathcal{M}^{\br{2}}_{z_2-2^{k_2}t^2}f_2\br{x,y}
                    }_{\ell^1\br{z\in\Z^2}}
                }_{L^1\br{\verts{t}\sim 1}}.
            \end{aligned}
        \end{equation}
        Yet, again, from our previous discussion about the toy model, we can preserve the frequency information by using the identity \(\psi=\psi\cdot\widetilde{\psi}\) and obtain:
        \begin{equation}
            \begin{aligned}
                &
                \verts{T_j
                    \br{
                        \Delta^{\br{1}}_{j+k_1}f_1,
                        \Delta^{\br{2}}_{2j+k_2}f_2
                    }\br{x,y}
                }\\
                \underset{N}{\lesssim} &
                \Verts{
                    \Verts{
                        \ang{z_1}^{-N}\ang{z_2}^{-N}\cdot
                        \mathcal{M}^{\br{1}}_{z_1-2^{k_1}t}
                        f^{\br{1}}_{j,k_1}\br{x,y}\cdot
                        \mathcal{M}^{\br{2}}_{z_2-2^{k_2}t^2}
                        f^{\br{2}}_{j,k_2}\br{x,y}
                    }_{\ell^1\br{z\in\Z^2}}
                }_{L^1\br{\verts{t}\sim 1}},
            \end{aligned}
        \end{equation}
        where we introduce the shorthand:
        \begin{equation*}
            \left\{
            \begin{aligned}
                f^{\br{1}}_{j,k_1}:= & \Delta^{\br{1}}_{j+k_1}f_1\\
                f^{\br{2}}_{j,k_2}:= &
                \Delta^{\br{2}}_{2j+k_2}f_2.
            \end{aligned}
            \right.
        \end{equation*}
        Summing over all the scale \(j\in\Z\), we get the following pointwise bound:
        \begin{equation*}
            \begin{aligned}
                &
                \verts{T^{\br{k}}\br{f_1,f_2}\br{x,y}}\\
                \underset{N}{\lesssim} & 
                \Verts{
                    \Verts{
                        \ang{z_1}^{-N}
                        \ang{z_2}^{-N}
                        \Verts{
                            \mathcal{M}_{z_1-2^{k_1}t}f^{\br{1}}_{j,k_1}\br{x,y}
                            \cdot
                            \mathcal{M}_{z_2-2^{k_2}t^2}f^{\br{2}}_{j,k_2}\br{x,y}
                        }_{\ell^1\br{j\in\Z}}
                    }_{\ell^1\br{z\in\Z^2}}
                }_{L^1\br{\verts{t}\sim 1}}.
            \end{aligned}
        \end{equation*}
        Applying Cauchy-Schwartz on \(\Verts{\cdot}_{\ell^1\br{j\in\Z}}\), we have:
        \begin{equation*}
            \verts{T^{\br{k}}\br{f_1,f_2}\br{x,y}}
            \underset{N}{\lesssim} 
            \Verts{
                \Verts{
                    \ang{z_1}^{-N}
                    \ang{z_2}^{-N}
                    \cdot 
                    \mathcal{S}_{t,z_1,k_1}f_1\br{x,y}
                    \cdot
                    \mathcal{S}_{t,z_2,k_2}f_2\br{x,y}
                }_{\ell^1\br{z\in\Z^2}}
            }_{L^1\br{\verts{t}\sim 1}},
        \end{equation*}
    where the two square functions are defined naturally as below:
    \begin{equation}
        \left\{
        \begin{aligned}
            \mathcal{S}_{t,z_1,k_1}f_1\br{x,y}:= &
            \Verts{
                \mathcal{M}_{z_1-2^{k_1}t}f^{\br{1}}_{j,k_1}\br{x,y}
            }_{\ell^2\br{j\in\Z}}\\
            \mathcal{S}_{t,z_2,k_2}f_2\br{x,y}:= &
            \Verts{
                \mathcal{M}_{z_2-2^{k_2}t^2}f^{\br{2}}_{j,k_2}\br{x,y}
            }_{\ell^2\br{j\in\Z}}.
        \end{aligned}
        \right.
    \end{equation}
    We now evaluate the \(L^r\) norm, use Minkowski's integral inequality to move the \(L^r\) norm to the innermost layer, and apply H\"{o}lder's inequality to get:
    \begin{equation*}
        \begin{aligned}
            \Verts{T^{\br{k}}\br{f_1,f_2}}_{L^r}
            \underset{N}{\lesssim} &
            \Verts{
                \Verts{
                    \ang{z_1}^{-N}
                    \ang{z_2}^{-N}
                    \cdot
                    \Verts{
                        \mathcal{S}_{t,z_1,k_1}f_1\cdot
                        \mathcal{S}_{t,z_2,k_2}f_2
                    }_{L^r}
                }_{\ell^1\br{z\in\Z^2}}
            }_{L^1\br{\verts{t}\sim 1}}\\
            \leq &
            \Verts{
                \Verts{
                    \ang{z_1}^{-N}
                    \ang{z_2}^{-N}
                    \cdot
                    \Verts{
                        \mathcal{S}_{t,z_1,k_1}f_1
                    }_{L^p}
                    \Verts{
                        \mathcal{S}_{t,z_2,k_2}f_2
                    }_{L^q}
                }_{\ell^1\br{z\in\Z^2}}
            }_{L^1\br{\verts{t}\sim 1}}.
        \end{aligned}
    \end{equation*}
    Lastly, we recall that the Fefferman-Stein inequality \eqref{eq_shifted_Fefferman-Stein_ineq} combined with standard square function estimate gives:
    \begin{equation*}
        \left\{
        \begin{aligned}
            \Verts{
                \mathcal{S}_{t,z_1,k_1}f_1
            }_{L^p}
            \underset{p}{\lesssim} & \log\br{2+\verts{z_1-2^{k_1}t}}^2
            \Verts{f_1}_{L^p}\\
            \Verts{
                \mathcal{S}_{t,z_2,k_2}f_2
            }_{L^q}
            \underset{q}{\lesssim} & \log\br{2+\verts{z_2-2^{k_2}t^2}}^2
            \Verts{f_2}_{L^q}
        \end{aligned}
        \right.
    \end{equation*}
    and thus,
    \begin{equation*}
        \begin{aligned}
            \Verts{T^{\br{k}}\br{f_1,f_2}}_{L^r}
            \underset{N,p,q}{\lesssim} &
            \bigg\Vert
                \sum_{z_1\in\Z}
                    \ang{z_1}^{-N}
                    \log\br{2+\verts{z_1-2^{k_1}t}}^2\\
            &\cdot
                \sum_{z_2\in\Z}
                    \ang{z_2}^{-N}
                    \log\br{2+\verts{z_2-2^{k_2}t^2}}^2
            \bigg\Vert_{L^1\br{\verts{t}\sim 1}}
            \Verts{f_1}_{L^p}\Verts{f_2}_{L^q}\\
            \lesssim &
            \Verts{
                \log\br{2+\verts{2^{k_1}t}}^2
                \log\br{2+\verts{2^{k_2}t^2}}^2
            }_{L^1\br{\verts{t}\sim 1}}
            \Verts{f_1}_{L^p}\Verts{f_2}_{L^q}\\
            \lesssim &
            \verts{k}^4
            \Verts{f_1}_{L^p}\Verts{f_2}_{L^q}.
        \end{aligned}
    \end{equation*}
    This proves the estimate \eqref{eq_Tk_dk_p_grw}. As for \eqref{eq_Tk_dk_2_dec}, we'll perform several further reductions. We first claim that \eqref{eq_Tk_dk_2_dec} can be derived from the following estimate
    \begin{equation}\label{eq_T_0_dec}
        \Verts{
            T_0
            \br{
                \Delta^{\br{1}}_{k_1}f_1,
                \Delta^{\br{2}}_{k_2}f_2
            }
        }_{L^1}
        \lesssim 
        2^{-\sigma\verts{k}}
        \Verts{f_1}_{L^2}
        \Verts{f_2}_{L^2}.
    \end{equation}
    Via dilation symmetry \eqref{eq_TH_j_dil_sym}, \eqref{eq_T_0_dec} implies a uniform estimate on the individual pieces:
    \begin{equation}\label{eq_T_jkuni_dec}
        \Verts{
            T_j
            \br{
                \Delta^{\br{1}}_{j+k_1}f_1,
                \Delta^{\br{2}}_{2j+k_2}f_2
            }
        }_{L^1}
        \lesssim 
        2^{-\sigma\verts{k}}
        \Verts{f_1}_{L^2}
        \Verts{f_2}_{L^2}.
    \end{equation}
    Yet, due to the frequency localization of the Littlewood-Paley pieces, we get 
    \begin{equation*}
        \Verts{
            T_j
            \br{
                \Delta^{\br{1}}_{j+k_1}f_1,
                \Delta^{\br{2}}_{2j+k_2}f_2
            }
        }_{L^1}
        \lesssim 
        2^{-\sigma\verts{k}}
        \Verts{\widetilde{\Delta}^{\br{1}}_{j+k_1}f_1}_{L^2}
        \Verts{\widetilde{\Delta}^{\br{2}}_{2j+k_2}f_2}_{L^2}
    \end{equation*}
    for free. We then sum over \(j\in\Z\) and apply Cauchy-Schwartz inequality to obtain \eqref{eq_Tk_dk_2_dec}:
    \begin{equation*}
        \begin{aligned}
            \Verts{T^{\br{k}}\br{f_1,f_2}}_{L^1}
            \lesssim &
            2^{-\sigma\verts{k}}
            \sum_{j\in\Z}
            \Verts{\widetilde{\Delta}^{\br{1}}_{j+k_1}f_1}_{L^2}
            \Verts{\widetilde{\Delta}^{\br{2}}_{2j+k_2}f_2}_{L^2}\\
            \lesssim & 2^{-\sigma\verts{k}}
            \Verts{f_1}_{L^2}
            \Verts{f_2}_{L^2}.
        \end{aligned}
    \end{equation*}
    Lastly, we claim that we only need to prove a localized version of \eqref{eq_T_0_dec} instead:
    \begin{equation}\label{eq_T_0loc_dec}
        \Verts{
            T_0
            \br{
                f_1,
                f_2
            }
        }_{L^1\br{I_0}}
        \lesssim 
        2^{-\sigma\verts{k}}
        \Verts{f_1}_{L^2}
        \Verts{f_2}_{L^2},
    \end{equation}
    where we write \(I_0:=\left[-\frac{1}{2},\frac{1}{2}\right]^2\) and assume that \(\Delta^{\br{l}}_{k_l}f_l=f_l\). Heuristically speaking, due to the physical localization of the kernel \(\frac{\psi\br{t}}{t}\), the physical localization gets carried to both of the input functions:
    \begin{equation*}
        \1_{I_0}
        T_0\br{f_1,f_2}
        =
        \1_{I_0}
        T_0\br{\varphi f_1,\varphi f_2}.
    \end{equation*}
    Thus, it is tempting to directly conclude that:
    \begin{equation}\label{eq_T_0loc_heu}
        \Verts{
            T_0
            \br{
                f_1,
                f_2
            }
        }_{L^1\br{I_0}}
        \lesssim 
        2^{-\sigma\verts{k}}
        \Verts{\varphi f_1}_{L^2}
        \Verts{\varphi f_2}_{L^2}
    \end{equation}
    Yet, in \eqref{eq_T_0loc_heu}, the physical localization ruins the frequency localization assumption:
    \begin{equation*}
        \Delta^{\br{l}}_{k_l}\br{\varphi f_l}
        \neq 
        \varphi f_l
        =
        \varphi \Delta^{\br{l}}_{k_l} f_l
        .
    \end{equation*}
    Naturally, one way to fix the above issue is to control the contribution from the commutator:
    \begin{equation*}
        \left[\varphi , \Delta^{\br{l}}_{k_l}\right] f_l :=
        \varphi \Delta^{\br{l}}_{k_l} f_l
        -
        \Delta^{\br{l}}_{k_l}\br{\varphi f_l}.
    \end{equation*}
    In other words, we shall write:
    \begin{equation*}
        \begin{aligned}
            &
            \1_{I_0}T_0\br{f_1,f_2} 
            = 
            \1_{I_0}T_0\br{\varphi \Delta^{\br{1}}_{k_1} f_1,\varphi \Delta^{\br{2}}_{k_2} f_2}\\
            = &
            \1_{I_0}T_0
            \br{
                \Delta^{\br{1}}_{k_1}\br{\varphi f_1},
                \Delta^{\br{2}}_{k_2}\br{\varphi f_2}
            }
            +
            \1_{I_0}T_0
            \br{ 
                \left[\varphi , \Delta^{\br{1}}_{k_1}\right] f_1,
                \left[\varphi , \Delta^{\br{2}}_{k_2}\right] f_2
            }\\
            + &
            \1_{I_0}T_0
            \br{
                \left[\varphi , \Delta^{\br{1}}_{k_1}\right] f_1,
                \Delta^{\br{2}}_{k_2}\br{\varphi f_2}
            }
            + 
            \1_{I_0}T_0
            \br{
                \Delta^{\br{1}}_{k_1}\br{ \varphi f_1},
                \left[\varphi , \Delta^{\br{2}}_{k_2}\right] f_2
            }
        \end{aligned}
    \end{equation*}
    and then pass the physical localization to all the commutator:
    \begin{equation}\label{eq_T_0loctel}
        \begin{aligned}
            &
            \1_{I_0}T_0\br{f_1,f_2}\\
            = &
            \1_{I_0}T_0
            \br{
                \Delta^{\br{1}}_{k_1}\br{\varphi f_1},
                \Delta^{\br{2}}_{k_2}\br{\varphi f_2}
            }
            +
            \1_{I_0}T_0
            \br{ 
                \varphi \left[\varphi , \Delta^{\br{1}}_{k_1}\right] f_1,
                \varphi \left[\varphi , \Delta^{\br{2}}_{k_2}\right] f_2
            }\\
            + &
            \1_{I_0}T_0
            \br{
                \varphi \left[\varphi , \Delta^{\br{1}}_{k_1}\right] f_1,
                \Delta^{\br{2}}_{k_2}\br{\varphi f_2}
            }
            + 
            \1_{I_0}T_0
            \br{
                \Delta^{\br{1}}_{k_1}\br{ \varphi f_1},
                \varphi \left[\varphi , \Delta^{\br{2}}_{k_2}\right] f_2
            }.
        \end{aligned}
    \end{equation}
    Using a trivial estimate:
    \begin{equation}
        \Verts{T_0\br{f,g}}_{L^1\br{I_0}}\leq \Verts{T_0\br{f,g}}_{L^2\br{I_0}}\lesssim \Verts{f}_{L^2}\Verts{g}_{L^2}
    \end{equation}
    and \eqref{eq_T_0loc_dec} on \eqref{eq_T_0loctel}, we obtain:
    \begin{equation*}
        \begin{aligned}
            \Verts{T_0\br{f_1,f_2}}_{L^1\br{I_0}}
            \lesssim & 2^{-\sigma\verts{k}}
            \Verts{
                \Delta^{\br{1}}_{k_1}\br{\varphi f_1}
            }_{L^2}
            \cdot
            \Verts{
                \Delta^{\br{2}}_{k_2}\br{\varphi f_2}
            }_{L^2}\\
            + & 
            \Verts{
                \varphi \left[\varphi , \Delta^{\br{1}}_{k_1}\right] f_1
            }_{L^2}
            \cdot
            \Verts{
                \varphi \left[\varphi , \Delta^{\br{2}}_{k_2}\right] f_2
            }_{L^2}\\
            + &
            \Verts{
                \varphi \left[\varphi , \Delta^{\br{1}}_{k_1}\right] f_1
            }_{L^2}
            \cdot
            \Verts{
                \Delta^{\br{2}}_{k_2}\br{\varphi f_2}
            }_{L^2}\\
            + &
            \Verts{
                \Delta^{\br{1}}_{k_1}\br{\varphi f_1}
            }_{L^2}
            \cdot
            \Verts{
                \varphi \left[\varphi , \Delta^{\br{2}}_{k_2}\right] f_2
            }_{L^2}
        \end{aligned}
    \end{equation*}
    Yet, an easy calculation shows that:
    \begin{equation}
        \begin{aligned}
            \Br{\phi,\Delta_{k}}f\br{x}= &
            \int_{\R}
                \br{
                    \phi\br{x}
                    -\phi\br{y}
                }
                \widecheck{\psi_k}\br{x-y}
                f\br{y}
            dy\\
            \implies
            \verts{\Br{\phi,\Delta_{k}}f}\br{x}\lesssim &
            \int_{\R}
                \Verts{\phi'}_{L^\infty}\cdot\verts{x-y}\cdot
                2^k\ang{2^k\br{x-y}}^{-N}
                \verts{f}\br{y}
            dy\\
            \lesssim & 
            2^{-k}\Verts{\phi'}_{L^\infty}
            \int_\R
                2^k\ang{2^k\br{x-y}}^{1-N}
                \verts{f}\br{y}
            dy\\
            \lesssim & 2^{-k} \Verts{\phi'}_{L^\infty} \mathcal{M}f\br{x},
        \end{aligned}
    \end{equation}
    for any \(\phi\in\mathcal{S}\br{\R}\). In other words, we have:
    \begin{equation}
        \verts{\Br{\varphi , \Delta^{\br{l}}_{k_l}} f_l}\lesssim 2^{-k_l} \mathcal{M}^{\br{l}}f_l
    \end{equation}
    and thus, a correction for \eqref{eq_T_0loc_heu}:
    \begin{equation}\label{eq_T_0loc_rig}
        \Verts{T_0\br{f_1,f_2}}_{L^1\br{I_0}}\lesssim 2^{-\sigma\verts{k}}\prod_{l=1,2}
        \br{
        \Verts{\varphi f_l}_{L^2}+
        \Verts{\varphi \mathcal{M}^{\br{l}}f_l}_{L^2}
        }.
    \end{equation}
    Next, making use of translation symmetry, \eqref{eq_T_0loc_rig} implies the corresponding estimate on any unit square \(Q\):
    \begin{equation}
        \Verts{
            T_0
            \br{
                f_1,
                f_2
            }
        }_{L^1\br{Q}}
        \lesssim 
        2^{-\sigma\verts{k}}
        \prod_{l=1,2}
            \br{
            \Verts{\varphi_Q f_l}_{L^2}+
            \Verts{\varphi_Q \mathcal{M}^{\br{l}}f_l}_{L^2}
            },
    \end{equation}
    where \(\varphi_Q\) is a bump function adapted to \(Q\). We now tiles the physical space \(\R^2\) with cubes of unit side-length and estimate:
    \begin{equation*}
        \Verts{T_0\br{f_1,f_2}}_{L^1}=
        \sum_{Q}
        \Verts{T_0\br{f_1,f_2}}_{L^1\br{Q}}
        \lesssim
        2^{-\sigma\verts{k}}
        \sum_{Q}
        \prod_{l=1,2}
            \br{
            \Verts{\varphi_Q f_l}_{L^2}+
            \Verts{\varphi_Q \mathcal{M}^{\br{l}}f_l}_{L^2}
            }.
    \end{equation*}
    Using the fact that the collection of support of \(\varphi_Q\)s only has finite overlaps, Cauchy-Schwartz inequality yield:
    \begin{equation}
        \Verts{T_0\br{f_1,f_2}}_{L^1}\lesssim 2^{-\sigma\verts{k}}\Verts{f_1}_{L^2}\Verts{f_2}_{L^2}.
    \end{equation}
    Finally, recall the assumptions that \(\Delta^{\br{l}}_{k_l}f_l=f_l\), we recover \eqref{eq_T_0_dec}. It remains to prove \eqref{eq_T_0loc_dec} under the corresponding frequency assumptions. We call such form of an estimate smoothing inequality.

\vspace{.5in}

\section{A trilinear smoothing inequality}

\subsection{Reduction} \ 

Recall the trilinear form in Theorem \ref{thm 5 - smoothing} given by
\begin{align*}
    \Lambda (f_1, f_2, f_3) = \int_{\R^3} f_1(x+t, y) f_2(x, y+t^2) f_3(x, y) \zeta(x, y, t) dx dy dt.
\end{align*}

The smoothing inequality in Theorem \ref{thm 5 - smoothing} is equivalent to 
\begin{align*}
    \lf \Lambda(f_1, f_2, f_3) \rf \leq C \norm{f_1}{H^{(-\sigma, 0)}} \norm{f_2}{H^{(0, -\sigma)}} \norm{f_3}{L^{\infty}}
\end{align*}
for $\lambda >1$, assuming that $\supp \wh{f}_l \subseteq \{|\xi_l| \sim \lb\}$ for at least one index $l=1, 2$. 

We dualize $|\Lambda (f_1, f_2, f_3)| = |\la T_{\loc} (f_1, f_2), f_3 \ra|$, where $T_{\loc}$ denotes the bilinear operator
\begin{equation}
    T_{\loc}(f_{1},f_{2})(x,y):=\int_{\mathbb{R}}f_{1}(x+t,y)f_{2}(x,y+t^{2})\zeta (x,y,t) dt,
\end{equation}
where $\zeta$ is a smooth test function on $\R^2 \times (0, \infty)$. Our goal is thus to prove the following \textit{bilinear} form estimate: there exists $\sigma>0$ such that
\begin{equation}\label{goal0}
    \| T_{\loc}(f_{1},f_{2})\|_{L^{1}}\lesssim \lambda^{-\sigma}\|f_{1}\|_{L^{2}}\|f_{2}\|_{L^{2}},
\end{equation}
where $\lb>1$ assuming that $\supp \wh{f}_l \subseteq \{|\xi_l| \sim \lb \}$ for at least one index $l = 1, 2$. Indeed, in the dual form, \eqref{goal0} implies we have the trilinear smoothing inequality
\begin{equation}
    |\Lambda(f_{1},f_{2},f_{3})|=|\langle  T_{\loc}(f_{1},f_{2}),f_{3} \rangle| \lesssim  \lambda^{-\sigma}\|f_{1}\|_{L^{2}}\|f_{2}\|_{L^{2}}\|f_{3}\|_{L^{\infty}}.
\end{equation}
By interpolation, it suffices to prove the following two estimates
\begin{equation}\label{trivial1}
     \| T_{\loc}(f_{1},f_{2})\|_{L^{1}}\lesssim \|f_{1}\|_{L^{\frac{3}{2}}}\|f_{2}\|_{L^{\frac{3}{2}}}
\end{equation}
\begin{equation}\label{goal1}
     \| T_{\loc}(f_{1},f_{2})\|_{L^{1}}\lesssim \lambda^{-\sigma}\|f_{1}\|_{L^{\infty}}\|f_{2}\|_{L^{\infty}}.
\end{equation}
We start with the proof of the first estimate which is more straightforward. 
\noindent\begin{proof}[Proof of \eqref{trivial1}]
\begin{equation}
    \begin{aligned}
         \| T_{\loc}(f_{1},f_{2})\|_{L^{1}}&= \int_{\mathbb{R}^{2}}|f_{1}(x,y)f_{2}(x+t,y+t^{2})\zeta (x+t,y,t)|dtdxdy\\
       &\leq\int_{\mathbb{R}^{2}}f_{1}(x,y)\left(\int_{\mathbb{R}}|f_{2}(x+t,y+t^{2})|\eta(t)dt  \right)dxdy \\
       &\leq \|f_{1}\|_{L^{\frac{3}{2}}}\|Tf_{2}\|_{L^{3}}\\
       &\lesssim \|f_{1}\|_{L^{\frac{3}{2}}}\|f_{2}\|_{L^{\frac{3}{2}}}
    \end{aligned}
\end{equation}
where we bound $|\zeta (x+t,y,t)|\leq \eta(t)$ and recall the result by Strichartz (see \cite{strichartz1970convolutions}) that averages along a parabola are bounded as maps on $L^{\frac{3}{2}} \rightarrow L^3$.
\end{proof}

The main difficulties arise in the proof of the second estimate \eqref{goal1}. We start by normalizing $\|f_{1}\|_{L^{\infty}}=\|f_{2}\|_{L^{\infty}}=1$. 

\subsection{Spatial localization} \

For our spatial localization, we cut the $\mathbb{R}^{2}$ plane into grid of squares of sidelength $\lambda^{-\gamma}$. That is, let $\eta \in C^{\infty}_c(\R^2)$ be s.t.
\begin{equation}
    1=\sum_{m\in \mathbb{Z}^{2}}\eta(\lambda^{\gamma}(x,y)-m).
\end{equation}
Then for $l=1,2$, we have
\begin{equation}
    f_{l}(x,y)=\sum_{m\in \mathbb{Z}^{2}}f_{l,m}(x,y)\quad with \quad f_{l,m}=f_{l}(x,y)\eta(\lambda^{\gamma}(x,y)-m).
\end{equation}
Let $Q_{m}$ denote the cube of sidelength $2\lambda^{-\gamma}$ centered at $\lambda^{-\gamma}m$ and let $\mathfrak{I}$ denote the set of all pairs $m=(m_{1},m_{2})\in (\mathbb{Z}^{2})^{2}$ such that
\begin{equation}
    \|T_{\loc}(f_{1,m_{1}},f_{2,m_{2}})\|_{L^{1}}\neq 0.
\end{equation}
By the compactness of $\supp \zeta$ and the composing map in $f_{1},f_{2}$, we may verify  $\#\mathfrak{I}=O(\lambda^{3\gamma})$. (Essentially, we have four variables which satisfy one relation so that we are left with $4-1=3$.)  By the triangle inequality,
\begin{equation}
   \|T_{\loc}(f_{1},f_{2})\|_{L^{1}}\leq \sum_{m\in \mathfrak{I}}\|T_{\loc}(f_{1,m_{1}},f_{2,m_{2}})\|_{L^{1}}.
\end{equation}
Next, we estimate each localized piece separately. By the Cauchy-Schwartz inequality,
\begin{equation}
    \begin{aligned}
        &\|T_{\loc}(f_{1,m_{1}},f_{2,m_{2}})\|_{L^{1}}^{2}\\
    =& \left(\int_{\mathbb{R}^{2}}|\int_{\mathbb{R}}f_{1,m_{1}}(x+t,y)f_{2,m_{2}}(x,y+t^{2})\zeta(x,y,t)dt|\cdot 1_{Q_{m}}(x,y) \; dxdy \right)^{2}\\
    \leq &\int_{\mathbb{R}^{2}}|\int_{\mathbb{R}}f_{1,m_{1}}(x+t,y)f_{2,m_{2}}(x,y+t^{2})\zeta(x,y,t)dt|^{2}dxdy \cdot \int_{\mathbb{R}^{2}}1_{Q_{m}}(x,y)dxdy\\
    \lesssim &\lambda^{-2\gamma}\int_{\mathbb{R}^{4}}f_{1,m_{1}}(x+t,y)f_{2,m_{2}}(x,y+t^{2})\zeta(x,y,t)\overline{f_{1,m_{1}}(x+s,y)}\overline{f_{2,m_{2}}(x,y+s^{2})}\overline{\zeta(x,y,s)}dtdsdxdy\\
    =&\lambda^{-2\gamma}\int_{\mathbb{R}^{4}}f_{1,m_{1}}(x+t,y)f_{2,m_{2}}(x,y+t^{2})\overline{f_{1,m_{1}}(x+(t+s),y)}\overline{f_{2,m_{2}}(x,y+(t+s)^{2})}\zeta_{m}(x,y,t,s)dxdydsdt,
    \end{aligned}
\end{equation}
where $\zeta_{m} \in C^{\infty}_c(\mathbb{R}^{2}\times (0,\infty)^{2})$ is a non-negative function supported in a cube with side length $O(\lambda^{-\gamma})$, so that for all $\ap$,
\begin{equation}
    \|\partial^{\alpha}\zeta_{m}\|_{C^{0}}\lesssim \lambda^{\gamma |\alpha|}.
\end{equation}

For each $m\in \mathfrak{I}$, we fix a point
\begin{equation}
    (\bar{x},\bar{y},\bar{t})=(\bar{x}_{m},\bar{y}_{m},\bar{t}_{m})\quad \text{such that}\quad (\bar{x}+\bar{t},\bar{y})\in Q_{m_{1}}, (\bar{x},\bar{y}+\bar{t}^{2})\in Q_{m_{2}}.
\end{equation}
Then for $(x,y,t,s)$ in the support of $\zeta_{m}$, we have
\begin{equation}
    |x-\bar{x}|+|y-\bar{y}|+|t-\bar{t}|\lesssim \lambda^{-\gamma}.
\end{equation}
Only when both $(x+t+s,y)$ and $(x+t,y)$ lie in the support of $f_{1,m_{1}}$, is $\|T_{\loc}(f_{1,m_{1}},f_{2,m_{2}})\|_{L^{1}}$ is nonzero. In this case, $|s|=|(x+t+s)-(x+t)|=O(\lambda^{-\gamma})$. By the mean value theorem, we have
\begin{equation}
    \begin{aligned}
        &f_{2,m_{2}}(x,y+(t+s)^{2})-f_{2,m_{2}}(x,y+t^{2}+2s\bar{t})\\
        \leq &\left((t+s)^{2}-(t^{2}+2s\bar{t})  \right)\|\partial_{2}f_{2,m_{2}}\|_{L^{\infty}}\\
        \lesssim &(2s(t-\bar{t})+s^{2})\cdot \lambda \lesssim \lambda^{-2\gamma+1}=\lambda^{-2\delta},
    \end{aligned}
\end{equation}
where we let $\delta=\gamma-\frac{1}{2}$. Notice that we use $\|\partial_{2}f_{2,m_{2}}\|_{L^{\infty}}\lesssim \lambda\|f_{2}\|_{L^{\infty}}$.

\subsection{Multiplicative derivatives} \ 

Define the multiplicative derivatives as follows:
\begin{equation}
    \mathcal{D}^{(1)}_{s}f(x,y):=f(x+s,y)\overline{f(x,y)},\quad \mathcal{D}^{(2)}_{s}f(x,y):=f(x,y+s)\overline{f(x,y)}.
\end{equation}
We then have
\begin{equation}\label{twoterm}
    \begin{aligned}
        &\|T_{\loc}(f_{1,m_{1}},f_{2,m_{2}})\|_{L^{1}}\\
    \lesssim &\textcolor{blue}{\lambda^{-\gamma}}\sum_{m\in J} \left|\textcolor{orange}{\int_{|s|\lesssim \lambda^{-\gamma}}\left(\int_{\mathbb{R}^{3}}\mathcal{D}^{(1)}_{s}f_{1,m_{1}}(x+t,y)\mathcal{D}^{(2)}_{2s\bar{t}}f_{2,m_{2}}(x,y+t^{2})\zeta_{m}(x,y,t,s)dxdydt  \right)ds}\right|^{\frac{1}{2}}\\
    +&\lambda^{-\gamma}\sum_{m\in J}\left(\int_{|s|\lesssim \lambda^{-\gamma}}\left(\int_{\mathbb{R}^{3}}\lambda^{-2\delta}\zeta_{m}(x,y,t,s)dxdydt \right) ds  \right)^{\frac{1}{2}}
    \end{aligned}
\end{equation}
where we bounded $f_{1}$ in the second term by its $L^{\infty}$ norm.
In addition, the second term can be bounded by
\begin{equation}
    \lambda^{-\gamma}\cdot \lambda^{3\gamma}\cdot(\lambda^{-\gamma}\cdot \lambda^{-2\delta}\cdot \lambda^{-3\gamma})^{\frac{1}{2}}=\lambda^{-\delta}.
\end{equation}
Now we estimate the first term in \eqref{twoterm}. We want to take the square outside. So we apply Cauchy-Schwartz in the following way
\begin{equation}
    \textcolor{blue}{\lambda^{-\gamma}}(\sum_{m\in J}|\textcolor{orange}{a_{m}}|^{\frac{1}{2}}\cdot 1)\leq \lambda^{-\gamma}(\sum_{m\in J}\textcolor{orange}{a_{m}})^{\frac{1}{2}}\cdot (\sum_{m\in J}1)^{\frac{1}{2}}\lesssim\lambda^{-\gamma}\cdot \lambda^{\frac{3}{2}\gamma}(\sum_{m\in J}\textcolor{orange}{a_{m}})^{\frac{1}{2}}=\lambda^{\frac{\gamma}{2}}(\sum_{m\in J}\textcolor{orange}{a_{m}})^{\frac{1}{2}}
\end{equation}
Hence we have the first term in \eqref{twoterm} is majorized by 
\begin{equation}\label{term1}
    \lambda^{\frac{\gamma}{2}}\left(\textcolor{orange}{\int_{|s|\lesssim \lambda^{-\gamma}}} \sum_{m\in J}\textcolor{orange}{\left|\int_{\mathbb{R}^{3}}\mathcal{D}^{(1)}_{s}f_{1,m_{1}}(x+t,y)\mathcal{D}^{(2)}_{2s\bar{t}}f_{2,m_{2}}(x,y+t^{2})\zeta_{m}(x,y,t,s)dxdydt \right|ds}\right)^{\frac{1}{2}}
\end{equation}
Let $s$ be s.t. $|s|=O(\lambda^{-\gamma})$, the function $\mathcal{D}_{s}^{(l)}f_{l,m_{l}}$ is in a cube $Q_{m_{l}}$ of side length $\lambda^{-\gamma}$. We may do a Fourier series expansion on that cube.
\begin{equation}
    \mathcal{D}_{s}^{(l)}f_{l,m_{l}}(x,y)=\eta_{m_{l}}(x,y)\sum_{k\in \mathbb{Z}^{2}}a_{l,m_{l},k,s}e^{2\pi i\lambda^{\gamma}k\cdot (x,y)}
\end{equation}
where $\eta_{m_{l}}(x,y)$ allows us to record the spatial localization around $Q_{m_{l}}$. The Fourier coefficients are given by
\begin{equation}
    a_{l,m_{l},k,s}=\lambda^{2\gamma}\widehat{ \mathcal{D}_{s}^{(l)}f_{l,m_{l}}}(\lambda^{\gamma}k)
\end{equation}
We then have
\begin{equation}\label{trivialcoeffes}
\begin{aligned}
    \sum_{k\in \mathbb{Z}^{2}}|a_{l,m_{l},k,s}|^{2}&= \sum_{k\in \mathbb{Z}^{2}}\lambda^{4\gamma}|\widehat{ \mathcal{D}_{s}^{(l)}f_{l,m_{l}}}(\lambda^{\gamma}k)|^{2}\\
    &=\lambda^{2\gamma}\sum_{k\in \mathbb{Z}^{2}}\lambda^{2\gamma}|\widehat{ \mathcal{D}_{s}^{(l)}f_{l,m_{l}}}(\lambda^{\gamma}k)|^{2}\\
    &=\lambda^{2\gamma}\sum_{k\in \mathbb{Z}^{2}}|\widehat{ \mathcal{D}_{s}^{(l)}f_{l,m_{l}}}(k)|^{2}\\
    &=\lambda^{2\gamma}\| \mathcal{D}_{s}^{(l)}f_{l,m_{l}}(k)\|_{L^{2}}^{2}\\
    &\lesssim \lambda^{2\gamma}\|f\|_{L^{\infty}}^{4}\cdot |Q_{m}|\lesssim 1
\end{aligned}
\end{equation}

On the other hand, since $f_{2}$ is such that $\wh{f}_2(\xi_1, \xi_2) \neq 0$ for $|\xi_2| \leq 2\lambda$, we can expect a stronger bound holds when $k$ is larger in an appropriate sense.
Later, me will use the following version of Bernstein's lemma.
\begin{lemma}[Bernstein]
    Let $B$ be a unit ball in $\mathbb{R}^{d}$. Suppose that $\widehat{f}\subseteq \lambda B$, then we have
\begin{equation}
    \|D^{k}f\|_{L^{q}}:=\underset{|\alpha|=k}{\operatorname{sup}}\|\partial^{\alpha}f\|_{L^{q}}\lesssim \lambda^{k+d(\frac{1}{p}-\frac{1}{q})}\|f\|_{L^{p}}
\end{equation}
for $p,q\in [1,\infty]$ with $1\leq p\leq q$.
\end{lemma}
In our case, we only use that $d=1,\: p=q=\infty$.
In the following, we integrate by parts in the first variable twice and the second variable $N$ times.

\begin{equation}
\begin{aligned}
    |a_{l,m_{l},k,s}|&=\left|\lambda^{2\gamma}\cdot \int_{\mathbb{R}^{2}} \mathcal{D}_{s}^{(l)}f_{l,m_{l}}(x,y)e^{2\pi i\lambda^{\gamma}k\cdot (x,y)}dxdy\right|\\
    &=\left|  \int_{\mathbb{R}^{2}} \mathcal{D}_{s}^{(l)}f_{l,m_{l}}(\frac{x}{\lambda^{\gamma}},\frac{y}{\lambda^{\gamma}})e^{2\pi ik\cdot (x,y)}dxdy  \right|\\
    &=|k_{1}|^{-2}|k_{2}|^{-N}(\lambda^{\gamma})^{-N-2}\int_{\mathbb{R}^{2}} (\partial^{2}_{1}\partial^{N}_{2} \mathcal{D}_{s}^{(l)}f_{l,m_{l}})(\frac{x}{\lambda^{\gamma}},\frac{y}{\lambda^{\gamma}})e^{2\pi ik\cdot (x,y)}dxdy\\
    &\lesssim |k_{1}|^{-2}|k_{2}|^{-N}(\lambda^{\gamma})^{-N-2}\|\partial^{2}_{1}\partial^{N}_{2} \mathcal{D}_{s}^{(l)}f_{l,m_{l}}\|_{L^{\infty}}\\
    &\lesssim |k_{1}|^{-2}|k_{2}|^{-N}(\lambda^{\gamma})^{-N-2}\cdot \lambda^{N}\\
    &=|k_{1}|^{-2}|k_{2}|^{-N}\lambda^{(1-\gamma-\frac{2}{N})\cdot N}
\end{aligned}
\end{equation}
where in the fourth line, notice that the support of the integral is roughly one so that we can bound it by the $L^{\infty}$ norm of the integrand. In the second to last inequality, notice that $\operatorname{supp}\mathcal{F}(f)\sim \operatorname{supp}\mathcal{F}(T_{-s}f_{2,m_{2}}\cdot f_{2,m_{2}}) $ are contained in the ball centered at $0$ and of radius at the scale $\lambda$. Then we have
\begin{equation}
\begin{aligned}
    \|\partial^{N}_{2} \mathcal{D}_{s}^{(l)}f_{l,m_{l}}\|_{L^{\infty}_{2}}&=\left\|\sum_{\alpha_{1}+\alpha_{2}=N}\partial^{\alpha_{1}}_{2}(T_{-s}\eta_{m_{2}}\cdot \eta_{m_{2}})\cdot \partial^{\alpha_{2}}_{2}(T_{-s}^{(2)}f_{2,m_{2}}f_{2,m_{2}})\right\|_{L^{\infty}_{2}}\\
    &\lesssim \sum_{\alpha_{1}+\alpha_{2}=N} \lambda^{\alpha_{1}}\cdot \lambda^{\alpha_{2}}\\
    &\lesssim_{N}\lambda^{N}
\end{aligned}
\end{equation}
Hence, we have
\begin{equation}
    \sum_{k_{1}\in \mathbb{Z}}|a_{l,m_{l},k,s}|^{2}\lesssim |k_{2}|^{-N}\lambda^{(1-\gamma-\frac{2}{N})\cdot N}
\end{equation}
   Fix a small $\varepsilon_{1}>0$, then 
\begin{equation}
    \sum_{k_{2}\geq \lambda^{1-\gamma+\varepsilon_{1}}}\sum_{k_{1}\in \mathbb{Z}}|a_{l,m_{l},k,s}|^{2}\lesssim \lambda^{(1-\gamma+\varepsilon_{1})(-N)}\cdot \lambda^{(1-\gamma-\frac{2}{N})\cdot N}=\lambda^{-N\varepsilon_{1}-2}
\end{equation}
For a fixed $\varepsilon_1$, we can choose $N$ large enough so that $N\varepsilon_1+2$ can be arbitrarily large. As such, we obtain an arbitrary polynomial decay in $\lambda$.
\begin{equation}
     \sum_{k_{2}\geq \lambda^{1-\gamma+\varepsilon_{1}}}\sum_{k_{1}\in \mathbb{Z}}|a_{l,m_{l},k,s}|^{2}\lesssim_{N,\varepsilon_{1}}\lambda^{-N}
\end{equation}
Hence, we can estimate \eqref{term1} by
\begin{equation}\label{k22small}
    O(\lambda^{-N})+\sum_{\substack{k\in (\mathbb{Z}^{2})^{2}\\ |k_{2,2}|\leq \lambda^{1-\gamma+\varepsilon_{1}}}}|a_{1,m_{1},k_{1},s}a_{2,m_{2},k_{2},2s\bar{t}}|\left| \int_{\mathbb{R}^{3}}e^{2\pi i \lambda^{\gamma}(k_{1}\cdot (x+t,y)+k_{2}\cdot (x,y+t^{2}))}\zeta_{m}(x,y,t,s)dxdydt  \right|
\end{equation}
The phase is given by
\begin{equation}
    \lambda^{\gamma}(k_{1,1}+k_{2,1})x+(k_{1,2}+k_{2,2})y+k_{1,1}t+k_{2,2}t^{2}
\end{equation}
The gradient of the phase is 
\begin{equation}
    \lambda^{\gamma}\left(k_{1,1}+k_{2,1},k_{1,2}+k_{2,2},k_{1,1}+2tk_{2,2} \right)=\lambda^{\gamma}\left(k_{1}+k_{2},k_{1,1}+2\bar{t}k_{2,2} \right)+O(\lambda^{1-\gamma +\varepsilon})
\end{equation}
where we use $|t-\bar{t}|\lesssim \lambda^{-\gamma}$ and $|k_{2,2}|\lesssim \lambda^{1-\gamma +\varepsilon}$. We get the desired decay of $O(\lambda^{1-\gamma +\varepsilon})$. After integrating by parts, the oscillatory integral decays fast enough when the gradient of the phase is nonzero. That is,
\begin{equation}
   \left| \int_{\mathbb{R}^{3}}e^{2\pi i \lambda^{\gamma}(k_{1}\cdot (x+t,y)+k_{2}\cdot (x,y+t^{2}))}\zeta_{m}(x,y,t,s)dxdydt  \right|\lesssim_{N}\lambda^{-N}
\end{equation}
unless 
\begin{equation}\label{krelation}
    |k_{1,1}+k_{2,1}|\lesssim \lambda^{\varepsilon_{2}}, \; |k_{1,2}+k_{2,2}|\lesssim \lambda^{\varepsilon_{2}}\; (|k_{1}+k_{2}|\lesssim \lambda^{\varepsilon_{2}}),\quad |k_{1,1}+2\bar{t}k_{2,2}|\lesssim \lambda^{\varepsilon_{2}}
\end{equation}
Since $\supp \zeta_{m}$ is contained in a cube of side length $\lambda^{-\gamma}$, after pushing the modulus inside \eqref{k22small}, for those $k$ which satisfy \eqref{krelation} (that is, those  $k$ for which we cannot apply the cancellation in the phase), we apply the trivial bound and obtain
\begin{equation}\label{reducecoeff}
    O(\lambda^{-3\gamma})\sum_{\substack{k\in (\mathbb{Z}^{2})^{2}\\ \eqref{krelation}\: \textrm{holds}}}|a_{1,m_{1},k_{1},s}a_{2,m_{2},k_{2},2s\bar{t}}|
\end{equation}
By Cauchy-Schwarz and the trivial estimate \eqref{trivialcoeffes}, we have have two estimates of \eqref{reducecoeff}
\begin{equation}
\begin{aligned}
    \eqref{reducecoeff}&\lesssim \left(\sum_{\substack{k\in (\mathbb{Z}^{2})^{2}\\ \eqref{krelation}\: \textrm{holds}}}|a_{1,m_{1},k_{1},s}|^{2} \right)^{\frac{1}{2}}\cdot \left(\sum_{\substack{k\in (\mathbb{Z}^{2})^{2}\\ \eqref{krelation}\: \textrm{holds}}}|a_{1,m_{1},k_{1},2s\bar{t}}|^{2} \right)^{\frac{1}{2}}\\
    &\lesssim \lambda^{-3\gamma+2\varepsilon_{2}} \left(\sum_{\substack{k_{1}\in \mathbb{Z}^{2}\\ |2\bar{t}k_{1,2}-k_{1,1}|\lesssim \lambda^{\varepsilon_{2}}}}|a_{1,m_{1},k_{1},s}|^{2} \right)^{\frac{1}{2}}
\end{aligned}
\end{equation}
This is by estimating the second term trivially by $(\lambda^{4\varepsilon_{2}}\cdot 1^{2})^{\frac{1}{2}}$ and 
\begin{equation}
    |2\bar{t}k_{1,2}-k_{1,1}|\lesssim |2\bar{t}k_{1,2}+2\bar{t}k_{2,2}|+|2\bar{t}k_{2,2}+k_{1,1}|\lesssim \lambda^{\varepsilon_{2}}
\end{equation}
Similarly, using the trivial estimate, we get
\begin{equation}
     \eqref{reducecoeff}\lesssim \lambda^{-3\gamma+2\varepsilon_{2}} \left(\sum_{\substack{k_{1}\in \mathbb{Z}^{2}\\ |2\bar{t}k_{2,2}-k_{2,1}|\lesssim \lambda^{\varepsilon_{2}}}}|a_{2,m_{2},k_{2},2s\bar{t}}|^{2} \right)^{\frac{1}{2}}
\end{equation}
Plugging in the above estimate in \eqref{term1}, and by Cauchy-Schwartz we can further estimate \eqref{twoterm} 
\begin{equation}
    \|T_{\loc}(f_{1},f_{2})\|_{L^{1}}\lesssim \lambda^{-\delta}+\lambda^{-\frac{\gamma}{2}+\varepsilon_{2}}\left( \int_{|s|=O(\lambda^{-\gamma})}\sum_{m\in J}\sum_{\substack{k_{l}\in \mathbb{Z}^{2}\\ |2\bar{t}_{m}k_{l,2}-k_{l,1}|\lesssim \lambda^{\varepsilon_{2}}}}|a_{l,m_{l},k_{l},s}|^{2}ds \right)^{\frac{1}{4}}
\end{equation}

Note that for $l=2$, a change of variable $2s\bar{t}_{m}\rightarrow s$ only cause a Jacobian of O(1). For the moment, suppose $l=1$. Write $m=(m_{1},m_{2})$ and holding $m_{1}$ fixed, there are only $O(\lambda^{\gamma})$ of $m_{2}$ so that $(m_{1},m_{2})\in J$. On the other hand, we have a better bound if $k_{1,1}$ is large enough. To better explain this remark, hold $(k_{1,1},k_{1,2})\in \mathbb{Z}^{2}$ fixed, we want to count how many $m_{2}$'s are such that $|2\bar{t}_{m}k_{1,2}-k_{1,1}|\lesssim \lambda^{\varepsilon_{2}}$. This is equivalent to counting $m_{2}$ such that $|2k_{1,2}-\bar{t}_{m}^{-1}k_{1,1}|\lesssim \lambda^{\varepsilon_{2}}$ since $\bar{t}$ is bounded from above and below by a positive constant. Note that we want $(x_{m}+t_{m},y_{m})\in Q_{m_{1}}$ and $(x_{m},y_{m}+t_{m}^{2})\in Q_{m_{2}}$. As such, when $m_{1}$ is fixed, we take $x_{m}$ as our parameter. To figure out how many $m_{2}$ are such that $(x_{m},y_{m}+(m_{1,1}-x_{m})^{2})\in Q_{m_{2}}$, (Now $t_{m}=m_{1,1}-x_{m}$) notice that 
\begin{equation}
    t_{m}^{-1}-t_{m'}^{-1}=\frac{1}{m_{1,1}-x_{m}}-\frac{1}{m_{1,1}-x_{m'}}=\frac{x_{m}-x_{m'}}{(m_{1,1}-x_{m})(m_{1,1}-x_{m'})}=\frac{\lambda^{-\delta}}{t_{m}t_{m'}}=\lambda^{-\delta}
\end{equation}
where the distance between $m$ and $m'$ is around 1 in the lattice. That is, moving $m$ by a unit will cause a $\lambda^{-\gamma}|k_{1,1}|$ to change in $|2k_{1,2}-\bar{t}_{m}^{-1}k_{1,1}|$. We can afford at most $\lambda^{\varepsilon_{2}}$ change, that is there are at most $1+\frac{\lambda^{\varepsilon_{2}}}{\lambda^{-\gamma}|k_{1,1}|}=1+\lambda^{\gamma +\varepsilon_{2}}|k_{1,1}|^{-1}$ such $m_{2}$'s. Now we take another auxiliary parameter $\kappa \in (\textcolor{red}{5}\varepsilon_{2},1-\gamma)$. Then for each fixed $s$ and $l\in \{1,2\}$, we have
\begin{equation}
    \sum_{m\in J}\sum_{\substack{k_{l}\in \mathbb{Z}^{2}\\ |2\bar{t}_{m}k_{l,2}-k_{l,1}|\lesssim \lambda^{\varepsilon_{2}}}}|a_{l,m_{l},k_{l},s}|^{2}\lesssim \lambda^{\gamma}\sum_{m_{\textcolor{red}{l}}\in J_{\textcolor{red}{l}}}\sum_{\substack{k_{l}\in \mathbb{Z}^{2}\\ |k_{l,l}|\leq \lambda^{\kappa}}}|a_{l,m_{l},k_{l},s}|^{2}+R_{s,l}
\end{equation}
where 
\begin{equation}
    R_{s,l}=(1+\lambda^{\gamma+\varepsilon_{2}}-\kappa)\sum_{m_{l}\in J_{l}}\sum_{k_{l}\in \mathbb{Z}^{2}}|a_{l,m_{l},k_{l},s}|^{2}=(1+\lambda^{\gamma+\varepsilon_{2}}-\kappa)\cdot \lambda^{2\gamma}\lesssim \lambda^{3\gamma+\varepsilon_{2}-\kappa}
\end{equation}

The reason why we can drop the $1$ term in the first bracket is because
\begin{equation}
    \lambda^{3\gamma+\varepsilon_{2}-\kappa}\geq \lambda^{2\gamma} \Leftrightarrow \gamma+\varepsilon>\kappa
\end{equation}

This is true becuase $\gamma>\frac{1}{2}$ and $\kappa<1-\gamma$. Then we have
\begin{equation}
\begin{aligned}
    \|T_{\loc}(f_{1},f_{2})\|_{L^{1}}&\lesssim \lambda^{\frac{1}{2}-\gamma}+\lambda^{-\frac{\gamma}{2}+\varepsilon_{2}}\left(\int_{|s|=O(\lambda^{-\gamma})}\left( \lambda^{\gamma}\sum_{m_{l}\in J_{l}}\sum_{\substack{k_{l}\in \mathbb{Z}^{2}\\ |k_{l,l}|\leq \lambda^{\kappa}}}|a_{l,m_{l},k_{l},s}|^{2}+ \lambda^{3\gamma+\varepsilon_{2}-\kappa}\right) ds \right)^{\frac{1}{4}}\\
    &\lesssim \lambda^{\frac{1}{2}-\gamma}+\lambda^{(-\frac{\gamma}{2}+\varepsilon_{2})+(\frac{2\gamma+\varepsilon_{2}-\kappa}{4})}\\
    &+\lambda^{(-\frac{\gamma}{2}+\varepsilon_{2})+\frac{\gamma}{4}}\left(\int_{|s|=O(\lambda^{-\gamma})} \sum_{m_{l}\in J_{l}}\sum_{\substack{k_{l}\in \mathbb{Z}^{2}\\ |k_{l,l}|\leq \lambda^{\kappa}}}|a_{l,m_{l},k_{l},s}|^{2}  ds \right)^{\frac{1}{4}}\\
    &=\lambda^{\frac{1}{2}-\gamma}+\lambda^{-\frac{\kappa}{4}+\frac{5}{4}\varepsilon_{2}}+\lambda^{-\frac{\lambda}{4}+\varepsilon_{2}}\left(\int_{|s|=O(\lambda^{-\gamma})} \sum_{m_{l}\in J_{l}}\sum_{\substack{k_{l}\in \mathbb{Z}^{2}\\ |k_{l,l}|\leq \lambda^{\kappa}}}|a_{l,m_{l},k_{l},s}|^{2}  ds \right)^{\frac{1}{4}}
\end{aligned}
\end{equation}
Note that we have the following trivial estimate
\begin{equation}
   \left(\int_{|s|=O(\lambda^{-\gamma})} \sum_{m_{l}\in J_{l}}\sum_{\substack{k_{l}\in \mathbb{Z}^{2}\\ |k_{l,l}|\leq \lambda^{\kappa}}}|a_{l,m_{l},k_{l},s}|^{2}  ds \right)^{\frac{1}{4}}\lesssim (\lambda^{2\gamma}\cdot 1\cdot \lambda^{-\gamma})^{\frac{1}{4}}=\lambda^{\frac{\gamma}{4}}
\end{equation}
Yet, to win a decay, we need the following frequency pruning lemma.

\subsection{Frequency Prunning: Structural Decomposition}
\begin{lemma}\label{freqprunlem1}
    Define $\mathcal{D}_{s}f(x):=(\Tr_{-s}f\cdot \bar{f})(x)=f(x+s)\overline{f(x)}$. Let $f\in L^{2}(\mathbb{R}^{d})$, $\rho \in (0,1)$. Suppose that
    \begin{equation}\label{assumpt1}
        \int_{\mathbb{R}^{d}}\int_{|\xi|\leq R}|\widehat{\mathcal{D}_{s}f}|^{2}(\xi)d\xi ds\leq \rho\|f\|_{L^{2}}^{4}
    \end{equation}
Then there is an orthogonal decomposition $f=g+h$ with $\widehat{g}$ supported in some ball of radius $R$, $g\perp h$ and $\|g\|_{L^{2}}\geq \rho^{\frac{1}{2}}\|f\|_{L^{2}}$.
\end{lemma}
\begin{proof}
    \begin{equation}
        (\widehat{\mathcal{D}_{s}f})(\xi)=(\widehat{\Tr_{-s}f\cdot f})(\xi)=\left( (\Mod_{s}\widehat{f})\ast \widehat{f} \right)(\xi)=\int_{\mathbb{R}^{d}}e^{2\pi is(\xi+\xi')}\widehat{f}(\xi+\xi')\bar{\widehat{f}}(\xi')d\xi'
    \end{equation}
Hence, we have
\begin{equation}
    \begin{aligned}
        \int_{\mathbb{R}^{d}}|\widehat{\mathcal{D}_{s}f}(\xi)|^{2}ds&= \int_{\mathbb{R}^{d}}\widehat{\mathcal{D}_{s}f}(\xi)\overline{\widehat{\mathcal{D}_{s}f}}(\xi)ds\\
        &= \int_{\mathbb{R}^{3d}}e^{2\pi is\left((\xi+\xi')-(\xi+\xi'')\right)}\widehat{f}(\xi+\xi')\bar{\widehat{f}}(\xi')\bar{\widehat{f}}(\xi+\xi'')\widehat{f}(\xi'')d\xi'd\xi''ds\\
        &=\int_{\mathbb{R}^{2d}}\left(\int_{\mathbb{R}^{d}}e^{2\pi is(\xi'-\xi'')}ds \right)\widehat{f}(\xi+\xi')\bar{\widehat{f}}(\xi')\bar{\widehat{f}}(\xi+\xi'')\widehat{f}(\xi'')d\xi'd\xi''\\
        &=\int_{\mathbb{R}^{2d}}\delta(\xi'-\xi'')\widehat{f}(\xi+\xi')\bar{\widehat{f}}(\xi')\bar{\widehat{f}}(\xi+\xi'')\widehat{f}(\xi'')d\xi'd\xi''\\
        &=\int_{\mathbb{R}^{d}}|\widehat{f}(\xi+\xi')|^{2}|\widehat{f}(\xi')|^{2}d\xi'
    \end{aligned}
\end{equation}
Then  by \eqref{assumpt1}, we have
\begin{equation}
    \begin{aligned}
        \rho\|f\|_{L^{2}}^{4}\leq& \int_{\mathbb{R}^{d}}\int_{|\xi|<R}|\widehat{\mathcal{D}_{s}f}(\xi)|^{2}d\xi ds\\
        =&\int_{|\xi-\xi'|<R}\widehat{f}(\xi)|^{2}\widehat{f}(\xi')|^{2}d\xi d\xi'\\
        \leq &\|f\|_{L^{2}}^{2}\cdot \underset{B}{\operatorname{sup}}\int_{B}|\widehat{f}|^{2}\\
        \leq  &\|f\|_{L^{2}}^{2}\|\widehat{g}\|_{L^{2}}^{2}= \|f\|_{L^{2}}^{2}\|g\|_{L^{2}}^{2}
    \end{aligned}
\end{equation}
where we take $B$ to be the ball essentially realize the supremum and $\widehat{g}:=1_{B}\widehat{f}$ and $h=f-g$. Support property of $g$ and orthogonality of $h$ and $g$ is by construction. And by above estimate, we also have
\begin{equation}
    \|g\|_{L^{2}}\geq \rho^{\frac{1}{2}}\|f\|_{L^{2}}
\end{equation}
\end{proof}

The following lemma allows us to decompose a function $f$ into two parts $f = f_{\sharp} + f_{\flat}$: the part where $\wh{\mathcal{D}_s f}$ has a large $L^2$ norm, $f_{\sharp}$, and the part where $\wh{\mathcal{D}_s f}$ has a small $L^2$ norm, $f_{\flat}$.

\begin{lemma}\label{freqprunlem2}
    Let $R\geq 1$ and $\rho\in (0,1)$. Let $f\in L^{\mathbb{R}}$. There exist a decomposition

\begin{equation}
    f=f_{\sharp}+f_{\flat}
\end{equation}
with the following properties.

(1) One has
\begin{equation}
    \|f_{\sharp}\|_{L^{2}}+\|f_{\flat}\|_{L^{2}}\lesssim \|f\|_{L^{2}}
\end{equation}

(2) The function $f_{\sharp}$ admits a decomposition
\begin{equation}\label{fsharpdecomps}
    f_{\sharp}(x)=\sum_{n=1}^{N}h_{n}(x)e^{2\pi i\alpha_{n}x}
\end{equation}
with each $\alpha_{n}\in \mathbb{R}$, $h_{n}$ is a smooth function satisfying $\|\partial^{\alpha}h_{n}\|\lesssim_{\alpha}R^{\alpha}\|f\|_{L^{\infty}}$ for all integer $\alpha \geq 0$ and $\|h_{n}\|_{L^{2}}\lesssim \|f\|_{L^{2}}$, $h_{n}$ is Fourier support in $[-R,R]$, and $N\lesssim \rho^{-1}$. Moreover, the support of $\widehat{f_{\sharp}}$ is contained in the support of $\widehat{f}$.

(3) One has the bound
\begin{equation}\label{fflatprop}
     \int_{\mathbb{R}^{d}}|\widehat{\mathcal{D}_{s}f_{\flat}(\xi)}|^{2}ds\lesssim \rho\|f\|_{L^{2}}^{4}
\end{equation}
\end{lemma}

\begin{proof}
    We want to decompose $\mathbb{R}$ into intervals of length $R$. To do so, consider a smooth test function $\varphi$ supported in $(-1,1)$ such that
\begin{equation}
    \sum_{n\in \mathbb{Z}}\varphi(R^{-1}\xi +n)=\sum_{n\in \mathbb{Z}}\varphi_{n}(R^{-1}\xi)=1
\end{equation}
Now we want to pick out those interval where $\widehat{f}$ has large local $L^{2}$ norm. Let $\mathfrak{R}$ be the set of all $n\in \mathbb{Z}$ for which there exists an interval $I$ of length $R$ intersecting $[(-n-1)R,(-n+1)R]$ so that 
\begin{equation}
    \int_{I}|\widehat{f}(\xi)|^{2}d\xi\geq \rho \|f\|_{L^{2}}^{2}
\end{equation}
Note that $\sharp \mathfrak{R}=O(\rho^{-1})$. If not, 
\begin{equation}
    \|f\|_{L^{2}}= \|\widehat{f}\|_{L^{2}}\geq \sum_{n\in \mathfrak{R}} \int_{I_{n}}|\widehat{f}(\xi)|^{2}d\xi >\rho^{-1} \cdot \rho\|f\|_{L^{2}}^{2}
\end{equation}
a contradiction.
We then define $f_{\sharp}$, $f_{\flat}$ by
\begin{equation}
    \widehat{f}_{\sharp}(\xi):=\sum_{n\in \mathfrak{R}}\varphi_{n}(R^{-1}\xi)\widehat{f}(\xi),\quad \widehat{f}_{\flat}(\xi):=\widehat{f}(\xi)-\widehat{f}_{\sharp}(\xi)=\sum_{n\in \mathbb{Z}\setminus \mathfrak{R}}\varphi_{n}(R^{-1}\xi)\widehat{f}(\xi)
\end{equation}
Now we show that $f_{\flat}$ satisfies the property \eqref{fflatprop}. Suppose not, apply \textit{Lemma} \ref{freqprunlem1} at level $c\rho$, we may decompose $f_{\flat}=g+b$ with $\widehat{g} $ support on an interval of length $R$ so that 
\begin{equation}
    \rho^{\frac{1}{2}}\|f\|_{L^{2}}\geq \|\widehat{f}_{\flat}1_{I}\|_{L^{2}}\geq \|g\|_{L^{2}}\geq c^{\frac{1}{2}}\rho^{\frac{1}{2}}\|f\|_{L^{2}}
\end{equation}
We again obtain a contradiction if we pick $c$ large enough. $f_{\sharp }$ is of the form
\begin{equation}
    \begin{aligned}
        f_{\sharp}(x)&=\mathcal{F}^{-1}\left(\sum_{n\in \mathfrak{R}}(\Dil^{\infty}_{R}\Tr_{-n}\varphi)\cdot \mathcal{F}f  \right)(x)\\
        &=\sum_{n\in \mathfrak{R}}\left((\Dil^{1}_{R^{-1}}\Mod_{n}\widecheck{{\varphi}})\ast f\right)(x)\\
        &=\sum_{n\in \mathfrak{R}}\int_{\mathbb{R}}f(t)Re^{2\pi i nR(x-t)}\widecheck{{\varphi}}(R(x-t))dt\\
        &=\sum_{n\in \mathfrak{R}}e^{2\pi i nR x}\mathcal{F}(f\cdot \widecheck{{\varphi}}_{R,x} )(nR)
    \end{aligned}
\end{equation}
where $\widecheck{{\varphi}}_{R,x}:=\Tr_{x}\Dil^{1}_{R^{-1}}\widecheck{{\varphi}}$ may be viewed as a $L^{1}$ bump at scale $R^{-1}$ centered around $x$. Hence $\mathcal{F}(f\cdot \widecheck{{\varphi}}_{R,x} )(nR)$ is the function $h$ in \eqref{fsharpdecomps}.

\end{proof}

\subsection{From an oscillatory integral to a sublevel set estimate} \ 

Let $\psi^{(l)}$ be a bump function with frequency support around the annulus at scale $\lambda$ and $\psi^{(3-l)}$ be a bump function with frequency support in the ball at scale $\lambda$. Then by the frequency support of $f_{l}$, we have $f_{l}=\psi^{(l)}\ast_{l}f_{l}$. Instead of merely decomposing in the spatial space, we can also decompose in the frequency space.
\begin{equation}
    f_{l}=\sum_{m\in \mathbb{Z}^{2}}\eta_{m}f_{l}=\sum_{m\in \mathbb{Z}^{2}}\eta_{m}\psi^{(l)}\ast_{l}f_{l}=\sum_{m\in \mathbb{Z}^{2}}\psi^{(l)}\ast_{l}(\eta_{m}f_{l})+\sum_{m\in \mathbb{Z}^{2}}[\eta_{m}\psi^{(l)}\ast_{l}f_{l}-\psi^{(l)}\ast_{l}(\eta_{m}f_{l})]
\end{equation}
The last step is to swap the spatial projection and the frequency projection. Since the area of the Heisenberg box is $\lambda^{-\gamma}\cdot \lambda=\lambda^{1-\gamma}$ which is larger than one, we expect the error term to be small and have enough decay to be integrable in scale. We have a pointwise bound $\lambda^{\gamma-1}$. Indeed,
\begin{equation}
    \begin{aligned}
        |\eta_{m}\psi^{(l)}\ast_{l}f_{l}-\psi^{(l)}\ast_{l}(\eta_{m}f_{l})|&=\int_{\mathbb{R}}|\psi^{(l)}(x-u)(\eta_{m}(x)-\eta_{m}(u))f_{l}(u)du|\\
        &\lesssim \int_{\mathbb{R}}|\psi^{(l)}(x-u)|(|\lambda^{\gamma}\eta_{m}'(\xi)|\cdot |x-u|) ) |f_{l}(u)|du\\
        &=\int_{\mathbb{R}}|x-u||\psi^{(l)}(x-u)||\lambda^{\gamma}\eta_{m}'(\xi)||f_{l}(u)|du\\
        &\lesssim \lambda^{-1}\cdot \lambda^{\gamma}=\lambda^{\gamma -1}
    \end{aligned}
\end{equation}
Intuitively, $\psi^{(l)}$ is now a $L^{1}$ normalized bump function supported at scale $\lambda^{-1}$, so $(x-u)$ will contribute roughly $\lambda^{-1}$. Since $\psi^{(l)}$ doesn't have a compact support, we argue carefully as follows. Let $\widehat{\psi}=D^{\infty}_{\lambda}\widehat{\phi}$, where $\phi $ is a bump with frequency support in unit ball. We have,
\begin{equation}
    \begin{aligned}
        \int_{\mathbb{R}} \verts{x^{k}\psi(x)} dx
        =
        \int_{\mathbb{R}}\verts{x^{k}\lambda \phi (\lambda x)}dx
        =
        \lambda^{-k}
        \int_{\R}
            \verts{
                \br{\lambda x}^k 
                \lambda \phi\br{\lambda x} 
            }
        dx
        =\lambda^{-k}
        \int_{\R}
            \verts{
                x^k\phi\br{x}
            }
        dx
    \end{aligned}
\end{equation}
where $ x^k\phi\br{x}$ is again a Schwartz function.

Fix $\tau=\gamma+ \kappa$ and a small parameter $\delta' >0$. Applying Lemma \ref{freqprunlem2} to the function $x\mapsto f_{1,m}(x,y)$ at the threshhold $\rho \sim \lambda^{\delta'}$ and radius $R=\lambda^{\tau}$, we obtain
\begin{equation}
    f_{1,m}=f_{1,m,\flat}+f_{1,m,\sharp}
\end{equation}
where
\begin{equation}
    f_{1,m,\sharp}(x,y)=\sum_{n=1}^{N}h_{1,n,m}e^{2\pi i \alpha_{n,m}(y)x}
\end{equation}
with $N\sim \lambda^{\delta'}$.
\begin{equation}
    \begin{aligned}
        &\int_{\mathbb{R}}\int_{\mathbb{R}^{2}}1_{|\xi_{1}|\leq \lambda^{\tau}}|\widehat{\mathcal{D}_{s}^{(1)}f_{1,m,\flat}}(\xi)|^{2}d\xi ds\\
        \lesssim &\lambda^{-\delta'}\left\|\left\|\mathcal{F}_{y}f_{1,m,\flat}   \right\|_{L^{2}_{x}}  \right\|_{L^{4}_{y}}^{4}\\
        \lesssim &\lambda^{-\delta'}\left\|\left\|\mathcal{F}_{y}f_{1,m,\flat}   \right\|_{L^{4}_{y}}  \right\|_{L^{2}_{x}}^{4}\\
        =&\lambda^{-\delta'}\left\|\left\|\mathcal{F}_{y}f_{1,m,\flat} \cdot \mathcal{F}_{y}f_{1,m,\flat}  \right\|_{L^{2}_{y}}^{\frac{1}{2}} \right\|_{L^{2}_{x}}^{4}\\
        =&\lambda^{-\delta'}\left\|\left\|\mathcal{F}_{y}(f_{1,m,\flat} \ast_{(2)} f_{1,m,\flat} )\right\|_{L^{2}_{y}}^{\frac{1}{2}} \right\|_{L^{2}_{x}}^{4}\\
        =&\lambda^{-\delta'}\left\|\left\|f_{1,m,\flat} \ast_{(2)} f_{1,m,\flat} \right\|_{L^{2}_{y}}^{\frac{1}{2}} \right\|_{L^{2}_{x}}^{4}\\
        =&\lambda^{-\delta'} \left( \lambda^{-\frac{\gamma}{4}}\cdot \lambda^{-\frac{\gamma}{2}} \right)^{4}=\lambda^{-\delta'-3\gamma}
    \end{aligned}
\end{equation}
where in the last equality, we use
\begin{equation}
    f_{1,m,\flat} \ast_{(2)} f_{1,m,\flat}(x,y)\lesssim 1_{I\times J}(x,y)\|f\|_{L^{\infty}}
\end{equation}
with $|I|\sim |J|\sim  \lambda^{-\gamma}$ and
\begin{equation}
    \left\|\left\| 1_{I\times J}(x,y)\right\|_{L^{2}_{y}}^{\frac{1}{2}} \right\|_{L^{2}_{x}}^{4}\sim \lambda^{-3\gamma}
\end{equation}
Similarly, we can decompose $f_{2,m}$. After a physical truncation, we have
\begin{equation}
    f_{l,m}=\widetilde{\eta}_{m}f_{l,m,\flat}+\widetilde{\eta}_{m}f_{l,m,\sharp}+(1-\widetilde{\eta}_{m})(\psi^{(l)}\ast_{l}(\eta_{m}f_{l}))
\end{equation}
By the mean value theorem, the third term enjoys a pointwise bound $\lambda^{\gamma -1}$. To summarize, we have
\begin{equation}
    f_{l,\flat}=\sum_{m\in \mathbb{Z}^{2}}\widetilde{\eta}_{m}f_{l,m,\flat}
\end{equation}
\begin{equation}
    f_{1,\sharp}(x,y)=\sum_{m\in \mathbb{Z}^{2}}\sum_{n=1}^{N}h_{1,n,m}e^{2\pi i \alpha_{n,m}(y)x}
\end{equation}
\begin{equation}
    f_{2,\sharp}(x,y)=\sum_{m\in \mathbb{Z}^{2}}\sum_{n=1}^{N}h_{2,n,m}e^{2\pi i \beta_{n,m}(y)x}
\end{equation}
\begin{equation}
    f_{l,\mathrm{err}}=\sum_{m\in \mathbb{Z}^{2}}f_{l,m,\mathrm{err}}
\end{equation}
where
\begin{equation}
    f_{l,m,\mathrm{err}}=\eta_{m}\psi^{(l)}\ast_{l}f_{l}-\psi^{(l)}\ast_{l}(\eta_{m}f_{l})+(1-\widetilde{\eta}_{m})(\psi^{(l)}\ast_{l}(\eta_{m}f_{l}))
\end{equation}
We have that $\| f_{l,m,err}\|_{L^{\infty}}\lesssim \lambda^{\gamma -1}$ uniformly in $m$. As such,
\begin{equation}
\begin{aligned}
    T_{\loc}(f_{1},f_{2})&= T_{\loc}(f_{1,\sharp}+f_{1,\flat}+f_{1,\mathrm{err}},f_{2,\sharp}+f_{2,\flat}+f_{2,\mathrm{err}})\\
    &=[ T_{\loc}(f_{1,\sharp},f_{2,\sharp}) ]+[ T_{\loc}(f_{1,\flat},f_{2})+ T_{\loc}(f_{1,\sharp},f_{2,\flat})]+[ T_{\loc}(f_{1,\sharp},f_{2,\mathrm{err}})+ T_{\loc}(f_{1,\mathrm{err}},f_{2})]\\
    &=T_{\sharp}+T_{\flat}+T_{\mathrm{err}}
\end{aligned}
\end{equation}
we have 
\begin{equation}
    \|T_{\mathrm{err}}\|_{L^{1}}\lesssim \lambda^{\gamma-1}
\end{equation}
\begin{equation}
    \|T_{\flat}\|_{L^{1}}\lesssim \lambda^{\frac{1}{2}-\gamma}+\lambda^{-\frac{\kappa}{4}+\frac{5}{4}\varepsilon_{2}}+\lambda^{-\frac{\delta'}{4}+\varepsilon_{2}}
\end{equation}
It thus remains to estimate the term $T_{\sharp}$. 

Since $N$ is at the scale $\lambda^{\delta'}$, we bound $\|T_{\sharp}\|_{L^{1}}$ by a sum of $O(\lambda^{2\delta'})$ terms of the form
\begin{equation}\label{Hform}
    \sum_{m\in (\mathbb{Z}^{2})^{2}}\int_{\mathbb{R}^{2}}\left|  \int_{\mathbb{R}}e^{2\pi i(\alpha_{m_{1}}(y)t+\beta_{m_{2}}(x)t^{2})}H_{m}(x,y,t)dt\right|dxdy
\end{equation}
where 
\begin{equation}
    H_{m}(x,y,t)=(\widetilde{\eta}_{m_{1}}h_{1,m_{1}})(x+t,y)(\widetilde{\eta}_{m_{2}}h_{2,m_{2}})(x,y+t^{2})\zeta (x,y,t)
\end{equation}
Let $\widetilde{\mathfrak{I}}$ be the set of $m\in (\mathbb{Z}^{2})^{2}$ such that \eqref{Hform} is nonzero. As before, $\sharp \widetilde{\mathfrak{I}}\lesssim \lambda^{3\gamma}$. For $m\in \widetilde{\mathfrak{I}}$, we fix $(\bar{x}_{m},\bar{y}_{m},\bar{t}_{m})$ in the support of $H_{m}$. For each $(x,y,t)$ in the support of $H_{m}$, we have that the t-derivative of the phase function is 
\begin{equation}
    \alpha_{m_{1}}(y)+2t\beta_{m_{2}}(x)=\alpha_{m_{1}}(y)+2\bar{t}_{m}\beta_{m_{2}}(x)+O(\lambda^{1-\gamma})
\end{equation}
Here we use $|\beta_{m_{2}}(x)|\lesssim \lambda $. Choose a small parameter $\rho$ and suppose that $(x,y,t)$ is such that
\begin{equation}
    |\alpha_{m_{1}}(y)+2\bar{t}_{m}\beta_{m_{2}}(x)|\geq \lambda^{\tau+\rho}
\end{equation}
Using $\|\partial^{N}_{t}H_{m}\|_{L^{\infty}}=O(\lambda^{\tau N})$ and integration by parts, we have
\begin{equation}
    \left| \int_{\mathbb{R}}e^{2\pi i(\alpha_{m_{1}}(y)t+\beta_{m_{2}}(x)t^{2})}H_{m}(x,y,t)dt \right| \lesssim \lambda^{\tau N}\cdot \lambda^{(-\tau-\rho) N}=\lambda^{-\rho N}
\end{equation}
for every $N>0$ (since $\rho$ is fixed, we may also choose a large enough $N$ to absorb $\rho$). Hence, it remains to bound
\begin{equation}\label{nocancel}
    \begin{aligned}
        &\sum_{m\in \widetilde{\mathfrak{I}}}\int_{\mathbb{R}^{2}}\left|  \int_{\mathbb{R}}e^{2\pi i(\alpha_{m_{1}}(y)t+\beta_{m_{2}}(x)t^{2})}H_{m}(x,y,t)dt\right|1_{ |\alpha_{m_{1}}(y)+2\bar{t}_{m}\beta_{m_{2}}(x)|\leq\lambda^{\tau+\rho}}dxdy\\
        \leq &\int_{K}\sum_{m\in \widetilde{\mathfrak{I}}}1_{(x+t,y)\in Q_{m_{1}}}1_{(x,y+t^{2})\in Q_{m_{2}}}1_{ |\alpha_{m_{1}}(y)+2\bar{t}_{m}\beta_{m_{2}}(x)|\leq\lambda^{\tau+\rho}} dxdydt
    \end{aligned}
\end{equation}
where $K$ is the compact support of $\zeta$. 

To conclude the proof of the smoothing inequality, we need the following sublevel set estimate whose proof we report to the next section. 

\begin{lemma}[Lemma 3.3 in \cite{CHRIST2021107863}]\label{lemma 3.3}
Let $K \subseteq \R^2 \times (0,\infty)$ be a compact set and $\ap, \beta: \R^2 \rightarrow \R$ measurable functions. Suppose that either $|\ap| \sim 1$ or $|\beta| \sim 1$, and $|\ap|+|\beta| \lesssim 1$. Then there exists $\sg = \sg(K)>0, C = C(K) >0$ s.t. $\forall \ep \in (0,1]$,
\begin{align}\label{eq 3.40}
    |\{ (x, y, t) \in K; |\ap(x+t, y) - 2t\beta(x,y+t^2)| \leq \ep\}| \leq C \ep^{\sg}.
\end{align}
\end{lemma}

To conclude the proof of the smoothing inequality, note that for each $(u,v)\in \mathbb{R}^{2}$, there exists at most one $m$ with $(u,v)\in Q_{m}\in \mathcal{Q}$. This defines a measurable function
\begin{equation}
    l:\mathbb{R}^{2} \mapsto \mathbb{Z}^{2}\quad \text{with} \quad l(u,v)=m \quad if \quad (u,v)\in Q_{m}
\end{equation}
and $l(u,v)=0$ if no such $l$ exists. Then for every $(x,y,t)\in K$,
\begin{equation}
    \sum_{m\in \widetilde{\mathfrak{I}}}1_{(x+t,y)\in Q_{m_{1}}}1_{(x,y+t^{2})\in Q_{m_{2}}}1_{ |\alpha_{m_{1}}(y)+2\bar{t}_{m}\beta_{m_{2}}(x)|\leq\lambda^{\tau+\rho}}\leq 1_{|\widetilde{\alpha}(x+t,y)-t\widetilde{\beta}(x,y+t^{2})|\leq \varepsilon}
\end{equation}
where
\begin{equation}
    \widetilde{\alpha}(x+t,y):=\lambda^{-1}\alpha_{l(x+t,y)}(y),\quad \widetilde{\beta}(x,y+t^{2}):=-\lambda^{-1}\beta_{l}(x,y+t^{2})(x),\quad \varepsilon:=2\lambda^{\tau+\rho -1}
\end{equation}
Then 
\begin{equation}
    \eqref{nocancel}\lesssim |\{(x,y,,t)\in K : |\widetilde{\alpha}(x+t,y)+2t\widetilde{\beta}(x,y+t^{2})|\leq \varepsilon\}|
\end{equation}
Finally, by the sublevel set estimate in Lemma \ref{lemma 3.3}, we get the desired result. 
$\Box$

\section{A sublevel set estimate}

As mentioned in the previous section, the proof of the trilinear smoothing inequality reduces to that of a sublevel set estimate in Lemma \ref{lemma 3.3} which relies heavily on the pair $(t, t^2)$ being two linearly independent monomials.

\begin{remark}
    In a recent preprint \cite{chen2023polynomial}, Xuezhi Chen and Jingwei Guo prove sublevel set estimates where the pair $(t, t^2)$ is replaced by $(P_1(t), P_2(t))$, where $P_1, P_2$ are two linearly independent polynomials with no constant term. 
\end{remark}

\noindent\begin{proof}[Proof of Lemma \ref{lemma 3.3}] Let $z = (x, y) \in \R^2$. Consider the case $|\ap| \sim 1$. After a change of variables, $(x,y,t) \mapsto (x-t, y, t)$, an affine transformation in $z$, and a rescaling in $t$, the estimate to prove can be rewritten as
\begin{align}\label{eq 3.41}
    |\mE|:=| \{(z,t) \in K; |\ap(z) - 2t\beta(z+(-t, t^2))| \leq \ep\} | \lesssim \ep^{\tau},
\end{align}
where $K=[0,1]^2 \times I$ and $I \subseteq (0, \infty)$ is a closed interval with $|I| =1$. 

\subsection{Constructing a parameter set} \ 

Following an approach reminiscent of the method of refinements, the authors introduce a parameter set $\mA$ and reduce the original sublevel set estimate to an estimate of $\mA$. 

\begin{claim}
There exist a point $\zb = (\xb, \yb) \in \R^2$, and a measurable set $\mA \subseteq I^3$ s.t. $|\mA| \gtrsim |\mE|^7$ and for every $(t_1, t_2, t_3 ) \in \mA$,
\begin{multline}\label{eq 3.42}
    \begin{cases}
        |\ap(\zb) - 2t_1 \beta(\zb + (-t_1, t_1^2))| \leq \ep\\
        |\ap(\zb + (-t_1, t_1^2) - (-t_2, t_2^2)) - 2t_2\beta(\zb +(-t_1,t_1^2))| \leq \ep \\
        |\ap(\zb + (-t_1, t_1^2) -(t_2,t_2^2)) - 2t_3 \beta(\zb + (-t_1, t_1^2) - (-t_2, t_2^2) + (-t_3,t_3^2))| \leq \ep.
    \end{cases}
\end{multline}
\end{claim}

\noindent\begin{proof}[Proof of Claim] To construct the parameter set $\mA$, we first cut the edges off of $\mE$ along the $z$-plane. That is,
\begin{align*}
    \mE_0'&:= \{z \in [0,1]^2; \  |\{t \in I; (z, t) \in \mE\}| \geq  |\mE|/2 \}.
\end{align*}

\begin{center}
\includegraphics[width=4in]{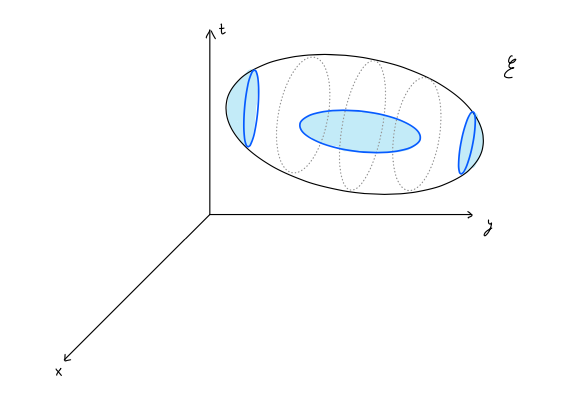}
\end{center}
After the first trim, we obtain $\mE_1$: 
\begin{align*}
    \mE_1 &:=\{(z, t) \in \mE; z \in \mE'_0 \}.
\end{align*}

In the second step, we trim out all $z$ values for which the curve $(z+(-t^2, t)) \in \mE_1$ is too short. That is, we consider
\begin{align*}
    \mE_1'&:= \{z \in \R^2; \ |\{t \in I; (z+(-t,t^2), t) \in \mE_1\}| \geq |\mE_1|/2\}.
\end{align*}

\begin{center}
    \includegraphics[width=4in]{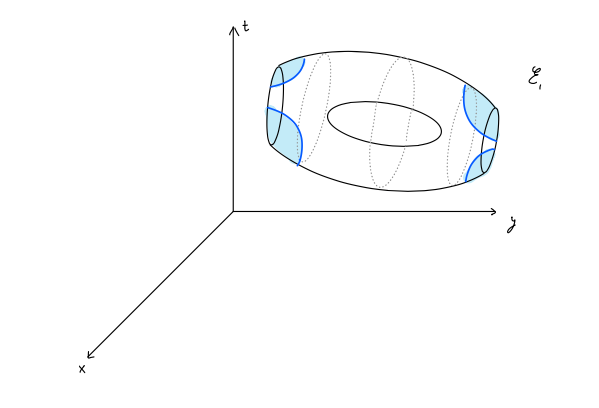}
\end{center}
After this second trim, we replace $\mE_1$ with $\mE_2$ defined by
\begin{align*}
    \mE_2& := \{(z,t) \in \mE_1; z\in \mE_1'\} = \{(z,t) \in \mE_1; (z+(-t,t^2), t) \in \mE_1\}.
\end{align*}

Finally, we trim out all $z$ values for which the curve $(z-(-t,t^2))\in \mE_2$ is too short.
\begin{align*}
    \mE_2'& := \{z\in \R^2; |\{t\in I; (z-(-t,t^2), t) \in \mE_2\}| \geq |\mE_2| /2\}.
\end{align*}
With Fubini's theorem, we can estimate the size of $\mE_0'$, $\mE_1'$, and $\mE_2'$. To estimate $\mE_0'$,
\begin{align}\label{eq 3.43}
    |\mE| = \int_{\mE_0'} |\{t\in I; (z, t) \in \mE \}| dz + \int_{[0,1]^2 \backslash \mE_0'} |\{t \in I; (z, t) \in \mE \}|dz
\end{align}
\begin{align*}
    \leq |\mE_0'| + \frac{1}{2} |\mE|.
\end{align*}
As such, $|\mE_0'| \geq \frac{1}{2} |\mE|$. 

To estimate $\mE_1'$, we first estimate the size of the set $\mE_1$ after the first trim. By Fubini again,
\begin{align}\label{eq 3.44}
    |\mE_1| = \int_{\R^2} \lp \int_I 1_{\mE}(w,t)dt \rp 1_{\mE_0'}(w)dw \geq \frac{1}{2} |\mE| \cdot |\mE_0'| \geq \frac{1}{4} |\mE|^2.
\end{align}
Hence, $|\mE_1'| \geq \frac{1}{2} |\mE_1| \geq \frac{1}{8} |\mE|^2$. In turn, following \eqref{eq 3.44}, $|\mE_2| \geq \frac{1}{2} |\mE_1| \cdot |\mE_1'| \geq 2^{-6} |\mE|^4$.

Finally, $|\mE_2'| \geq \frac{1}{2} |\mE_2| \geq 2^{-7} |\mE|^4 >0 $. Since $\mE_2' \neq \emptyset$, we can choose a point $\zb \in \mE_2'$. To construct the parameter space $\mA$ in $\R^3$, define the following three sets:
\begin{align}\label{eq 3.45}
    U := \{t\in I; (\zb -(-t,t^2), t) \in \mE_2 \}.
\end{align}
For each $t_1 \in U$,
\begin{align}\label{eq 3.46}
    U_{t_1} := \{t\in I; (\zb - (t_1, t_1^2) + (-t,t^2),t) \in \mE_1\}.
\end{align}
For each $t_1 \in U, t_2 \in U_{t_1}$,
\begin{align}\label{eq 3.47}
    U_{t_1, t_2} := \{t\in I; (\zb -(-t_1, t_1^2) + (-t_2, t_2^2), t) \in \mE\}.
\end{align}
$\mA$ is then defined by
\begin{align*}
    \mA :=\{ (t_1, t_2, t_3) \in I^3; t_1 \in U, t_2 \in U_{t_2}, t_3 \in U_{t_1, t_2}\}.
\end{align*}
To obtain a lower bound for the size of $\mA$, by Fubini's theorem, 
\begin{align*}
    |\mA| &= \int_I 1_U(t_1) \int_I 1_{U_{t_1}}(t_2) \int_I 1_{U_{t_1,t_2}}(t_3) dt_3 dt_2 dt_1 \geq |U| \cdot \inf_{t_1 \in U} |U_{t_1}| \cdot \inf_{t_1 \in U, t_2 \in U_{t_1}} |U_{t_1, t_2}|\\
    & \gtrsim |\mE_2| \cdot  |\mE_1| \cdot|\mE| \gtrsim |\mE|^4 \cdot |\mE|^2 \cdot |\mE| = |\mE|^7.
\end{align*}
Thus completing the proof of the claim. 

\end{proof}

\vspace{.2in}

The proof of the sublevel set estimate thus reduces to proving that $|\mA| \lesssim \ep^{c}$, for some $c>0$. Let $\ap_0:= \ap(\zb)$ and $\tbf =(t_1, t_2, t_3) \in \R^3$. Consider the function\footnote{Note the typo in the definition of $F$ in \cite{CHRIST2021107863} which we fixed here.} 
\begin{align*}
    F(\tbf) & = \ap_0 t_3^{-1} t_2 t_1^{-1} - 2 \beta ( \zb + (-t_1, t_1^2) - (-t_2, t_2^2) + (-t_3,t_3^2)).
\end{align*}

Recall that $I \Subset (0, \infty)$. Hence, $0<c\lesssim t_j \lesssim C$ for some $c,C>0$. As such,
\begin{align*}
    |F(\tbf)| &\lesssim |\ap_0 t_1^{-1} - 2 t_2^{-1}t_3 \beta ( \zb + (-t_1, t_1^2) - (-t_2, t_2^2) + (-t_3,t_3^2)) |\\
    &\lesssim |\ap_0 t_1^{-1} -2 \beta(\zb + (-t_1, t_1^2))| + |t_2^{-1}\ap(\zb + (-t_1, t_1^2) - (-t_2, t_2^2)) -  2 \beta(\zb + (-t_1, t_1^2))| \\
    &+ |t_2^{-1}\ap(\zb + (-t_1, t_1^2) - (-t_2, t_2^2)) - 2 t_2^{-1} t_3  \beta ( \zb + (-t_1, t_1^2) - (-t_2, t_2^2) + (-t_3,t_3^2))|\\
    & \lesssim |\ap_0  -2 t_1 \beta(\zb + (-t_1, t_1^2))| + |\ap(\zb + (-t_1, t_1^2) - (-t_2, t_2^2)) -  2 t_2\beta(\zb + (-t_1, t_1^2))| \\
    &+ |\ap(\zb + (-t_1, t_1^2) - (-t_2, t_2^2)) - 2 t_3  \beta ( \zb + (-t_1, t_1^2) - (-t_2, t_2^2) + (-t_3,t_3^2))|.
\end{align*}
By \eqref{eq 3.42}, $|F(\tbf)| \lesssim \ep$ for every $\tbf \in \mA$. The proof of \eqref{eq 3.41} reduces to proving that the sublevel sets of the function $F$ are small \textit{uniformly in all measurable functions} $\beta$. 

\subsection{Geometric tools: changes of coordinates} \ 

To simplify the notation, define the following functions
\begin{align*}
    \theta_1(\tbf) = -t_1 +t_2 - t_3 \hspace{.3in} \theta_2(\tbf) = t_1^2 - t_2^2 + t_3^2 \hspace{.3in}  \nu(\tbf) = t_3^{-1} t_2 t_1^{-1}.
\end{align*}
With this new notation, observe that $F(\tbf) = \ap_0 \nu(\tbf) -\beta (\zb +(\theta_1(\tbf), \theta_2(\tbf)))$. 

By evaluating $F$ along the flow of a carefully chosen vector field $V$ in $\R^3$, we obtain bounds uniform in all measurable functions $\beta$. The vector field of interest is given by:
\begin{equation*}
    V(\tbf) = (\nabla \theta_1 \times \nabla \theta_2) (\tbf) = 2(t_3-t_2) \p_{t_1} + 2(t_3 - t_1) \p_{t_2} + 2 (t_2 - t_1) \p_{t_3}.
\end{equation*}
Note that $V$ only vanishes on a set of measure zero: the diagonal line $\Delta = \{(t_1, t_2, t_3) ; t_1 = t_2 = t_3\} \subseteq \R^3$. With $\g : \R^3 \times \R \rightarrow \R^3$ denoting the flow associated to $V$, given $\tbf \in \R^3$, 
\begin{equation*}
    \frac{d}{ds} F(\g(\tbf, s)) = (VF)(\g(\tbf, s)) = \ap_0V( \nu(\g(\tbf, s)) - \beta(\zb +(\theta_1(\g(\tbf, s)), \theta_2(\g(\tbf, s))),
\end{equation*}
where $\theta_1, \theta_2$ are constant along the integral curves of the vector field $V = \nabla \theta_1 \times \nabla \theta_2$. Indeed, $\frac{\p}{\p t_j} \theta_l(\g(\tbf)) = \p_{t_j} \theta_l(\g(\tbf)) \cdot \nabla \theta_1 \times \nabla \theta_2 = 0$ for $l=1, 2$ and $j=1, 2, 3$. As such, $\beta(\zb +(\theta_1(\g(\tbf, s)), \theta_2(\g(\tbf, s)))$ is also constant. We can thus write
\begin{equation*}
    s \mapsto F(\g(\tbf, s)) = \ap_0 \nu(\g(\tbf,s)) - \textrm{constant},
\end{equation*}
where first part is nonconstant and non-zero away from a set of measure zero. Indeed,
\begin{align}\label{eq 3.48}
    V \nu(\tbf) = 2t_1^{-2} t_3^{-2} ((t_2-t_3)t_3 t_2 + (t_3 - t_1) t_3 t_1 + (t_1- t_2) t_2 t_1),
\end{align}
where $t_j \gtrsim 1$ for $j=1, 2, 3$. To proceed with the proof, we change coordinates three times to obtain the desired estimate.

First consider the coordinates $\mathbf{u} = J \tbf$ with $J \in \R^{3 \times 3}$ given by
\begin{align*}
    (u_1, u_2, u_3) = \begin{pmatrix}
        -1 & 1&0\\
        0&1&-1\\
        1& -1 & 1
    \end{pmatrix} \begin{pmatrix}
        t_1 \\ t_2 \\ t_3
    \end{pmatrix}  = (t_2 - t_1, t_2 - t_3, t_1 - t_2 + t_3),
\end{align*}
where observe that $|\det J| =1$. Let $\Omega := J(I^3) \subseteq [-1,1]^2 \times (I- I +I)$. After decomposing $\Omega$ into subsets:
\begin{align*}
    \Omega_{k, d} = \{\textbf{u} \in \Omega; 2^{-k-1} \leq \max (|u_1|, |u_2|) \leq 2^{-k},\ \  |u_3 - 2^{-k}d| \leq 2^{-k} \},
\end{align*}
where $k,d \in \Z$, we have
\begin{align*}
    1_{\Omega} \leq \sum_{k=0}^{\infty} \sum_{|d| \lesssim 2^k} 1_{\Omega_{k, d}}.
\end{align*}
Using this initial change of coordinates, we obtain
\begin{align*}
    |\mA| = \int_{\R^3} 1_{\mA} = \int_{\R^3} 1_{\Om \cap J \mA} \leq \sum_{k = 0 }^{\infty} \sum_{|d| \leq 2^k} |\Om_{k,d} \cap J \mA|.
\end{align*}
To estimate $|\Om_{k,d} \cap J \mA|$ for a given $k \geq 0$ and $|d| \lesssim 2^k$, we proceed with a second change of variables:
\begin{align*}
    \textbf{v} = (v_1, v_2, v_3) = 2^k (u_1, u_2, u_3 - 2^{-k} d) = \Lb_{k,d} \textbf{u}.
\end{align*}
As such, $\Om_{k, d}$ maps into a normalized box
\begin{align*}
    \Lb_{k,d} \Om_{k,d} \subseteq  \Box = \{\textbf{v} = (v_1, v_2, v_3); 2^{-1} \leq \max (|v_1|, |v_2|) \leq 1, \ \ |v_3| \leq 1 \}.
\end{align*}
The vector field $V$ in the normalized coordinates is given by
\begin{align*}
    \frac{1}{2} \wt{V}(\textbf{v}) &= \frac{1}{2} ((\Lb_{k,d} \circ J)_* V) (\textbf{v})  =\frac{1}{2} d(\Lb_{k,d} \circ J)_{(\Lb_{k,d} \circ J)^{-1}(v)} V((\Lb_{k,d} \circ J)^{-1}v) \\
    &= v_1 \p_{v_1} -v_2 \p_{v_2}.
\end{align*}
To see, by unpacking the changes of coordinates, we see that
\begin{equation}\label{eq J inv lb inv v}
\begin{aligned}
     J^{-1} \circ \Lb_{k,d}^{-1} v & = \begin{pmatrix}
        0&1&1\\
        1&1&1\\
        1&0&1
    \end{pmatrix} \begin{pmatrix}
        2^{-k} v_1 & 2^{-k} v_2 & 2^{-k}(v_3 + d)
    \end{pmatrix}\\
    &= \begin{pmatrix}
        2^{-k} v_2 + 2^{-k} (v_3 + d) \\
        2^{-k} v_1 + 2^{-k} v_2 + 2^{-k} (v_3 + d)\\
        2^{-k} v_1 + 2^{-k} (v_3 + d)
    \end{pmatrix}.
    \end{aligned}
\end{equation}
In addition,
\begin{align*}
    \frac{1}{2} d(\Lb_{k,d} \circ J)_{(\Lb_{k,d} \circ J)^{-1}(v)} V((\Lb_{k,d} \circ J)^{-1}v) &= \begin{pmatrix}
        2^k & 0& 0\\
        0 & 2^k & 0 \\
        0 & 0 &2^k
    \end{pmatrix} \begin{pmatrix}
        -1 & 1 & 0 \\
        0 & 1 & -1\\
        1 & -1 &1
    \end{pmatrix}
    \begin{pmatrix}
        -2^{-k} v_2 \\ 2^{-k}v_1 -2^{-k} v_2\\ 2^{-k }v_1
    \end{pmatrix}\\
    &= \begin{pmatrix}
        v_1 \\ -v_2 \\ 0
    \end{pmatrix}.
\end{align*}
To compute the integral curves $\g(s) = (v_1(s), v_2(s))$, we solve the following ODE:
\begin{align*}
    \begin{cases}
        v_1'(s) =2v_1(s) \\
        v_2'(s) = -2v_2(s)
    \end{cases}
\end{align*}
and obtain
\begin{align*}
    \g: \R^3 \times \R \rightarrow \R^3, \hspace{.3in} \g(c; s) = (c_1 e^{2s}, c_2e^{-2s}, c_3).
\end{align*}
Setting $c_l = \pm 1$, for $l =1,2$, we obtain a foliation of the half-space $(\pm v_l >0)$ by one-dimensional integral curves.

Finally, the third change of variables is given by:
\begin{align*}
    &\vp_{1, \pm}: \R^3 \rightarrow \{\pm v_1 >0\}, \hspace{.3in} (w_1, w_2, w_3) \mapsto \g((\pm 1, w_2, w_3); w_1), \\
    &\vp_{2, \pm}: \R^3 \rightarrow \{\pm v_2 >0\} \hspace{.3in} (w_1, w_2, w_3) \mapsto \g((w_2, \pm 1, w_3); w_2).
\end{align*}
As such, $\vp_{l,\pm}$ is a diffeomorphism and $\det D \vp_{l,\pm} = \pm 2$.

Decomposing $\Box$ into four boxes, 
\begin{align*}
    \Box = \bigcup_{l = 1, 2; \pm } \Box_{l,\pm}, \hspace{.5in} \Box_{l,\pm} = \Box \cap \{\pm v_{l} \geq 2^{-1}\}
\end{align*}
we obtain compactness in the $\textbf{w}$-coordinates. For example, after retracing the changes of variables, we obtain that for $(v_1, v_2, v_3) \in \Box_{1,+}$ (so that $v_1 \geq 2^{-1}$), 
\begin{align*}
    \vp_{1,+}^{-1} (v_1, v_2, v_3) \in [ 1/2 \log (1/2), 0] \times [-1, 1] \times [-1, 1].
\end{align*}

Going back to estimating $|\Om_{k, d} \cap J \mA|$, keeping in mind that $\vp_{l,\pm}$ is a diffeomorphism with Jacobian determinant equal to $\pm 2$, and recalling the decomposition into boxes $\Lb_{k,d} \Om_{k,d} \subseteq \Box = \bigcup_{l =1,2} \Box_{l,\pm}$, we have
\begin{align*}
    |\Om_{k, d} \cap J \mA| &= 2^{-3k} |\Lb_{k,d} \Om_{k,d} \cap \Lb_{k, d} J \mA |\\
    & = 2^{-3k +1} \sum_{l \in \{1,2\}, \pm  } |\vp_{l, \pm}^{-1} (\Lb_{k,d} \Om_{k,d} \cap \Box_{l,\pm} ) \cap \vp_{l, \pm}^{-1} \Lb_{k, d} J \mA|.
\end{align*}
Letting
\begin{align*}
    G_{l, \pm,k,d}:= F \circ J^{-1} \circ \Lb_{k,d}^{-1} \circ \vp_{l, \pm},
\end{align*}
where $F$ is the original map whose sublevel sets we wish to estimate, $J$ diagonalizes the system given by the vector field $V$, $\Lb_{k,d}$ normalizes the coordinates and finally $\vp_{l,\pm}$ gave rise to compact sets. Let $\mK_{l, \pm, k, d} := \vp_{l, \pm}^{-1} (\Lb_{k,d} \Om_{k,d} \cap \Box_{l,\pm} )$ so that  $ \mK_{l, \pm, k, d} \subseteq \vp_{l,\pm}^{-1}(\Box_{l,\pm})$ is a compact set. We thus have,
\begin{align*}
    |\Om_{k, d} \cap J \mA| & \lesssim 2^{-3k+1} \sum_{l \in \{1, 2\}, \pm } |\{w \in \mK_{l, \pm, k, d}; |G_{l, \pm,k,d}(w)| \lesssim \ep \}|.
\end{align*}
Indeed, for $w \in \vp_{l, \pm}^{-1} \Lb_{k, d} J \mA$, tracing back the changes of coordinates, we have
\begin{align*}
    |G_{l, \pm,k,d}(w)| = |F \circ J^{-1} \circ \Lb_{k,d}^{-1} (v)| = |F \circ J^{-1} (u)| = |F(t)| \lesssim \ep.
\end{align*}

To apply Van der Corput's lemma, first note that the first partial of $G_{l, \pm, k, d}$ is given by
\begin{align*}
    \p_{w_1}G_{l, \pm, k ,d}(w) = \p_{w_1} F \circ J^{-1} \circ \Lb_{k,d}^{-1} \circ \vp_{l, \pm} (w),
\end{align*}
where $\vp_{l,\pm}$ are defined in terms of the integral curve of $\wt{V}$ and $F$ is thus evaluated along the integral curve of $V$, as before, after two additional changes of coordinates. Hence,
\begin{align}\label{eq 3.49}
    \p_{w_1}G_{l, \pm, k ,d}(w) = \ap_0 \wt{V}(\nu \circ J^{-1} \circ \Lb_{k,d}^{-1})|_{\vp_{l, \pm} (w)}.
\end{align}

We detail the computation to show that $\wt{V}(\nu \circ J^{-1} \circ \Lb_{k,d}^{-1})|_{\textbf{v}}$ can be written as a product of a scalar multiple of a homogeneous polynomial of degree three with a term that is approximately unimodular. 
By \eqref{eq J inv lb inv v},
\begin{align*}
    J^{-1} \circ \Lb_{k,d}^{-1} (v) =
    \begin{pmatrix}
        2^{-k} (v_2 + v_3 +d) \\
        2^{-k} ( v_1 + v_2 + v_3 + d) \\
        2^{-k} ( v_1 +v_3 +d)
    \end{pmatrix}.
\end{align*}
As such,
\begin{align*}
    \nu \circ J^{-1} \circ \Lb_{k,d}^{-1} (v) = (2^{-k} ( v_1 +v_3 +d))^{-1} (2^{-k} ( v_1 + v_2 + v_3 + d)) (2^{-k} (v_2 + v_3 +d))^{-1}.
\end{align*}
Applying the vector field $\wt{V}$ to the above function, we have
\begin{align*}
    \wt{V}(\nu \circ J^{-1} \circ \Lb_{k,d}^{-1})|_{\textbf{v}}
    & = 2^{2k}(v_1 \p_{v_1} - v_2 \p_{v_2}) ( v_1 +v_3 +d)^{-1} ( v_1 + v_2 + v_3 + d) (v_2 + v_3 +d)^{-1}\\
    &= v_1 (2^{-k} (v_1 + v_3 +d))^{-2} - v_2 (2^{-k} (v_2+v_3+d))^{-2}\\
    &= 2^{-2k} (2^{-k} (v_1 + v_3 +d))^{-2} (2^{-k} (v_2+v_3+d))^{-2} \\
    & \hspace{1in}(v_1  (v_2+v_3+d)^2 - v_2 (v_1 + v_3 +d)^{2})\\
    &= 2^{-2k} R_{k,d}(\textbf{v}) P(\textbf{v}),
\end{align*}
where $P$ is a homogeneous polynomial of degree 3 that is independent of $k$ and does not vanish identically, and
\begin{align}\label{eq 3.50}
    R_{k,d}(\textbf{v}) = (2^{-k}(v_2+ v_3+d))^{-2} (2^{-k}(v_1 +v_3+d))^{-2}.
\end{align}
If $\textbf{v} \in \Lb_{k,d}\Om_{k,d} \subseteq \Lb_{k,d} J(I^3)$, then $|R_{k,d}(\textbf{v})| \sim 1$. This can easily be seen by retracing the changes of variables. Hence, if $w \in \mK_{l,\pm,k,d}$
\begin{align}\label{eq 3.51}
    |\p_{w_1}G_{l, \pm, k ,d}(w)| \sim 2^{-2k}|P(\vp_{l,\pm}(w))|.
\end{align}

By the multivariable van der Corput's lemma, there exists $a>0$ s.t. 
\begin{align}\label{eq 3.52}
    |\{w \in \mK_{l, \pm, k, d}; |P(\vp_{l, \pm}(w))| \leq \ep\}| \lesssim \ep^a.
\end{align}

\vspace{.3in}

\subsection{Analyticity considerations: one last sublevel set estimate} \

\begin{claim}
\eqref{eq 3.52} implies 
\begin{align}\label{eq 3.53}
    |\{w \in \mK_{l,\pm, k,d};|G_{l, \pm ,k,d}(w)| \lesssim \ep \}| \lesssim 2^{Ck} \ep^c
\end{align}
uniformly in $k, d$, where $c, C>0$ are some fixed constants.
\end{claim}

\hspace{.3in}

Before proving the claim, we walk through the conclusion of the proof of Lemma \ref{lemma 3.3}. By \eqref{eq 3.53}, we have
\begin{align*}
    |\Om_{k,d} \cap J \mA| \lesssim 2^{Ck} \ep^c.
\end{align*}
Recalling the trivial estimate $|\Om_{k,d} \cap J \mA| \lesssim 2^{-3k}$, we have
\begin{align*}
    |\mA| \lesssim \sum_{k=0}^{\infty} \sum_{|d| \leq 2^k} \min (2^{Ck} \ep^c, 2^{-3k})\approx \sum_{k=0}^{\infty} 2^k \min (2^{Ck} \ep^c, 2^{-3k}) \lesssim \ep^{c'}.
\end{align*}
As such, $|\mE| \lesssim |\mA|^{1/7} \lesssim \ep^{c'/7}$. Thus proving Lemma \ref{lemma 3.3}.
\end{proof}

\hspace{.5in}

\noindent\begin{proof}[Proof of Claim]
The authors focus on the case $(l, \pm) = (1, +)$ since the other three cases follow in a similar way. The authors introduce the ${\L}$ojasiewicz inequality\footnote{\textbf{$\L$ojasiewicz inequality} - Let $f:U \rightarrow \R$, where $U \subseteq \R^n$ is open, be a real analytic function. Let $\mZ$ denote the zero set of $f$. Assume $\mZ \neq \emptyset$. Then for all compact $K \Subset U$, $\exists \ap, C>0$ s.t. $\forall x\in K$, \begin{equation*}
    \textrm{dist}(x,\mZ)^{\ap} \leq C |f(x)|.
\end{equation*}} which can only be applied to an \textit{analytic} function. However $G_{k,d}$ is only analytic in one variable $w_1$. Indeed,
\begin{align*}
    G_{k,d} (w_1, w_2, w_3) &= F \circ J^{-1} \circ \Lb_{k,d}^{-1} \circ \vp_{1,+}(w_1, w_2, w_3) \\
    &=F \circ J^{-1} \circ \Lb_{k,d}^{-1} (e^{2w_1}, w_2 e^{-2w_1}, w_3).
\end{align*}
On the other hand, recalling the calculation above, $\p_{w_1} G_{k,d} = 2^{-2k} R_{k,d}(\gamma((+,w_2, w_3); w_1)) \\ P(\gamma((+,w_2, w_3); w_1))$ is analytic in all three variables $(w_1, w_2, w_3)$ as desired. We thus apply the $\L$ojasiewicz inequality to $\p_{w_1} G_{k, d}$.

In addition, in view of carefully recording how the constants depend on $k$, thanks to \eqref{eq 3.51}, we apply the $\L$ojasiewicz inequality to $P \circ \vp$ instead of $\p_1 G_{k,d}$, keeping in mind that \eqref{eq 3.51} only holds on $\mK_{k,d}$. 

As observed earlier, $\mK_{k,d} \subseteq \mK = [1/2 \log (1/2), 0] \times [-1,1]^2$. Let $\mK^*$ denote an open $2^{-10}$-neighborhood of $\mK$ and consider the zero set
\begin{align*}
    \mathcal{Z}:= \{w \in \mK^*; P(\vp(w)) = 0\}.
\end{align*}
The $\L$ojasiewicz inequality gives an upper bound for the distance between points in the compact set $\mK$ and the zero set of the real analytic function $P(\vp(w))$: there exists $b>0$ s.t. for $w \in \mK$, 
\begin{align}\label{eq 3.54}
    |P(\vp(w))| \gtrsim \textrm{dist}(w,\mZ)^b.
\end{align}

To obtain the sublevel set estimate for $G_{1,+,k,d}$, we cover $\mK_{k,d} = \vp^{-1}_{1,+} (\Lb_{k,d}\Om_{k,d} \cap \Box_{1,+})$ with a grid of closed, axis-aligned cubes $Q \subseteq \mK$ with pairwise disjoint interiors, each of sidelength $\rho$, where $\ep$ is sufficiently small and $\rho$ is such that $\ep << \rho << 1$ and will determined at the end of the computations. Let $\mQ$ denote the collection of all these cubes. The cubes are then distributed into three subcollections:
\begin{align*}
    \mQ = \mQ_{\textrm{near}} \cup \mQ_{\textrm{far}} \cup \mQ_{\textrm{bdry}}
\end{align*}
where
\begin{align*}
    \mQ_{\textrm{near}} &= \{Q\in \mQ; Q \subseteq \mK_{k,d}, \textrm{dist}(Q,\mZ) \leq \rho \}, \\
    \mQ_{\textrm{far}} &= \{Q\in \mQ; Q \subseteq \mK_{k,d}, \textrm{dist}(Q,\mZ) > \rho\}, \\
    \mQ_{\textrm{bdry}}&=\{Q\in \mQ ; Q \cap \p \mK_{k,d} \neq \emptyset \}.
\end{align*}

If $Q \in \mQ_{\textrm{near}}$, by the mean value theorem, $|P(\vp(w))| \lesssim \rho$, $\forall w \in Q$. By \eqref{eq 3.52}, 
\begin{align*}
    |\bigcup_{\mQ_{\textrm{near}}}Q| \lesssim \rho^a.
\end{align*}

If $Q \in \mQ_{\textrm{far}}$, by \eqref{eq 3.54}, $|P(\vp(w))| \gtrsim \rho^b$, $\forall w \in Q$. In turn, by \eqref{eq 3.51} noting that $Q \subseteq \mK_{k,d}$, we obtain a lower bound on the derivative of $G_{k,d}$: $|\p_{w_1}G_{k,d}(w)| \gtrsim 2^{-3k} \rho^b$. By van der Corput's lemma again, $\forall (w_2, w_3) \in [-1,1]^2$, 
\begin{align*}
    |\{w_1 \in \R; (w_1, w_2, w_3) \in Q; |G_{k,d}(w)| \leq \ep\}| \lesssim 2^{3k} \ep \rho^{-b}
\end{align*}
with implicit constant independent of $k,d,w_2,w_3$. By Fubini, integrating over the two missing variables $w_2, w_3$, we have
\begin{align*}
    |\{w \in Q; |G_{k,d}(w)| \leq \ep\}| \lesssim 2^{3k} \ep \rho^{2-b}.
\end{align*}
Since there are $O(\rho^{-3})$ cubes, 
\begin{align*}
    \bigcup_{Q \in \mQ_{\textrm{far}}} |\{w \in Q; G_{k,d}(w)| \leq \ep\}| \lesssim 2^{3k} \ep \rho^{-1-b}.
\end{align*}

We apply the trivial estimate to the third subcollection $\mQ_{\textrm{bdry}}$ since it only contains $O(\rho^{-2})$ cubes. That is,
\begin{align*}
    |\bigcup_{Q \in \mQ_{\textrm{bdry}}}Q|\lesssim \rho^{3} \rho^{-2} = \rho.
\end{align*}

Combining all three sublevel sets estimates for $G_{k,d}$, and letting $\rho = \ep^{c'}$, for some $c'>0$ such that
\begin{align*}
    |\{w \in \mK; |G_{k,d}(w)|\leq \ep\}| &\lesssim \rho^a + 2^{3k} \ep \rho^{-1-b} + \rho\\
    & \lesssim 2^{Ck} \ep^c.
\end{align*}
    
\end{proof}

\newpage

\section{Applications}

\subsection{Application to Bilinear Maximal Type Operator}
~\\
In this section, we introduce a close sibling of the variant of the triangular Hilbert transform:
\begin{equation}
    \mathscr{M}\br{f_1,f_2}\br{x,y}
    :=
    \sup_{\epsilon>0}
        \fint_{-\epsilon}^\epsilon
            \verts{f_1}\br{x+t,y}
            \verts{f_2}\br{x,y+t^2}
        dt
\end{equation}
We aim to prove the following:
\begin{theorem}[Theorem 3 \cite{CHRIST2021107863}]\label{thm_bilinear_maximal}
    For every \(p,q\in\br{1,\infty}\) and \(r\in\left[1,\infty\right)\) with \(\frac{1}{r}=\frac{1}{p}+\frac{1}{q}\) we have:
    \begin{equation}
        \Verts{\mathscr{M}\br{f_1,f_2}}_{L^r}
        \lesssim
        \Verts{f_1}_{L^p}
        \Verts{f_2}_{L^q}.
    \end{equation}
\end{theorem}
Indeed, as is mentioned in \cite{CHRIST2021107863}, the result for \(r>1\) follows from H\"{o}lder's inequality and the bounds on Hardy-Littlewood maximal function. The interesting part of the statement is the strong type endpoint bound at \(r=1\).
\begin{proof}
    W.L.O.G. we only need to consider non-negative Schwartz functions \(f_1,f_2\). Moreover, instead of working with \(\mathscr{M}\) directly, we shall consider the following auxiliary operator:
    \begin{equation}
        M\br{f_1,f_2}\br{x,y}:=
        \sup_{j\in\Z}
            \int
                f_1\br{x+t,y}
                f_2\br{x,y+t^2}
            2^j\psi\br{2^j t}dt.
    \end{equation}
    We claim that:
    \begin{equation}
     \mathscr{M}\br{f_1,f_2}\br{x,y}\lesssim
     M\br{f_1,f_2}\br{x,y}.
    \end{equation}
    Indeed, for fixed \(\epsilon>0\), we may find unique \(j_0\in\Z\) satisfying \(\epsilon\in\left[2^{-j_0-1},2^{-j_0}\right)\). As a direct consequence, we have:
    \[
    \1_{\br{-\epsilon,\epsilon}\setminus\BR{0}}\br{t}
    \leq
    \varphi\br{2^{j_0}t}
    =
    \sum_{j\geq j_0} 
        \psi\br{2^j t}.
    \]
    With a bit more algebraic manipulation, we derive an estimate for the kernel:
    \[
    \frac{1}{2\epsilon}\1_{\br{-\epsilon,\epsilon}\setminus\BR{0}}\br{t}
    \lesssim \sum_{j\geq j_0} 2^{j_0-j}\cdot 2^j\psi\br{2^j t}
    \]
    and thus,
    \[
    \begin{aligned}
        &
        \fint_{-\epsilon}^\epsilon
            f_1\br{x+t,y}
            f_2\br{x,y+t^2}
        dt\\
        \lesssim &
        \sum_{j\geq j_0}
            2^{j_0-j}
            \int
                f_1\br{x+t,y}
                f_2\br{x,y+t^2}
            2^j\psi\br{2^j t}dt\\
        \leq &
        \sum_{j\geq j_0}
            2^{j_0-j}
            M\br{f_1,f_2}\br{x,y}
        =2M\br{f_1,f_2}\br{x,y}.
    \end{aligned}
    \]
    Taking \(\sup_{\epsilon>0}\) proves the claim. It remains to obtain the analogous bound for \(M\br{f_1,f_2}\) given in \ref{thm_bilinear_maximal}. To draw the comparison between the argument to bound \(T\br{f_1,f_2}\) and \(M\br{f_1,f_2}\), we introduce the following notation:
    \begin{equation}
        M_j\br{f_1,f_2}\br{x,y}
        :=
        \int
            f_1\br{x+t,y}
            f_2\br{x,y+t^2}
        2^j\psi\br{2^j t}dt.
    \end{equation}
    Indeed, using this notation, we see that:
    \begin{equation}
    \left\{
        \begin{aligned}
            M\br{f_1,f_2}\br{x,y}= & 
            \underbrace{\sup_{j\in\Z}}
                M_j\br{f_1,f_2}\br{x,y}\\
                    &\hspace{1.75ex}\text{v.s.}\\
            T\br{f_1,f_2}\br{x,y}= &
            \overbrace{\sum_{j\in\Z}}
                T_j\br{f_1,f_2}\br{x,y},
        \end{aligned}
    \right.
    \end{equation}
    where \(T_j\) and \(M_j\) are almost identical with their kernels being the main differences:
    \[
        \overset{\textbf{Mean zero}}{\frac{\psi\br{2^j t}}{t}}
        \hspace{1.5ex}\text{v.s.}\hspace{1.5ex}
        \overset{\textbf{Positive}}{2^j\psi\br{2^j t}}.
    \]
    This suggests that we shall perform the analogous decomposition on \(M_j\):
    \begin{equation}\label{eq 2.1 M LMH}
        M_j = M_j^L + M_j^M +M_j^H,
    \end{equation}
    where each component is defined below:
    \begin{equation*}
        \left\{
        \begin{aligned}
            M_j^L = &     
                \sum_{\substack{k\in \Z^2\\ 
                :k_1 \vee k_2 \leq 0}} M_j
                \br{
                    \D{1}{j+k_1} \otimes \D{2}{2j+k_2} 
                }\\
            M_j^M = & 
                \sum_{\substack{
                k\in\Z^2\\
                :k_1 \vee k_2 >0 \\ 
                |k_1 - k_2| \geq 100}} M_j
                \br{
                    \D{1}{j+k_1} \otimes \D{2}{2j+k_2}
                }\\
            M_j^H = &
                \sum_{\substack{
                k\in\Z^2\\
                :k_1 \vee k_2 >0 \\ 
                |k_1 - k_2| < 100}} M_j
                \br{
                    \D{1}{j+k_1} \otimes \D{2}{2j+k_2}
                }.
        \end{aligned}
        \right.
    \end{equation*}
    With the notations introduced, we formulate the proof strategy:
    \begin{itemize}
        \item Let \(\omega=L,M\), we aim to prove the following pointwise bound on the unit scale:
        \begin{equation}\label{eq_bi_M_pointwise_bound}
            \verts{M^\omega_0\br{f_1,f_2}\br{x,y}}
            \lesssim 
            \prod_{l=1,2}\M^{\br{l}}\br{f_l}\br{x,y}.
        \end{equation}
        Once we have the pointwise bound, the dilation symmetry:
        \begin{equation}
            M^\omega_j\br{f_1,f_2}=
            \Dil^r_{2^{-j},2^{-2j}}
            M^\omega_0
            \br{
                \Dil^p_{2^j,2^{2j}}f_1,
                \Dil^q_{2^j,2^{2j}}f_2
            }
        \end{equation}
        extends the pointwise bound to all scales and yields a pointwise estimate:
        \begin{equation}
            \sup_{j\in\Z}
                \verts{M^\omega_j\br{f_1,f_2}\br{x,y}}
            \lesssim 
            \prod_{l=1,2}\M^{\br{l}}\br{f_l}\br{x,y}
        \end{equation}
        that implies the desired H\"{o}lder type bound.
        \item For high-frequency components, we made the following key observations:
        \begin{equation}
            \sup_{j\in\Z}
                \verts{M^H_j\br{f_1,f_2}\br{x,y}}
            \leq
            \sum_{j\in\Z}
                \verts{M^H_j\br{f_1,f_2}\br{x,y}}.
        \end{equation}
        Additionally, since the argument throughout the high-frequency component \(T^H\) is absolutely summable, replacing \(T_j\br{\Delta_{k_1}^{\br{1}}f_1,\Delta_{k_2}^{\br{2}}f_2}\) with \(\verts{M_j\br{\Delta_{k_1}^{\br{1}}f_1,\Delta_{k_2}^{\br{2}}f_2}}\) and then going through the analogous argument gives the desired result.
    \end{itemize}
    It remains to prove \eqref{eq_bi_M_pointwise_bound}. Direct calculation gives:
    \begin{equation}
        M^\omega_0\br{f_1,f_2}\br{x,y}
        =
        \int
            f_1\br{x-u,y}
            f_2\br{x,y-v}
            k^\omega\br{u,v}
        dudv,
    \end{equation}
    where the kernel \(k^\omega\) can be defined in the following way:
    \begin{equation}
        \widehat{k^\omega}\br{\xi,\eta}:=
        \sum_{k:\omega}
            \psi_{k_1}\br{\xi}\psi_{k_2}\br{\eta}
            m_+\br{\xi,\eta}.
    \end{equation}
    Here, we set:
    \begin{equation}
        m_+\br{\xi,\eta}:=\int e^{2\pi i\br{\xi t+\eta t^2}}\psi\br{t}dt.
    \end{equation}
    Comparing to \eqref{eq_multiplier_1st_appearence},
    \eqref{eq_multiplier_2nd_appearence}, \(m_+\br{\xi,\eta}\) behaves almost identically to \(m\br{\xi,\eta}\), the only difference is \(m_+\br{\xi,\eta}\)'s lack of cancellation for low-frequency \(\verts{\xi}\vee\verts{\eta}\lesssim 1\). Fortunately, this is not an issue for our purpose. In fact, for \(\omega=L\), once we recognize that
    \begin{equation}
        \widehat{k^L}\br{\xi,\eta}=
        \varphi\br{\xi}\varphi\br{\eta}m_+\br{\xi,\eta},
    \end{equation}
    all we need is the trivial estimate \eqref{eq_partial_m_trivial_est} and the compactness of \(\supp \varphi\):
    \begin{equation}
        \Verts{
            \partial_\xi^\alpha
            \partial_\eta^\beta
            \widehat{k^L}
        }_{L^1}
        =
        \Verts{
            \partial_\xi^\alpha
            \partial_\eta^\beta
            \br{
                \varphi\otimes\varphi\cdot
                m_+
            }
        }_{L^1}
        \underset{\alpha,\beta}{\lesssim} 1.
    \end{equation}
    This implies an arbitrary fast decay on the spatial side:
    \begin{equation}
        \verts{k^L\br{u,v}}\underset{N}{\lesssim}\ang{u}^{-N}\ang{v}^{-N}
    \end{equation}
    and the tensor structure on the right-hand side implies:
    \begin{equation}
        \verts{M^L_0\br{f_1,f_2}\br{x,y}}
        \lesssim 
        \prod_{l=1,2}\M^{\br{l}}\br{f_l}\br{x,y}.
    \end{equation}
    For \(\omega=M\), we first recognize that
    \begin{equation}
        \begin{aligned}
            \widehat{k^M}\br{\xi,\eta}
            =&
            \sum_{k_1\geq 0} 
                \psi_{k_1}\br{\xi}
                \varphi_{k_1-100}\br{\eta}
                m_+\br{\xi,\eta}\\
            + &
            \sum_{k_2\geq 0} 
                \varphi_{k_2-100}\br{\xi}
                \psi_{k_2}\br{\eta}
                m_+\br{\xi,\eta}.
        \end{aligned}
    \end{equation}
    Using the non-stationary phase principle in the form of \eqref{eq non-stationary}, we have:
    \begin{equation}
        \begin{aligned}
            \Verts{
                \partial_\xi^\alpha
                \partial_\eta^\beta
                \widehat{k^M}
            }_{L^1}
            \lesssim &
            \sum_{k_1\geq 0}
                \Verts{
                    \partial_\xi^\alpha
                    \partial_\eta^\beta
                    \br{
                        \psi_{k_1}\otimes
                        \varphi_{k_1-100}\cdot
                        m_+
                    }
                }_{L^1}
            +
            \sum_{k_2\geq 0}
                \Verts{
                    \partial_\xi^\alpha
                    \partial_\eta^\beta
                    \br{
                        \varphi_{k_2-100}\otimes
                        \psi_{k_2}\cdot
                        m_+
                    }
                }_{L^1}\\
            \underset{\alpha,\beta,N}{\lesssim} &
            \sum_{k_1\geq 0}
                2^{-Nk_1}
                \verts{
                    \supp
                        \psi_{k_1}\otimes
                        \varphi_{k_1-100}
                }
                +
            \sum_{k_2\geq 0}
                2^{-Nk_2}
                \verts{
                    \supp
                        \varphi_{k_2-100}\otimes
                        \psi_{k_2}
                }\\
            \sim &
            \sum_{k_1\geq 0}
                2^{\br{2-N}k_1}
            +
            \sum_{k_2\geq 0}
                2^{\br{2-N}k_2}
            \lesssim 1,\hspace{1.5ex}\textbf{whenever}\hspace{1.5ex}
            N \gg 1.
        \end{aligned}
    \end{equation}
    Therefore, for the same reason, we have fast decay on the spatial side:
    \begin{equation}
        \verts{k^M\br{u,v}}\underset{N}{\lesssim}\ang{u}^{-N}\ang{v}^{-N}
    \end{equation}
    and also the desired estimate:
    \begin{equation}
        \verts{M^M_0\br{f_1,f_2}\br{x,y}}
        \lesssim 
        \prod_{l=1,2}\M^{\br{l}}\br{f_l}\br{x,y}.
    \end{equation}
\end{proof}

\subsection{Application to Roth Type Theorem}
\begin{theorem}\label{roth}
    Let $\varepsilon \in (0,1)$ and $S\subseteq [0,1]^{2}$ a measurable set of Lebesgue measure at least $\varepsilon$. Then there exist
\begin{equation}
    (x,y),(x+t,y),(x,y+t^{2})\in S
\end{equation}
with $t>\operatorname{exp}(-\operatorname{exp}(\varepsilon^{-C}))$ for some constant $C>0$ not depending on $S$ or $\varepsilon$.
\end{theorem}

\begin{lemma}\label{noncancel}
    Let $\theta \geq 0$ be an even smooth functino which is supported in $[-2,2]$, constant on $[-1,1]$, monotone on $[1,2]$ and normalized such that $\int \theta =1$. Let $\theta_{k}(x):=D^{1}_{2^{-k}}\theta$. For any bounded non-negative function $f$ on $\mathbb{R}^{2}$ that is supported in $[0,1]^{2}$ and any $k,l \in \mathbb{N}$, we have
\begin{equation}
    \int_{[0,1]^{2}}f(f\ast_{1}\theta_{k})(f\ast_{2}\theta_{l})\geq c_{0}\left( \int_{[0,1]^{2}} f\right)^{3}
\end{equation}
for some constant $c_{0}>0$ depending only on $\theta$.
\end{lemma}

\begin{proof}
    We may treat \(f\ast_1\theta_k\) and \(f\ast_2\theta_l\) as if taking certain conditional expectations. This can be made precise by comparing the two with the following dyadic variant:
    \begin{equation}
        \left\{
        \begin{aligned}
            E_k^{\br{1}}f\br{x,y}
            := &
            \sum_{I\in\mathcal{D}_k}
                \1_{I}\br{x}\fint_{I}f\br{x',y}dx'\\
            E_l^{\br{2}}f\br{x,y}
            := &
            \sum_{J\in\mathcal{D}_l}
                \1_{J}\br{y}\fint_{J}f\br{x,y'}dy',
        \end{aligned}
        \right.
    \end{equation}
    where \(\mathcal{D}_k:=\BR{2^{-k}a+\left[0,2^{-k}\right)}_{a=0}^{2^k-1}\) denotes the collection of dyadic intervals of length \(2^{-k}\) contained in the unit interval \(\Br{0,1}\). Indeed, for any \(x\in I\in\mathcal{D}_k\), since clearly
    \begin{equation}
        I\subset x+\Br{-2^{-k},2^k},
    \end{equation}
    we must have:
    \begin{equation}
        \fint_I f\br{x',y}dx' \underset{\theta}{\lesssim} \int f\br{x-t,y}\theta_k\br{t}dt.
    \end{equation}
    This implies the pointwise estimate \(E^{\br{1}}_k f \underset{\theta}{\lesssim} f\ast_1\theta_k\). By symmetry, we also derive the analogous estimate \(E^{\br{2}}_l f \underset{\theta}{\lesssim} f\ast_2\theta_l\). This reduces the matter to proving the following:
    \begin{equation}
        \int_{\Br{0,1}^2} 
            f 
            \br{E^{\br{1}}_k f}
            \br{E^{\br{2}}_l f}
        \geq 
        \br{
            \int_{\Br{0,1}^2} f
        }^3.
    \end{equation}
    For technical reasons, we make an additional assumption that \(f\geq \epsilon>0\) on \(\Br{0,1}^2\). (We'll later remove such an assumption.) We proceed from the right-hand-side:
    \begin{equation}
        \int_{\Br{0,1}^2} f 
        =
        \sum_{\substack{
            I\in\mathcal{D}_k\\
            J\in\mathcal{D}_l
        }}
            \verts{I\times J}\cdot
            \fint_{I\times J} 
                f.
    \end{equation}
    Since \(1=\sum_{\substack{I\in\mathcal{D}_k\\J\in\mathcal{D}_l}}\verts{I\times J}\), the integral can be interpreted as a weighted average of \(\fint_{I\times j}f\). This turns out to be a perfect setting to apply Jensen's inequality:
    \begin{equation}
        \br{
            \int_{\Br{0,1}^2} f 
        }^3
        \leq 
        \sum_{\substack{
            I\in\mathcal{D}_k\\
            J\in\mathcal{D}_l
        }}
            \verts{I\times J}\cdot
            \br{
            \fint_{I\times J} 
                f
            }^3.
    \end{equation}
    Now, our goal is clear--extract a term of the following form:
    \begin{equation}
        \left.
            f
            \br{E^{\br{1}}_k f}
            \br{E^{\br{2}}_l f}
        \right\vert_{I\times J}
        =
        \left.
            f
            \br{\fint_I f\br{x',\cdot}dx'}
            \br{\fint_J f\br{\cdot,y'}dy'}
        \right\vert_{I\times J}
    \end{equation}
    out of the integral expression \(\fint_{I\times J} f\). This suggests that we split \(f\) into:
    \begin{equation}
        \begin{aligned}
            f= &
            \Br{
                f
                \br{\fint_I f\br{x',\cdot}dx'}
                \br{\fint_J f\br{\cdot,y'}dy'}
            }^{1/3}\\
            \times &
            \Br{
                f
                \br{\fint_I f\br{x',\cdot}dx'}^{-1}
            }^{1/3}\\
            \times &
            \Br{
                f
                \br{\fint_J f\br{\cdot,y'}dy'}^{-1}
            }^{1/3}.
        \end{aligned}
    \end{equation}
    Note that the second and the third term make sense owing to our initial assumption \(f\geq \epsilon>0\). We make use of the above decomposition and perform a \(L^3\times L^3\times L^3\) H\"{o}lder's inequality on \(\fint_{I\times J} f\) to obtain the following:
    \begin{equation}
        \begin{aligned}
            \fint_{I\times J} f
            \leq &
            \br{
                \fint_{I\times J}
                    f
                    \br{E^{\br{1}}_k f}
                    \br{E^{\br{2}}_l f}
            }^{1/3}\\
            \times &
            \cancelto{1}{
            \br{
                \fint_{I\times J}
                    f
                    \br{\fint_I f\br{x',\cdot}dx'}^{-1}
            }^{1/3}
            }\\
            \times &
            \cancelto{1}{
            \br{
                \fint_{I\times J}
                    f
                    \br{\fint_J f\br{\cdot,y'}dy'}^{-1}
            }^{1/3}
            }.
        \end{aligned}
    \end{equation}
    Combining everything, we have:
    \begin{equation}
        \begin{aligned}
            \br{
                \int_{\Br{0,1}^2} f 
            }^3
            \leq &
            \sum_{\substack{
                I\in\mathcal{D}_k\\
                J\in\mathcal{D}_l
            }}
                \verts{I\times J}\cdot
                \br{
                \fint_{I\times J} 
                    f
                }^3\\
            \leq &
            \sum_{\substack{
                I\in\mathcal{D}_k\\
                J\in\mathcal{D}_l
            }}
                \verts{I\times J}\cdot
                \fint_{I\times J}
                    f
                    \br{E^{\br{1}}_k f}
                    \br{E^{\br{2}}_l f}\\
            = &
            \int_{\Br{0,1}^2}
                f
                \br{E^{\br{1}}_k f}
                \br{E^{\br{2}}_l f}.
        \end{aligned}
    \end{equation}
    This finishes the proof given that we have \(f\geq \epsilon >0\). In general, we consider \(f_\epsilon:=\1_{\Br{0,1}^2}\cdot\br{f\vee \epsilon}\), and pass the limit \(\epsilon \searrow 0\) via dominated convergence theorem.
\end{proof}

Now we prove the main Roth-type theorem.
\begin{proof}[proof of theorem \ref{roth}]
Write
\begin{equation}
     I=\int_{[0,1]^{3}}f(x,y)f(x+t,y)f(x,y+t^{2})dxdydt   
\end{equation}
We will show that for every measurable function $f$ on $[0,1]^{2}$ with $0\leq f\leq 1$(Hence, $\|f\|_{L^{\infty}}\leq 1$ and $\|f\|_{L^{2}}\leq 1$.) and $\int_{[0,1]^{2}}f\geq \varepsilon$, we have $I>\delta (\varepsilon)$. Then we take $f=1_{S}$, we have
\begin{equation}
    \begin{aligned}
I&=\int_{[0,1]^{3}}1_{S}(x,y)1_{S}(x+t,y)1_{S}(x,y+t^{2})dxdydt\\
&=\int_{[0,1]^{2}}1_{S}(x,y)\int_{0}^{\delta}1_{S}(x+t,y)1_{S}(x,y+t^{2})dtdxdy\\
&+\int_{[0,1]^{2}}1_{S}(x,y)\int_{\delta}^{1}1_{S}(x+t,y)1_{S}(x,y+t^{2})dtdxdy > \delta
    \end{aligned}
\end{equation}
Since trivially, we have
\begin{equation}
    \int_{[0,1]^{2}}1_{S}(x,y)\int_{0}^{\delta}1_{S}(x+t,y)1_{S}(x,y+t^{2})dtdxdy\leq \delta
\end{equation}
Then we have
\begin{equation}
    \int_{[0,1]^{2}}1_{S}(x,y)\int_{\delta}^{1}1_{S}(x+t,y)1_{S}(x,y+t^{2})dtdxdy >0
\end{equation}
That is there exist $(x,y),(x+t,y),(x,y+t^{2})\in S$ with $t>\delta (\varepsilon)$. We may also observe that once we have better lower bound of $I$, then we have better gap estimate.
Roughly speaking, we first decompose $1_{[0,1]}$ into different scales. Let $\tau$ be a smooth bump, supported on $[\frac{1}{2},2]$, taking value in $[0,1]$ and $\int \tau =1$. Let $\tau_{k}=D^{1}_{2^{-k}}\tau$. Then we have
\begin{equation}
    I\geq \sum_{k\in \mathbb{N}}2^{-k}\int_{[0,1]^{3}}f(x,y)f(x+t,y)f(x,y+t^{2})\tau_{k}(t)dxdydt
\end{equation}
At the end, we will not use the full sum, we will just pick one specific $k$ which is large enough to give us lower bound. Fix a scale $k$ of the bump, we will decompose $f$ in frequency side in scale and split into three terms, higher or lower or comparable to the scale of the bump. Let $k_{L}<k<k_{H}$. We have
\begin{equation}\label{threeterm}
    \begin{aligned}
        &\int_{[0,1]^{3}}f(x,y)f(x+t,y)f(x,y+t^{2})\tau_{k}(t)dxdydt\\
    =&\int_{[0,1]^{3}}f(x,y)f(x+t,y)(f-f\ast_{2}\theta_{k_{H}})(x,y+t^{2})\tau_{k}(t)dxdydt\\
    +&\int_{[0,1]^{3}}f(x,y)f(x+t,y)(f\ast_{2}\theta_{k_{H}}-f\ast_{2}\theta_{k_{L}})(x,y+t^{2})\tau_{k}(t)dxdydt\\
    +&\int_{[0,1]^{3}}f(x,y)f(x+t,y)(f\ast_{2}\theta_{k_{L}})(x,y+t^{2})\tau_{k}(t)dxdydt\\
    =&I_{1}+I_{2}+I_{3}
    \end{aligned}
\end{equation}

The first term $I_{1}$ is high frequency compare to the bump and the cancellation may be seen. We therefore want to apply the smoothing inequality \eqref{thm 5 - smoothing}. One way to argue is to split the high frequency part into dyadic annulus then apply smoothing inequality. Here we directly use the equivalent form of smoothing inequality. The first term is then dominated by
\begin{equation}\label{smoothappl}
    \begin{aligned}
        \|f\|_{L^{\infty}}\cdot \|T_{k}(f,f-f\ast_{2}\theta_{k_{H}})\|_{L^{1}}=&\|D^{1}_{2^{-k},2^{-2k}}T_{0}(D^{2}_{2^{k},2^{2k}}f,D^{2}_{2^{k},2^{2k}}(f-f\ast_{2}\theta_{k_{H}}))\|_{L^{1}}\\
        =&\|T_{0}(D^{2}_{2^{k},2^{2k}}f,D^{2}_{2^{k},2^{2k}}(f-f\ast_{2}\theta_{k_{H}}))\|_{L^{1}}\\
        \lesssim &\|D^{2}_{2^{k},2^{2k}}f\|_{H^{(-\sigma ,0)}}\cdot \|D^{2}_{2^{k},2^{2k}}(f-f\ast_{2}\theta_{k_{H}})\|_{H^{(0,-\sigma )}}\\
        \lesssim &\|f\|_{L^{2}}\cdot \|(1+|\eta|^{2})^{-\frac{\sigma}{2}}D^{2}_{2^{-k},2^{-2k}}\left(\widehat{f}(\xi ,\eta)(\widehat{\theta}_{k_{H}}(0)-\widehat{\theta}_{k_{H}}(\eta))\right)\|_{L^{2}}\\
        \lesssim  &\|(1+|\eta|^{2})^{-\frac{\sigma}{2}}\left(2^{-k_{H}+2k}|\eta| \wedge 1  \right)\|_{L^{\infty}}\\
        \sim &\|\left( 1\wedge |\eta|^{-\sigma}\right) \left(2^{-k_{H}+2k}|\eta| \wedge 1  \right)\|_{L^{\infty}}
    \end{aligned}
\end{equation}
where the last inequality is by mean value theorem
\begin{equation}
    \widehat{\theta}_{k_{H}}(0)-\widehat{\theta}_{k_{H}}(\eta)\leq 1\wedge \|\widehat{\theta}_{k_{H}}\|_{L^{\infty}}\cdot |\eta|
\end{equation}
The trivial bound 1 is by triangle inequality and $\|\widehat{\theta}\|_{L^{\infty}}\lesssim 1$.
Let $0<\delta <1$ be a fix number. We have two cases
\begin{equation}
    |\eta|\lesssim  2^{\delta(k_{H}-2k)} \Longrightarrow \eqref{smoothappl}\lesssim 1\cdot 2^{-(1-\delta)(k_{H}-2k)}
\end{equation}
\begin{equation}
    |\eta|\gtrsim \ 2^{\delta(k_{H}-2k)} \Longrightarrow \eqref{smoothappl}\lesssim 1\cdot 2^{-\sigma \delta(k_{H}-2k)}
\end{equation}
That is, there exist a $\tilde{\sigma}>0$ such that
\begin{equation}
    \eqref{smoothappl}\lesssim 1\cdot 2^{\tilde{\sigma}(2k-k_{H})}
\end{equation}
As for the second term $I_{2}$, we just use $(\infty ,2,2)$ H\"older inequality and may bound it by
\begin{equation}
    \|f\ast_{2}\theta_{k_{H}}-f\ast_{2}\theta_{k_{L}}\|_{L^{2}}
\end{equation}
We further split the third term $I_{3}$ into three terms.
\begin{equation}
    \begin{aligned}
        I_{3}&=I_{3}-\int_{[0,1]^{2}}f(f\ast_{1}\tau_{k})(f\ast_{2}\tau_{k_{L}})\\
        &+\int_{[0,1]^{2}}f(f\ast_{1}\tau_{k})(f\ast_{2}\tau_{k_{L}})-\int_{[0,1]^{2}}f(f\ast_{1}\theta_{k_{L}})(f\ast_{2}\tau_{k_{L}})\\
        &+\int_{[0,1]^{2}}f(f\ast_{1}\theta_{k_{L}})(f\ast_{2}\tau_{k_{L}})\\
        &=I_{3-1}+I_{3-2}+I_{3-3}
    \end{aligned}
\end{equation}
We estimate $I_{3-1}$ in the following
\begin{equation}
    \begin{aligned}
|I_{3,1}|=&|\int_{[0,1]^{4}}f(x,y)f(x+t,y)f(x,y+s)\theta_{k_{L}}(s+t^{2})\tau_{k}(t)dxdydsdt \\
-&\int_{[0,1]^{4}}f(x,y)f(x+t,y)f(x,y+s)\tau_{k}(t)\theta_{k_{L}}(s)dsdtdxdy |\\
\lesssim &\int_{[0,1]^{4}}f(x,y)f(x+t,y)f(x,y+s)\left(\theta_{k_{L}}(s+t^{2})-\theta_{k_{L}}(s)\right)\tau_{k}(t)dxdydsdt\\
\lesssim &\int_{[0,1]}2^{k_{L}}\cdot 2^{-2k}|t|^{2}2^{k}\tau(2^{k}t)dt\lesssim 2^{k_{L}-k}
    \end{aligned}
\end{equation}
where we use
\begin{equation}
   |\theta_{k_{L}}(s+t^{2})-\theta_{k_{L}}(s)|\lesssim \|\theta_{k_{L}}'\|_{L^{\infty}}\cdot |t|^{2}\lesssim 2^{k_{L}}\cdot 2^{-2k}
\end{equation}
As for $I_{3-2}$, we can estimate as follow
\begin{equation}
    \begin{aligned}
        |I_{3-2}|&\lesssim \|f\ast_{1}\tau_{k}-f\ast_{1}\theta_{k_{L}}\|_{L^{2}}\\
        &\lesssim |f\ast_{1}\tau_{k}-f\ast_{1}\tau_{k}\ast_{1}\theta_{k_{H}}\|_{L^{2}}+\|f\ast_{1}\tau_{k}\ast_{1}\theta_{k_{H}}-f\ast_{1}\tau_{k}\ast_{1}\theta_{k_{L}}\|_{L^{2}}+\|f\ast_{1}\tau_{k}\ast_{1}\theta_{k_{L}}-f\ast_{1}\theta_{k_{L}}\|_{L^{2}}\\
        &\lesssim \|\tau_{k}-\tau_{k}\ast \theta_{k_{H}}\|_{L^{1}}+\|f\ast_{1}\theta_{k_{H}}-f\ast_{1}\theta_{k_{L}}\|_{L^{2}}+\|\tau_{k}\ast \theta_{k_{L}}-\theta_{k_{L}}\|_{{L^{1}}}\\
        &\lesssim \|f\ast_{1}\theta_{k_{H}}-f\ast_{1}\theta_{k_{L}}\|_{L^{2}}+O(2^{k_{L}-k})+O(2^{k-k_{H}})
    \end{aligned}
\end{equation}
where we again use the mean value theorem. As for $I_{3-3}$, by \textit{Lemma} \ref{noncancel}, we have $|I_{3-3}|\geq c_{0}\varepsilon^{3}$.
All in all fix a $k$, we have
\begin{equation}
    I>2^{-k}\left(c_{0}\varepsilon^{3} - 2^{\tilde{\sigma}(2k-k_{H})}-2^{k_{L}-k}-2^{k-k_{H}}- \|f\ast_{1}\theta_{k_{H}}-f\ast_{1}\theta_{k_{L}}\|_{L^{2}}- \|f\ast_{2}\theta_{k_{H}}-f\ast_{2}\theta_{k_{L}}\|_{L^{2}}  \right)
\end{equation}
Now we put the condition $k_{1,L}= Mk_{1}= M^{2}k_{1,H}=M^{3}k_{2,L}=M^{4}k_{2}= M^{5}k_{2,H}= \cdots$. We can take $k_{1,L}\geq \operatorname{log}(\varepsilon^{-1})$ so large and a fix small $M$ independent of $\varepsilon$ so that
\begin{equation}
     2^{\tilde{\sigma}(2k-k_{H})}+2^{k_{L}-k}+2^{k-k_{H}}<10^{-10}c_{0}\varepsilon^{3}
\end{equation}
If for one of $k$, we have 
\begin{equation}
   \|f\ast_{1}\theta_{k_{H}}-f\ast_{1}\theta_{k_{L}}\|_{L^{2}}+\|f\ast_{2}\theta_{k_{H}}-f\ast_{2}\theta_{k_{L}}\|_{L^{2}}<\frac{c_{0}}{2}\varepsilon^{3}
\end{equation}
Suppose, we do it $L$ times and we still fail, this means that
\begin{equation}
    \begin{aligned}
        \frac{L}{2}c_{0}\varepsilon^{3}\leq & \sum_{i=1}^{L}\|f\ast_{1}\theta_{k_{i,H}}-f\ast_{1}\theta_{k_{iL}}\|_{L^{2}}+\|f\ast_{2}\theta_{k_{i,H}}-f\ast_{2}\theta_{k_{i,L}}\|_{L^{2}}\\
        \lesssim & L^{1/2}
        \br{
        \sum_{i=1}^L
            \|f\ast_{1}\theta_{k_{i,H}}-f\ast_{1}\theta_{k_{iL}}\|_{L^{2}}^2+\|f\ast_{2}\theta_{k_{i,H}}-f\ast_{2}\theta_{k_{i,L}}\|_{L^{2}}^2
        }^{1/2}
        \lesssim L^{1/2}
    \end{aligned}
\end{equation}
where Plancherel obtains the upper bound.
Then $L\lesssim\varepsilon^{-6}$, this means that we find some $i<L$ such that
\begin{equation}
   \|f\ast_{1}\theta_{k_{i,H}}-f\ast_{1}\theta_{k_{i,L}}\|_{L^{2}}+\|f\ast_{2}\theta_{k_{i,H}}-f\ast_{2}\theta_{k_{i,L}}\|_{L^{2}}<\frac{c_{0}}{2}\varepsilon^{3}
\end{equation}
This finish our proof.

\end{proof}

\newpage

\appendix
\section{Errata}

\noindent\textit{p.10} - To figure out the range of integration of $s$, for a fixed $n \in \Z$ and $l>0$, we have
\begin{align*}
    l2^{-\kp} \leq &t \leq (l+1) 2^{-\kp}\\
    l^2 2^{-2\kp} \leq & t^2 \leq l^22^{-2\kp} +l 2^{-2\kp+1} +2^{-2\kp}\\
    \implies 2^{-\kp} n - (l^22^{-2\kp} + l 2^{-2\kp+1} +2^{-2\kp}) \leq &s \leq 2^{-\kp} (n+1) - l^2 2^{-2\kp}
\end{align*}
As such, the range of integration of $s$ is given by 
\begin{align*}
    [2^{-\kp} n - l^22^{-2\kp} - l 2^{-2\kp+1} - 2^{-2\kp}, 2^{-\kp} (n+1) - l^2 2^{-2\kp}]
\end{align*}
which is contained in
\begin{align*}
    [2^{-\kp} n -l^22^{-2k +3} , 2^{-\kp}n - l^2 2^{-2\kp} +2^{-\kp } ] 
\end{align*}
which in turn is contained in\footnote{Note that $|l^2 2^{-2\kp} - l^2 2^{-2\kp +3}| \leq 2^{3}$. }
\begin{align*}
    J_l:= [2^{-\kp} n -l^22^{-2k +3} , 2^{-\kp}n - l^2 2^{-2\kp+3} +2^{-\kp +3} ] = [0, 2^{-\kp+3}] + \sg_{l, n} 2^{-\kp+3}
\end{align*}
where $\sg_{l,n} = 2^{-3}n - l^2 2^{-\kp}$. 

\newpage


\printbibliography[title={References}]

\vspace{.3in}

\Addresses

\end{document}